\documentclass{article}

\usepackage{amsmath,amsfonts,latexsym,graphicx,amssymb,amsthm}

\usepackage[tips,matrix,arrow,frame]{xy}

\newcommand{\C}{\mathbb{C}}
\renewcommand{\P}{\mathbb{P}}
\newcommand{\Q}{\mathbb{Q}}
\newcommand{\R}{\mathbb{R}}
\newcommand{\Z}{\mathbb{Z}}

\newcommand{\ad}{\mathrm{ad}\,}
\newcommand{\Ad}{\mathrm{Ad}\,}
\newcommand{\Aut}{\mathrm{Aut}\,}

\newcommand{\Card}{\mathrm{Card}\,}

\newcommand{\Der}{\mathrm{Der}\,}
\renewcommand{\d}{\mathrm{d}\,}

\newcommand{\End}{\mathrm{End}}
\newcommand{\Hom}{\mathrm{Hom}}

\renewcommand{\Im}{\mathrm{Im}\,}
\newcommand{\id}{\mathrm{id}}
\newcommand{\Inn}{\mathrm{Inn}\,}
\newcommand{\Ind}{\mathrm{Ind}\,}

\newcommand{\Ker}{\mathrm{Ker}\,}
\newcommand{\Max}{\mathrm{Max}\,}

\newcommand{\Supp}{\mathrm{Supp}\,}

\newcommand{\sd}{\ltimes}
\newcommand{\Spec}{\mathrm{Spec}\,}

\newcommand{\tr}{\mathrm{tr}\,}

\newcommand{\g}{\mathfrak{g}}

\newtheorem{thm}{Theorem}

\newtheorem{lemma}{Lemma}[section]

\begin{document}
\title{CLASSIFICATION OF SIMPLE LIE ALGEBRAS ON A LATTICE}
\author{Kenji Iohara and Olivier Mathieu\footnote{Reasearch supported by UMR 5028 du CNRS.}}
\maketitle
\begin{abstract}
Let $\Lambda=\Z^n$ for some $n\geq 1$. The aim of the paper is to classify all simple $\Lambda$-graded Lie algebras $\mathcal{L}=\bigoplus_{\lambda \in \Lambda} \mathcal{L}_\lambda$, such that $\dim \mathcal{L}_\lambda=1$ for all $\lambda$. The classification involves two affine Lie algebras, namely $A_1^{(1)}$ and $A_2^{(2)}$, and a familly $(W_\pi)$, parametrized by a dense open set of the space of all embeddings $\pi: \Lambda \rightarrow \C^2$. The family $(W_l)$ of generalized Witt algebras, indexed by all  embeddings $l:\Lambda \rightarrow \C$, appears as a sub-family.  In general, the algebras $W_\pi$ are described as Lie algebras of symbols of twisted pseudo-differential operators. 
\end{abstract}

\vskip3cm

\setcounter{section}{-1}

\section{Introduction}

\subsection{Statement of the Theorem proved in the paper}\label{sect_01}

Let $\Lambda$ be a lattice of rank $n$, i.e. $\Lambda\simeq \Z^n$.
Following I. M. Gelfand's terminology, a \textit{Lie algebra on the lattice $\Lambda$}  is a $\Lambda$-graded Lie algebra 
$\mathcal{L}=\oplus_{\lambda\in\Lambda}\,\mathcal{L}_{\lambda}$ such that
$\dim\,\mathcal{L}_\lambda=1$ for all $\lambda$. 

I.M. Gelfand (in his Seminar) and A.A. Kirillov \cite{Ki} raised the
question of the classification of all Lie algebras on a lattice.
Of course, $\mathcal{L}$ should satisfy
additional properties in order  to expect 
an answer. The present paper investigates the case of simple graded
Lie algebras on a lattice. Recall that the  Lie
algebra
$\mathcal{L}$ is called \textit{simple graded} if $\mathcal{L}$ has no non-trivial
proper graded ideal (here it is assumed that $n>0$, otherwise one has to
assume that
$\dim \mathcal{L}>1$).

To our best of our  knowledge, the first instance of this question is
the V.G. Kac paper \cite{Ka2}, where he gave an explicit conjecture  for
the classification of all simple $\Z$-graded Lie algebras
$\mathcal{L}=\oplus_{n\in \Z}\,\mathcal{L}_n$ for which
$\dim\mathcal{L}_n=1$ for all $n$. He conjectured that such
a Lie algebra is isomorphic to the loop algebras
$A^{(1)}_1$, $A^{(2)}_2$ or to the Witt algebra $W$, see Section \ref{sect_02} for
the definition of these algebras. His conjecture is now proved \cite{M1}.

However, for lattices of rank $>1$, there was no explicit conjecture
(however see   Yu's Theorem \cite{Y} cited below) .

The main result of the paper is the classification of all
simple graded  Lie algebras on a lattice. To  clarify the
statement, the notion of primitivity is defined. Let
$\mathcal{L}$ be a simple
 $\Lambda$-graded Lie algebra, and let $m>0$.
The Lie algebra 
$\mathcal{L}(m)=\mathcal{L}\otimes \C[z_1^{\pm 1},\dots,z_1^{\pm m}]$ is a 
simple $\Lambda\times \Z^m$-graded Lie algebra, and
it is called an \textit{imprimitive form} of $\mathcal{L}$. 
A simple $\Lambda$-graded  Lie algebra $\mathcal{L}$ is called
\textit{primitive} if it is not an imprimitive form. It is clear
that any  simple $\Z^n$-graded Lie algebra is isomorphic
to $\mathcal{L}(m)$ for some primitive $\Z^{n-m}$-graded Lie algebra
 $\mathcal{L}$.

In this paper, a family of Lie algebras $W_\pi$ is
introduced. This family is parametrized by
an injective additive map $\pi:\Lambda\rightarrow \C^2$,
and it contains all generalized Witt algebras.
The result proved in this paper is the following one: \\

\noindent{\textbf{Main Theorem:}} \textit{Let 
$\mathcal{L}$ be a primitive  Lie algebra on $\Lambda$.
Then $\mathcal{L}$ is isomorphic to $A^{(1)}_1$, $A^{(2)}_2$ or to some
$W_\pi$, where $\pi:\Lambda\rightarrow \C^2$ is
an injective and additive map satisfying condition $\mathcal{C}$.} \hfill $\Box$
\\

The next section of the introduction will be devoted to
the precise definitions of the Lie algebras $A^{(1)}_1$, $A^{(2)}_2$ and
$W_\pi$ involved in the theorem (as well as the condition
$\mathcal{C}$). 

Since the proof of the theorem is quite long, the  paper is 
divided into three chapters (see Section \ref{sect_04}). Each chapter is 
briefly described in Sections \ref{sect_05}-\ref{sect_07}. References for 
the definitions of the Lie algebras are given in Section \ref{sect_03},
and for the tools used in the proof in Section \ref{sect_08}.

\subsection{Definition of the the Lie algebras involved in the
classification}\label{sect_02}

In what follows, the following convention will be used.
The identity $\deg x=\lambda$ tacitly means
that $x$ is a homogeneous element, and its degree is $\lambda$. 

\smallskip
 {\it The Lie algebra $A_1^{(1)}$:}
By definition, it is the Lie algebra 
$\mathfrak{sl}(2)\otimes \C[T,T^{-1}]$, where
the $\Z$-gradation is defined  
by the following requirements:

\centerline {$\deg e\otimes T^n=3n+1$, $\deg h\otimes T^n=3n$ and 
$\deg f\otimes T^n=3n-1$.}

\noindent Here  $\{e,\,f,\,\,h\}$
is the standard basis of $\mathfrak{sl}(2)$. 

\smallskip
{\it  The Lie algebra $A_2^{(2)}$:}
For  $x\in\mathfrak{sl}(3)$, set $\eta(x)=-x^t$,
where $t$ denotes the transposition.
Define an  involution $\theta$ 
of $\mathfrak{sl}(3)\otimes \C[T,T^{-1}]$ by

\centerline{$\theta(x\otimes T^n)=(-1)^{n}\, \eta(x)\otimes T^n$,}

\noindent for any $x\in\mathfrak{sl}(3)$ and $n\in \Z$.
By definition, $A_2^{(2)}$ is the Lie algebra of fixed points of the
involution $\theta$.
  
The $\Z$-gradation of $A_2^{(2)}$ is more delicate to define. Let
$(e_i,\,f_i\,\,h_i)_{i=1,\,2}$  be Chevalley's generators of $\mathfrak{
sl}(3)$.  Relative to these
generators, we have $\eta(x_1)=x_2$ and $\eta(x_2)=x_1$, where the letter
$x$ stands for $e$, $f$ or $h$.

Then the gradation is
defined by the following requirements:

$\deg (f_1+f_2)\otimes T^{2n}=8n-1$\hskip1.3cm $\deg (h_1+h_2)\otimes
T^{2n}=8n$, 

$\deg (e_1+e_2)\otimes T^{2n}=8n+1$ \hskip1.2cm 
$\deg [f_1,f_2]\otimes T^{2n+1}=8n+2$,

$\deg (f_1-f_2)\otimes T^{2n+1}=8n+3$
\hskip0.8cm $\deg (h_1-h_2)\otimes T^{2n+1}=8n+4$, 

$\deg (e_1-e_2)\otimes T^{2n+1}=8n+5$
\hskip0.8cm $\deg [e_1,e_2]\otimes T^{2n+1}=8n+6$.

\smallskip
\textit{The generalized Witt algebras $W_l$:}
Let $A=\C [z,z^{-1}]$ be the Laurent polynomial ring.
In what follows, its spectrum $\Spec A=\C^\ast$ is called
the {\it circle}. The {\it Witt algebra} is the Lie algebra
$W=\Der A$ of vector fields on the circle.  
It has basis $(L_n)_{n\in \Z}$, where
$L_n=z^{n+1}\dfrac{d}{dz}$,
and the Lie bracket is given by:

\centerline{$[L_n,L_m]=(m-n)\,L_{n+m}$.}

Let $\mathcal{A}$ be the {\it twisted Laurent polynomial ring}.
By definition $\mathcal{A}$ has basis $(z^s)_{s\in \C}$ and
the product is given by $z^s.z^t=z^{s+t}$. The operator
$\dfrac{d}{dz}$ extends to a
derivation of $\partial:\mathcal{A}\rightarrow \mathcal{A}$ defined
by $\partial z^s=sz^{s-1}$, for all $s\in \C$.
The Lie algebra $\mathcal{W}=\mathcal{A}.\partial$ will be
called the \textit{twisted Witt algebra}. It has basis $(L_s)_{s\in \C}$,
$L_s=z^{s+1}\dfrac{d}{dz}$
and the Lie bracket is given by:

\centerline{$[L_s,L_t]=(t-s)\,L_{s+t}$.}

For any injective additive map
$l:\Lambda\rightarrow \C$, denote by $W_l$ the subalgebra of
$\mathcal{W}$ with basis $(L_{s})$, where $s$ runs over 
the subgroup $l(\Lambda)$. Using the notation $L_\lambda$ for
$L_{l(\lambda)}$, the algebra $W_l$ has basis
$(L_\lambda)_{\lambda\in\Lambda}$ and the bracket is
given by:

\centerline{$[L_\lambda,L_\mu]=l(\mu-\lambda)\, L_{\lambda+\mu}$.}

For the natural gradation of $W_l$, relative to which each $L_\lambda$ is
homogeneous of degree $\lambda$,
$W_l$ is Lie algebra on the lattice $\Lambda$. Moreover
$W_l$ is simple (\cite{Y}, Theorem 3.7).

\smallskip
\textit{The algebra $W_\pi$:}
Recall that the ordinary pseudo-differential operators on the circle
are formal series $\sum_{n\in \Z}\,a_n\partial^n$, where
$a_n\in \C[z,z^{-1}]$ and $a_n=0$ for $n>>0$, and where
$\partial =\dfrac{d}{dz}$. The
definition of \textit{twisted pseudo-differential operators} is similar,
except that complex powers of $z$ and of $\partial$ are allowed
(for a rigorous definition, see Section \ref{sect_12}). For
$\lambda=(u,v)\in \C^2$, let $E_\lambda$ be the
symbol of  $z^{u+1}\partial^{v+1}$. 

Thus the algebra $\mathcal{P}$ of symbols of twisted
pseudo-differential operators  has basis
$E_\lambda$,  where $\lambda$ runs over $\C^2$, and 
the Poisson bracket of symbols is given by the following formula:

\centerline{$\{E_\lambda,E_\mu\}=<\lambda+\rho\vert\,\mu+\rho>\,E_{\lambda+\mu},$}

\noindent where $<\vert>$ denotes the usual symplectic form on
$\C^2$, and where $\rho=(1,1)$. 

Set $\Lambda=\Z^n$ and let $\pi:\Lambda\rightarrow \C^2$
be any injective additive map. By definition, $W_\pi$
is the Lie subalgebra with basis $(E_\lambda)_{\lambda\in\pi(\Lambda)}$.
The Lie algebra $W_\pi$ is obviouly $\Lambda$-graded, for the requirement
that each $E_{\pi(\lambda)}$ is homogeneous of degree $\lambda$.

Now consider the following condition:

 $(\mathcal{C})$\hskip1cm $\pi(\Lambda)\not\subset \C\rho$ and 
$2\rho\notin\pi(\Lambda)$. 

\noindent Under Condition $\mathcal{C}$, the Lie algebra
$W_\pi$ is simple, see Lemma 49. 
Moreover the generalized Witt
algebras constitute a sub-family of the family 
$(W_\pi)$: they correspond to the case where
$\pi(\Lambda)$ lies inside a complex line of $\C^2$.

\subsection{Some references for the Lie algebras:}\label{sect_03}
In the context of the
classification of infinite dimensional Lie algebras, the algebras 
$A^{(1)}_1$, $A^{(2)}_2$ (and all affine Lie algebras)
first appeared in the work of V.G. Kac \cite{Ka1}. At the same time,
they were also introduced by R. Moody in other context \cite{Mo}.

In the context of the classification of
infinite dimensional Lie algebras, the generalized Witt algebras
appeared in the work of R. Yu \cite{Y}. He considered 
simple graded Lie algebras $\mathcal{L}=\oplus_{\lambda}\,\mathcal{L}_\lambda$,
where each homogeneous component $\mathcal{L}_\lambda$ has dimension
one with basis $L_\lambda$.

In our terminology,
the Yu Theorem can be restated as follows:

\bigskip
\noindent{\textbf{Theorem:}} (R.Yu) \textit{Assume that the Lie bracket is given by: \\
\centerline{ $[L_\lambda,L_\mu]=(f(\mu)-f(\lambda))L_{\lambda+\mu}$,}
\noindent for some function $f:\Lambda\rightarrow \C$. Then
$\mathcal{L}$ is an imprimitive form of $A^{(1)}_1$ or an imprimitive form
of a generalized Witt algebra.}

\subsection{General structure of the paper:}\label{sect_04}
The paper is divided into three chapters. For $i=1$ to $3$, Chapter $i$
is devoted to the proof of
Theorem $i$. The Main Theorem is an
immediate consequence of Theorems 1-3, which are stated below.

\subsection{About Theorem 1}\label{sect_05}

Let $\mathcal{G}$ be the class of all
simple graded Lie algebras 
$\mathcal{L}= \oplus_{\lambda}\,\mathcal{L}_\lambda$ where each
homogeneous component $\mathcal{L}_\lambda$ has dimension one.
For  $\mathcal{L}\in \mathcal{G}$, let
$L_\lambda$ be a basis of
$\mathcal{L}_\lambda$, for each $\lambda\in \Lambda$.

We have $[L_0,L_\lambda]=l(\lambda) L_\lambda$ for some function
$l:\Lambda\rightarrow \C$. 
The first step of the
proof is the following alternative:

\bigskip
\noindent{\textbf{Theorem $\mathbf{1}$}}\textit{
\begin{enumerate}
\item[(i)] \textit{The function $l:\Lambda\rightarrow \C$ is additive, or }
\item[(ii)] \textit{there exists $a\in \C$ such
that $l(\Lambda)=[-N,N].a$, for some positive integer $N$.}
\end{enumerate}}

Here $[-N,N]$ denotes the set of integers between $-N$ and $N$.
In the first case, $\mathcal{L}$ is called \textit{non-integrable} and in the
second case $\mathcal{L}$ is called \textit{integrable of type $N$}. Moreover
the only possible value for the type is $1$ or $2$.

Thus Theorem 1 separates the proof into two cases,
and for the integrable case there are some specificities to type 1 
and to type 2. Some statements, like the crucial Main Lemma,  are common for 
all the cases, but they do not admit a unified proof.

In order to state this lemma, denote by 
$\Sigma$ the set of all $\lambda\in\Lambda$ such that
the Lie subalgebra $\C L_\lambda\oplus \C L_0
\oplus \C L_{-\lambda}$ is isomorphic to 
$\mathfrak{sl}(2)$. Obviously, $\lambda$ belongs to
$\Sigma$ iff 

\centerline{$l(\lambda)\neq 0$ and $[L_\lambda,L_{-\lambda}]\neq 0$.}
\noindent
The Main Lemma is the following statement:

\bigskip
\textbf{Main Lemma} \textit{An element $\lambda\in\Lambda$ belongs
to $\Sigma$ whenever  $l(\lambda)\neq 0$.}

\bigskip
For the proof of the Main Lemma, see Lemma \ref{lemma_25} for the case of
integrable Lie algebras of type $2$, Lemma \ref{lemma_33} for
the case of integrable Lie algebras of type $1$ and Lemma \ref{lemma_62}
for the non-integrable case.

\subsection{Statement of Theorem 2}\label{sect_06}

\bigskip
\noindent{\textbf{Theorem $\mathbf{2}$}} \textit{Any primitive integrable 
 Lie algebra in the class $\mathcal{G}$ is isomorphic
to $A^{(1)}_1$ or $A^{(2)}_2$.}

The Main Lemma provides a lot of subalgebras isomorphic to
$\mathfrak{sl}(2)$. Therefore the  proof of Theorem 2 is based on
basic notions, among them $\mathfrak{sl}(2)$-theory,
Jordan algebra, Weyl group, centroid.

\subsection{About the proof of Theorem 3}\label{sect_07}
The main difficulty of the paper is the following statement: \\

\noindent{\textbf{Theorem $\mathbf{3}$}}
\textit{Any primitive non-integrable 
 Lie algebra in class $\mathcal{G}$ is isomorphic
to $W_\pi$ for some injective additive map
$\pi:\Lambda\rightarrow \C^2$ satisfying
condition $\mathcal{C}$.} \\

\noindent The first step  is the proof that for any $\alpha\in\Sigma$, the Lie
algebra  \\
\centerline{$\mathcal{L}(\alpha)=\oplus_{n\in \Z}\,\mathcal{L}_{n\alpha}$}

\noindent is isomorphic to the Witt algebra $W$. Roughly speaking, it
means that any subalgebra $\mathcal{L}_{-\alpha}\oplus\mathcal{L}_0\oplus\mathcal{L}_{\alpha}$ isomorphic to
$\mathfrak{sl}(2)$ "extends" to a Witt algebra, see Lemma \ref{lemma_46}.

For any $\Z \alpha$-coset
$\beta$, set $\mathcal{M}(\beta)=
\oplus_{n\in \Z}\,\mathcal{L}(\beta+n\alpha)$. Thanks to
the Kap\-lan\-sky-Santharoubane Theorem \cite{KS}, the possible 
 $W$-module structures on  $\mathcal{M}(\beta)$ are explicitely classified.

Moreover the Lie bracket in $\mathcal{L}$ provides 
some $W$-equivariant bilinear maps 
$B_{\beta,\gamma}:\mathcal{M}(\beta)\times \mathcal{M}(\gamma)\rightarrow 
\mathcal{M}(\beta+\gamma)$. It turns out that all 
$W$-equivariant bilinear maps $b:L\times M\rightarrow N$, where
$L$, $M$ and $N$ are in the Kaplansky Santharoubane list, have been
recently classified in \cite{IM}. The whole list of \cite{IM} is  
intricate, because it contains many special cases. 
However it allows  a detailed analysis of the Lie bracket of $\mathcal{L}$.

The Main Lemma also holds in this setting, but its proof has little
in common with the previous two cases.

\subsection{Some references for the tools used in the proof}\label{sect_08}
There are two types of tools used in the proof, ``local analysis'' and
``global analysis''. 

A $\Lambda$-graded Lie algebras $\mathcal{L}$ contains a lot
of $\Z$-graded subalgebras $L=\oplus_{n\in \Z}\,L_n$.
Thus it contains some local Lie algebras 
$L_1\oplus L_0\oplus L_{-1}$, which determines some sections of
$\mathcal{L}$. The notion of a local Lie algebra is due to
V.G. Kac \cite{Ka1}. Here, the novelty is the use of infinite dimensional 
local Lie algebras.

Global analysis means to investigate the decomposition of
 $\mathcal{L}$ under a Lie subalgebra $\mathfrak{s}$ and 
to analyse the Lie bracket as an $\mathfrak{s}$-equivariant bilinear map
$\mathcal{L}\times\mathcal{L}\rightarrow\mathcal{L}$.
The prototype of global methods is Koecher-Kantor-Tits
construction, which occurs when a subalgebra $\mathfrak{sl}(2)$ acts
over a Lie algebra
$L$ in a such way that the non-zero isotypical components are
trivial or adjoint: 
the Lie bracket of $L$ is encoded by a Jordan algebra
(up to some technical conditions). Indeed this tool, introduced by Tits
 \cite{T} (see also \cite{Kan}, \cite{Ko}), 
is used in chapter \ref{chapter_II}, as the main ingredient to study 
type $1$ integrable Lie algebras.

The global analysis in the non-integrable case (Section \ref{sect_17}, \ref{sect_18}
and \ref{sect_19}) is quite similar, except that the subalgebra is
the  Witt algebra $W$. In this case, it is proved in Section 18 that 
the Lie bracket is encoded by a certain map 
$c:\Lambda\times\Lambda\rightarrow \C^\ast$. Indeed $c$ satisfies a two-cocyle identity, except 
that it is valid only on a subset of $\Lambda^3$. Therefore,
$c$ is informally called a  ``quasi-two-cocycle".

\subsection{Ground field} 
For simplicity, it has been fixed that the
ground  field is $\C$. However, the final result is valid on
any field $K$ of characteristic zero, but the proof requires the  
following modifications. 

In Section \ref{sect_3},  any norm of the $\Q$-vector space $K$ 
should be used instead of the
complex absolute value. The proof of Lemma \ref{lemma_80} uses that 
$\C$ is uncountable, but it  can be improved to include the case of a
countable field.
Also, the fact that $\C$ is algebraically closed is not essential,
because $\mathcal{L}_0$ has dimension 1 and it can be check that
every structure used in the proof is split.

\setcounter{tocdepth}{1}
\tableofcontents

\part{The alternative integrable/non-integrable}\label{chapter_I}
\section{Generalities, notations, conventions}\label{sect_1}
This section contains the main definitions used throughout the paper.
It also contains a few lemmas which are  repeatedly used.

\subsection{}\label{sect_1.1}
Given an integers $a$, 
set $\Z_{\geq a}=\{n\in \Z\vert\, n\geq a\}$. 
Similarly define 
$\Z_{>a}$, $\Z_{\leq a}$, and $\Z_{<a}$.
Given two integers $a,\, b$ with $a\leq b$, it is 
convenient to set $[a,b]={\Z}_{\geq a}\cap{\Z}_{\leq b}$.

\subsection{}
Let $M$ be an abelian group.
A \textit{weakly $M$-graded vector space}
is a vector space $E$ endowed with a decomposition
$E=\oplus_{m\in M}\,E_{m}$. The subspaces 
$E_m$ are called the \textit{homogeneous components} of $E$.
Given two weakly $M$-graded vector space,
a map $\psi:E\rightarrow F$ is called 
\textit{homogeneous of degree l} if
$\psi(E_m)\subset F_{m+l}$ for all $m\in M$. 

When all homogeneous components are finite dimensional,
$E$ is called a \textit{$M$-graded vector space}.
When the grading group $M$ is tacitely 
determined, $E$ will be called a \textit{graded vector space}.
A subspace $F\subset E$ is \textit{graded} if
$F=\oplus_{m\in M} F_m$, where $F_m=F\cap E_m$.
The set  $\Supp E:=\{m\in M\vert\,E_m\neq 0\}$ is called the 
\textit{support} of $M$.

Set $E'=\oplus_{m\in M}\,(E_m)^\ast$. The space $E'$, which is a subspace of
the ordinary dual $E^\ast$ of $E$, is called the \textit{graded dual}
of $E$. The graded dual $E'$ admits a $M$-gradation defined by
$E'_m=(E_{-m})^\ast$. As it is defined, the duality pairing is homogeneous of
degree zero.

\subsection{}
Let $M$ be an abelian group. A \textit{$M$-graded algebra} is 
$M$-graded vector space $A=\oplus_{m\in M}\,A_m$ with an algebra structure
such that $A_l.A_m\subset A_{l+m}$ for all $l,\,m\in M$, where $.$ denotes
the algebra product. When the notion of
$A$-module is defined, a \textit{graded module} is
a $A$-module $E=\oplus E_m$ such that $A_l.E_m\subset E_{m+l}$.
The notion of \textit{weakly graded algebras} and \textit{weakly graded modules}
are similarly defined.

\subsection{}
An algebra $A$ is called \textit{simple} if
$A.A\neq 0$ and $A$ contains no non-trivial proper two-sided ideal.
A $M$-graded algebra $A$ is called \textit{simple graded}
if $A.A\neq 0$ and $A$ contains no graded
non-trivial proper two-sided ideals. Moreover the
graded algebra $A$ is called \textit{graded simple} if
$A$ is simple (as a non-graded algebra).

\subsection{}
Let $M$ be an abelian group.
Its \textit{group algebra}  is the algebra $C[M]$
with basis $(e^\lambda)_{\lambda\in M}$ and the product is defined by
$e^\lambda\,e^\mu=e^{\lambda+\mu}$. This algebra has a natural
$M$ gradation, for which $e^\lambda$ is homogeneous of degree $\lambda$.
It is clear that  $C[M]$ is a simple graded algebra, although this
algebra is not simple.

\subsection{}\label{sect_1.6}
Let $Q$ and $M$ be abelian groups, let
$\pi: Q\rightarrow M$ be an additive map, and let
$A$ be a $M$-graded algebra. By definition
$\pi^\ast A$ is the subalgebra of $A\otimes \C[Q]$ defined by:

\centerline{
$\pi^\ast A:= \oplus_{\lambda\in Q}\, A_{\pi(\lambda)}\otimes e^\lambda$.}

\noindent The algebra $\pi^\ast A$ is  $Q$-graded and there is
a natural algebra morphism 
$\psi:\pi^\ast A\rightarrow A,\,a\otimes e^\lambda\mapsto a$. Indeed for each
$\lambda\in Q$, $\psi$ induces an isomorphism
$(\pi^\ast A)_\lambda\simeq A_{\pi(\lambda)}$. The following obvious
lemma characterizes $\pi^\ast A$.

\begin{lemma}\label{lemma_1}
Let $B$ be a $Q$-graded algebra and let 
$\phi:B\rightarrow A$ be an algebra morphism. Assume
that $\phi(B_\lambda)\subset A_{\pi(\lambda)}$ 
and that the induced map $B_\lambda\rightarrow  A_{\pi(\lambda)}$
is an isomorphism for each $\lambda\in Q$. Then 
the $Q$-graded algebra $B$ is isomorphic to  $\pi^\ast A$.
\end{lemma}

\subsection{}
From now on, $\Lambda$ denotes a lattice, i.e. a
free abelian group of finite rank.

\subsection{}
Let $M$ be an abelian group.
Denote by $\mathcal{G}(M)$ the class of all simple $M$-graded Lie
algebras $\mathcal{L}=\oplus_{\lambda\in M} \mathcal{L}_{\lambda}$ such
that  $\dim \mathcal{L}_{\lambda}=1$ for all $\lambda\in M$.
 
Denote by $\mathcal{G}'(M)$ the class of all simple $M$-graded Lie
algebras $\mathcal{L}=\oplus_{\lambda\in M} \mathcal{L}_{\lambda}$ such
that  
\begin{enumerate}
\item[(i)] $\dim \mathcal{L}_{\lambda}\leq 1$ for all $\lambda\in\ M$,
\item[(ii)]  $\dim \mathcal{L}_{0}= 1$, and
\item[(iii)] $\Supp \mathcal{L}$ generates $M$.
\end{enumerate}
When it is tacitely assumed that $M=\Lambda$, these classes will
be denoted by $\mathcal{G}$ and $\mathcal{G}'$.

\subsection{}
Let $\mathcal{L}\in \mathcal{G}'$. For any $\lambda\in\Supp\mathcal{L}$,
denotes by $L_\lambda$  any non-zero vector of
$\mathcal{L}_\lambda$. Also denote by
$L^\ast_\lambda$ the element of $\mathcal{L}'$ defined by

\centerline{$<L^\ast_\lambda\vert\,L_\mu>=\delta_{\lambda,\mu}$,}

\noindent where $\delta_{\lambda,\mu}$ is Kronecker's symbol.
Note that $L^\ast_\lambda$ is a homogeneous  element of $\mathcal{L}'$
of degree $-\lambda$.

For $\lambda\neq 0$, the exact normalization
of $L_\lambda$ does not matter. 
However, we will fix once for all the
vector $L_0$. This allows to define the function $l:\Supp \mathcal{L}\rightarrow 
\C$ by the requirement

\centerline{$[L_0,L_{\lambda}]=l(\lambda)L_\lambda$.}

\noindent
The following Lemma is obvious.
\begin{lemma}\label{lemma_2} Let $\lambda,\mu\in\Supp\mathcal{L}$.
If $[L_\lambda,L_\mu]\neq 0$, then 

\centerline{$l(\lambda+\mu)=l(\lambda)+l(\mu)$.}
\end{lemma}

\subsection{}
Let $\Pi$ be the set of all $\alpha\in\Lambda$ such that
$\pm\alpha\in\Supp\mathcal{L}$ and $[L_{\alpha},L_{-\alpha}]\neq 0$. Let
$\Sigma$ be the set of all $\alpha\in\Pi$ such that
$l(-\alpha)\neq 0$.
 
For  $\alpha\in\Pi$, set
$\mathfrak{s}(\alpha)=\C L_{-\alpha}
\oplus \C L_0\oplus \C L_\alpha$. It is
clear that $\mathfrak{s}(\alpha)$ is a Lie subalgebra.
Since $l(-\alpha)=-l(\alpha)$, the Lie algebra $\mathfrak{s}(\alpha)$ is
isomorphic with $\mathfrak{sl}(2)$ if $\alpha\in\Sigma$  and it is 
the Heisenberg algebra of dimension three otherwise.

\subsection{}
The following lemma will be used repeatedly.

\begin{lemma}\label{lemma_3} Let
$L$ be a Lie algebra and let
$A$, $B$ be two subspaces such that
$L=A+B$ and $[A,B]\subset B$. Then $B+[B,B]$ is an ideal of $L$.
\end{lemma}
The proof is easy, see \cite{M1}, Lemma 6.

\subsection{}
Let $X,Y$ be subsets of $\Lambda$. Say that 
$X$ is \textit{dominated} by $Y$, and denote it by $X\leq Y$,  iff there
exists a finite subset $F$ of $\Lambda$ with $X\subset F+Y$. Say that
$X$ is \textit{equivalent} to $Y$, and  denote it by $X\equiv Y$, iff
$X\leq Y$ and $Y\leq X$.

In what follows, a \textit{simple graded module} means a module without
non-trivial graded submodule. The following abstract lemma will be
useful:

\begin{lemma}\label{lemma_4} 
Let $L$ be a $\Lambda$-graded Lie algebra, and let
$M$ be a simple $\Lambda$-graded module. Then for any non-zero
homogeneous elements $m_1, m_2$ of $M$, we have

\centerline{$\Supp L.m_1\equiv\Supp L.m_2$.}
\end{lemma}
\begin{proof}
For any homogeneous element $m\in M$,
set $\Omega(m)=\Supp L.m$. More generally any
$m\in M$ can be decomposed as
$m=\sum\,m_\lambda$ where $m_\lambda\in M_\lambda$. In general set
$\Omega(m)=\cup\Omega(m_\lambda)$. 

Let $m\in M$, $x\in L$  be  homogeneous elements.
We have $L.x.m\subset L.m+x.L.m$, and therefore
$\Omega(x.m)\subset \Omega(m) \cup \deg x +\Omega(m)$. 
It follows that $\Omega(x.m)\leq \Omega(m)$ for any (not necessarily
homogeneous) elements $m\in M$ and $x\in L$. So the subspace $X$ of all
$m$ such that
$\Omega(m)\leq\Omega(m_1)$ is a graded submodule. Since
$M$ is simple as a graded $L$-module, we have $X=M$ and therefore
$\Omega(m_2)\leq\Omega(m_1)$. Similarly we have
$\Omega(m_1)\leq\Omega(m_2)$. Thus we get

\centerline{$\Supp L.m_1\equiv\Supp L.m_2$.}
\end{proof}

\section{Generalities on centroids of graded algebras}\label{sect_2}
Let $A$ be an algebra of  any type
(Jordan, Lie, associative...). 
Its \textit{centroid}, denoted by  $C(A)$, is the
algebra of all maps $\psi:A\rightarrow A$ which commute
with  left and right multiplications, namely

\centerline{$\psi(ab)=\psi(a)b$ and  $\psi(ab)=a\psi(b)$,}

\noindent for any $a,\,b\in A$.

\begin{lemma}\label{lemma_5} Assume that $A=A.A$.
Then the algebra  $C(A)$ is commutative. 
\end{lemma}
\begin{proof}
Let $\phi,\psi\in C(A)$ and
$a,b\in A$. It follows from the definition that:

\centerline{$\phi\circ\psi(ab)=\phi(a)\psi(b)$ and
$\phi\circ\psi(ab)=\psi(a)\phi(b)$. }

\noindent Hence
$[\phi,\psi]$ vanishes on $A.A$ and
therefore $C(A)$ is commutative whenever
$A=A.A$. 
\end{proof}
From now on, let  $Q$ be a countable abelian group.

\begin{lemma}\label{lemma_6}  
Let $A$ be a simple $Q$-graded 
algebra. Then:
\begin{enumerate}
\item[(i)] Any non-zero homogeneous element $\psi\in C(A)$  is invertible.
\item[(ii)] The algebra $C(A)$ is a $Q$-graded algebra.
\item[(iii)] $M:=\Supp C(A)$ is a subgroup of $Q$.
\item[(iv)] The algebra $C(A)$ is  isomorphic
with the group algebra $\C[M]$.
\end{enumerate}
\end{lemma}
\begin{proof}
\textit{Point (i):}
Let $\psi\in C(A)$ be a non-zero and homogeneous. Its image and its
kernel are graded two-sided ideals. Therefore  $\psi$  is
bijective. It is clear that its inverse lies in $C(A)$.

\textit{Point (ii):}  
Let $\psi\in C(A)$. Then $\psi$ can be decomposed as
$\psi=\sum_{\mu\in Q}\,\psi_\mu$, where:
\begin{enumerate}
\item[(i)] the linear map $\psi_\mu:A\rightarrow A$ is homogeneous of degree
$\mu$,
\item[(ii)] For any $x\in A$, $\psi_\mu(x)=0$ for almost all $\mu$.
\end{enumerate}
It is clear that each homogeneous component $\psi_\mu$ is in $C(A)$.
Since each non-zero component $\psi_\mu$ is injective, it follows that 
almost all of them are zero. Therefore each $\psi\in C(A)$ is a
finite sum of its homogeneous components, and thus $C(A)$ admits a
decomposition $C(A)=\oplus_\mu\,C(A)_\mu$.

It is obvious that $C(A)$ is a weakly graded
algebra, i.e. $C(A)_\lambda.C(A)_\mu\subset C(A)_{\lambda+\mu}$
for all $\lambda,\mu\in Q$. So it
remains to prove that each homogeneous component
$C(A)_\mu$ is finite dimensional.
Choose any  homogeneous element $a\in A\setminus\{0\}$ of
degree $\nu$. Since any
non-zero homogeneous element of $C(A)$ is one-to-one, the natural map
$C(A)\rightarrow A, \psi\mapsto \psi(a)$ is injective. Therefore
$\dim\,C(A)_\mu\leq \dim A_{\mu+\nu}<\infty$. Thus $C(A)$ is a
$Q$-graded algebra.

\textit{Point (iii):} Since
any  non-zero homogeneous element in $C(A)$
is invertible,  $M$ is a subgroup of $Q$.

\textit{Point (iv):}  
By Lemma 5, $C(A)$ is commutative. Since $C(A)$ has
countable dimension, any maximal
ideal of $C(A)$ determines a morphism $\chi:C(A)\rightarrow \C$.
Since each non-zero homogeneous element is invertible, the restriction of
$\chi$ to each homogeneous component
$C(A)_\mu$ is injective. Therefore 
$C(A)_\mu$ is one dimensional and  there is a unique
element $E_\mu\in C(A)_\mu$ with $\chi(E_\mu)=1$. It follows that
$(E_\mu)_{\mu\in M}$ is a basis of $C(A)$ and that
$E_{\mu_1}E_{\mu_2}=E_{\mu_1+\mu_2}$. Hence $C(A)$ is
isomorphic with the group algebra $\C[M]$.
\end{proof}
Let $\mathcal{U}=\oplus_{\lambda\in Q}\,\mathcal{U}_\lambda$
be an associative weakly graded algebra. Set 
$\mathcal{A}=\mathcal{U}_0$.

\begin{lemma}\label{lemma_7} 
Let $E$ be a  simple graded $\mathcal{U}$-module.
\begin{enumerate}
\item[(i)] The $\mathcal{A}$-module $E_\lambda$ is simple for any
$\lambda\in \Supp E$.
\item[(ii)] If the $\mathcal{A}$-module $E_\lambda$ and $E_\mu$ are isomorphic
for some $\lambda,\,\mu\in\Supp E$, then there is an
isomorphism of $\mathcal{U}$-modules $\psi:E\rightarrow E$
which is homogeneous of degree $\lambda-\mu$.
\end{enumerate}
\end{lemma}
\begin{proof}
\textit{Point (i)} 
Let $x$ be any non-zero element in $E_\lambda$. Since 
$\mathcal{U}.x$ is a graded submodule  of $E$, we have $\mathcal{U}.x=E$,
i.e. $E_\mu=\mathcal{U}_{\mu-\lambda}.x$ for any $\mu\in Q$. So we get 
$\mathcal{A}.x=E_\lambda$, for any $x\in E_\lambda\setminus 0$,
which proves that the $\mathcal{U}$-module $E_\lambda$ is simple.

\textit{Intermediate  step:} Let $\mu\in Q$. Let $Irr(\mathcal{A})$ be the
category of  simple $\mathcal{A}$-modules and let 
$Irr_\mu(\mathcal{U})$ be the category of simple weakly graded
$\mathcal{U}$-modules $Y$ whose support contains $\mu$. We claim
that the categories $Irr(\mathcal{A})$ and $Irr_\mu(\mathcal{U})$ are
equivalent.

By the first point, the map $E\mapsto E_\mu$ provides a functor 

\centerline{$F^\mu:Irr_\mu(\mathcal{U})\rightarrow Irr(\mathcal{A})$.} 

Conversely, let $X\in Irr(\mathcal{A})$.
Set $I^\mu(X)=\mathcal{U}\otimes_{\mathcal{A}}\,X$, and for
any $\lambda\in Q$, set 
$I^\mu_\lambda(X)=\mathcal{U}_{\lambda-\mu}\otimes_{\mathcal{A}}\,X$.
Relative to the decomposition
 
\centerline{$I^\mu(X)=\oplus_{\lambda\in Q}\,\,I^\mu_\lambda(M)$,}

\noindent $I^\mu(X)$ is a weakly graded $\mathcal{U}$-module. 
Let $K(X)$ be the biggest $\mathcal{U}$-submodule lying in
$\oplus_{\lambda\neq\mu}\,I^\mu_\lambda(M)$ and set
$S^\mu(X)=I^\mu(X)/K(X)$. Since $K(X)$
is a graded subspace, 
$S^\mu(X)$ is a weakly graded $\mathcal{U}$-module, and it
is clear that $S^\mu(X)$ is simple as weakly graded module. 
Moreover its homogeneous component of degre $\mu$
is $X$. Therefore the map $X\mapsto S^\mu(X)$ provides a functor

\centerline{$S^\mu:Irr(\mathcal{A})\rightarrow Irr_\mu(\mathcal{U})$.}

\noindent The functor $F^\mu$ and $S^\mu$ are inverse to each other,
which proves that the category $Irr(\mathcal{A})$ and $Irr_\mu(\mathcal{U})$ are
equivalent.

\textit{Point (ii):} Let  $\lambda,\,\mu\in\Supp E$.
Assume that the $\mathcal{U}$-modules
$E_\lambda$ and $E_\mu$  are isomorphic to the same $\mathcal{A}$-module 
$X$. As $\mathcal{U}$-modules, $S^\mu(X)$ and $S^\lambda(X)$ are isomorphic,
up to a shift by $\lambda-\mu$ of their gradation.
Since we have  $E\simeq S^\mu(X)$ and $E\simeq S^\lambda(X)$,
there is isomorphism of $\mathcal{U}$-modules $\psi:A\rightarrow A$
which is homogeneous of degree $\lambda-\mu$. 
\end{proof}

Let $A$ be a graded algebra. Let $\mathcal{U}\subset \End(A)$ be the
associative subalgebra generated by left and right multiplications. 
Let $\mathcal{A}\subset \mathcal{U}$ be the subalgebra of all
$u\in \mathcal{U}$ such that $u(A_\mu)\subset A_\mu$ for all 
$\mu\in Q$.

\begin{lemma}\label{lemma_8} 
Let $A$ be a simple $Q$-graded  algebra.
Set $M=\Supp C(A)$.
\begin{enumerate}
\item[(i)] For any $\mu\in \Supp A$, the $\mathcal{A}$-module
$A_\mu$ is simple.
\item[(ii)] For any $\lambda,\,\mu\in \Supp A$, the $\mathcal{A}$-modules
$A_\mu$ and $A_\lambda$ are isomorphic iff
$\lambda-\mu\in M$.
\end{enumerate}
\end{lemma}
\begin{proof}
For each $\mu\in Q$, set
$\mathcal{U}_\mu=\{u\in \mathcal{U}\vert\,u(A_\nu)\subset
A_{\mu+\nu},\forall\nu\}$.
Relative to the decomposition 
$\mathcal{U}=\oplus_{\mu\in Q}\,\mathcal{U}_\mu$, the algebra
$\mathcal{U}$ is a weakly graded, and $A$ is a simple graded 
$\mathcal{U}$-module.

The simplicity of each $\mathcal{A}$-module
$A_\mu$ follows from Lemma \ref{lemma_7} (i). 

Assume that $\lambda-\mu\in M$.
Choose any non-zero $\psi\in C(A)$ which is homogeneous 
of degree $\lambda-\mu$. Then $\psi$ provides a morphism of
$\mathcal{A}$-modules $A_\mu\rightarrow A_\lambda$. By Lemma \ref{lemma_6} (i), it is an isomorphism, so 
$A_\lambda$ and $A_\mu$ are isomorphic. 

Conversely, assume that the $\mathcal{A}$-modules
$A_\lambda$ and $A_\mu$ are isomorphic. By
Lemma \ref{lemma_7} (ii), there is a  isomorphism of $\mathcal{U}$-modules
$\psi:E\rightarrow E$ which is homogeneous of degree $\lambda-\mu$.
Since $\psi$ belongs to $C(A)$, it follows that $\lambda-\mu$ belongs to
$M$.
\end{proof}

\begin{lemma}\label{lemma_9} 
Let $A$ be a simple $Q$-graded algebra,
let $M$ be the support of $C(A)$. Set 
$\overline{Q}=Q/M$ and let $\pi:Q\rightarrow \overline{Q}$ be the natural
projection. Then there exists a $\overline{Q}$-graded
algebra $\overline{A}$ such that:
\begin{enumerate}
\item[(i)] $A=\pi^\ast\overline{A}$\hskip.5mm,
\item[(ii)] the algebra  $\overline{A}$ is simple
(as a non-graded algebra).
\end{enumerate}
\end{lemma}
\begin{proof}
Let $m$ be any maximal ideal of $C(A)$.
Set $\overline{ A}=A/m.A$ and let $\psi:A\rightarrow\overline{A}$ 
be the natural projection. 

For $\lambda\in Q$, set 
$\mathcal{M}(\lambda)=\oplus_{\mu\in M}\,A_{\lambda+\mu}$. It follows from
Lemma \ref{lemma_6}  that 
$\mathcal{M}(\lambda)=C(A)\otimes A_\lambda$, $C(A)/m \cong \C$ 
and so we have
$\mathcal{M}(\lambda)/m.\mathcal{M}(\lambda)\simeq A_\lambda$.
Thus $\overline{A}$ is  a $\overline{Q}$-graded algebra. Moreover 
for each $\lambda\in Q$, the morphism $\psi$ induces
an isomorphism $A_\lambda\rightarrow{\overline A}_{\pi(\lambda)}$.
It follows from Lemma \ref{lemma_1}  that
$A\simeq \pi^\ast \overline{A}$.

Since $\psi$ is onto,  it is clear
that $\overline{A}$ is simple graded. By Lemma \ref{lemma_8},  the family
$(\overline{A}_{\overline{\lambda}})_
{\overline{\lambda}\in\Supp \overline{A}} $ consists of
simple, pairwise non-isomorphic $\mathcal{A}$-modules.
Thus any two-sided ideal in $\overline{A}$ is graded,
which proves that the algebra  $\overline{A}$ is simple
(as a non-graded algebra).
\end{proof}

\section{The alternative for the class $\mathcal{G}$}\label{sect_3}
Let $\Lambda$ be a lattice.
This section contains the first key result for
the classification, namely the
alternative for the class $\mathcal{G}$.

However in chapter II, the class $\mathcal{G}'$ will be used
as a tool, and some results are more natural in this 
larger class. Except stated otherwise, $\mathcal{L}$ will
be a Lie algebra in the class $\mathcal{G}'$.

\begin{lemma}\label{lemma_10} 
Let $\alpha\in\Pi$ and  $\lambda\in \Supp \mathcal{L}$.
Assume $l(\lambda)\neq 0$, and set
$x=\vert 2l(\lambda)/l(\alpha)\vert$ if $l(\alpha)\neq 0$ or
$x=+\infty$ otherwise. Then at least one of the following
two assertions holds:
\begin{enumerate}
\item[(i)] $\ad^n(L_\alpha)(L_\lambda)\neq 0$ for all $n\leq x$, or 
\item[(ii)] $\ad^n(L_{-\alpha})(L_\lambda)\neq 0$ for all $n\leq x$.
\end{enumerate}
\end{lemma}
\begin{proof}
Set
$\mathfrak{s}(\alpha):=\C L_{\alpha}\oplus
\C L_0\oplus \C L_{-\alpha}$.
Recall that 
$\mathfrak{s}(\alpha)$ is a Lie subalgebra. Moreover
$\mathfrak{s}(\alpha)$ is isomorphic 
to $\mathfrak{sl}(2)$  if $\alpha\in\Sigma$ or
to the 3-dimensional Heisenberg algebra
$\mathfrak{H}$
if $\alpha\in\Pi\setminus\Sigma$. So the lemma follows from standard
representation theory of $\mathfrak{sl}(2)$ and
$\mathfrak{H}$.

Here is a short proof.
One can assume that 
$\ad^n(L_{\pm\alpha})(L_\lambda)=0$ for $n$ big enough,
otherwise the assertion   is obvious. Set 

\centerline{$N^{\pm}=\mathrm{Sup} \, \{n\vert
\ad^n(L_{\pm\alpha})(L_\lambda)\neq 0\}$.}

Let $M$ be the $\mathfrak{s}(\alpha)$-module generated by 
$L_{\lambda}$. Then the family

\noindent$(L_{\lambda+n\alpha})_{n\in [-N^-, N^+]}$  is a
basis for $M$, and therefore $M$ is finite dimensional.
Since $L_0$ is a scalar mutiple of $[L_\alpha,L_{-\alpha}]$,
we have $\tr \,L_0\vert_M=0$. Since 
$[L_0, L_{\lambda+n\alpha}]=(l(\lambda)+nl(\alpha))\,
L_{\lambda+n\alpha}$  for all $n\in [-N^-, N^+]$, we get
\begin{align*}
0=
&\sum _{-N^-\leq n\leq N^+}\,(l(\lambda)+nl(\alpha)) \\
=
&(N^++N^-+1)[(N^+-N^-)l(\alpha)/2+l(\lambda)]. 
\end{align*}
\noindent Therefore we have
$\vert N^+-N^-\vert=2\vert l(\lambda)/ l(\alpha)\vert=x$, which proves
that
$N^+$ or $N^-$ is $\geq x$. 
\end{proof} 

\begin{lemma}\label{lemma_11} Let 
$\alpha_1,\alpha_2,...,\alpha_n
\in\Pi$. 

There exists a positive real number 
$a=a(\alpha_1,\dots,\alpha_n)$ such that
for every $\lambda\in \Supp \mathcal{L}$, and every  n-uple
$(m_1,m_2,...,m_n)$ of integers with $0\leq m_i\leq a\vert
l(\lambda)\vert$ for all $i$, there exists 
$(\epsilon_1,\dots,\epsilon_n)\in\{\pm 1\}^n$ such that:

\centerline{$\ad^{m_n}(L_{\epsilon_n\alpha_n})\dots
\ad^{m_1}(L_{\epsilon_1\alpha_1}) (L_\lambda)
\neq0$.}
\end{lemma}
\begin{proof}
First define the real number $a$. 
If $l(\alpha_i)=0$ for all $i$, choose any $a>0$. Otherwise, set
$s=\Max \vert l(\alpha_i)\vert$ and set 
$a=1/ns$. Let $\lambda\in\Lambda$. Then the following assertion:

$(\mathcal{H}_k)$:
for every  k-uple $(m_1,m_2,...,m_k)$ of
integers with $0\leq m_i\leq a\vert l(\lambda)\vert$ for all $i$, there
exists  $(\epsilon_1,\dots,\epsilon_k)\in\{\pm 1\}^k$ such that

\centerline{$\ad^{m_k}(L_{\epsilon_k\alpha_k})\dots
\ad^{m_1}(L_{\epsilon_1\alpha_1}) (L_\lambda)
\neq0$.}

\noindent will be proved by induction on $k$, 
for $0\leq k\leq n$. Clearly, one can
assume that $l(\lambda)\neq 0$. 

Assume that ${\cal H}_k$ holds for some $k<n$. Let 
$(m_1,m_2,...,m_{k+1})$ be a $(k+1)$-uple of
integers with $0\leq m_i\leq a\vert l(\lambda)\vert$ for all $i$.
By hypothesis, there exists 
$(\epsilon_1,\dots,\epsilon_k)\in\{\pm 1\}^k$ such that 
$X:=\ad^{m_k}(L_{\epsilon_k\alpha_k})\dots
\ad^{m_1}(L_{\epsilon_1\alpha_1}) (L_\lambda)$ is not
zero. Set $\mu=\lambda+\sum_{i\leq k}\,\epsilon_i m_i \alpha_i$.
Since $\deg X=\mu$, we have
$l(\mu)=l(\lambda)+\sum_{i\leq k}\,\epsilon_i m_i l(\alpha_i)$.
It follows that 

\begin{align*}
\vert l(\mu)\vert\geq 
&\vert l(\lambda)\vert-\sum_{i\leq k}\,m_i \vert l(\alpha_i)\vert \\
\geq 
&\vert l(\lambda)\vert-
\sum_{i\leq k}\,a \vert l(\lambda)\vert \\
\geq 
&\vert l(\lambda)\vert-
\sum_{i\leq k}\, \vert l(\lambda)\vert/n \\
\geq  
&(n-k) \vert l(\lambda)\vert/n \\
\geq   
&\vert l(\lambda)\vert/n,
\end{align*}
\noindent  and therefore we have:

 \centerline{$\vert 2 l(\mu)/l(\alpha_{i+1})\vert
\geq \vert  l(\mu)/l(\alpha_{i+1})\vert
\geq \vert l(\lambda)\vert/ns= a \vert l(\lambda)\vert$. }

\noindent By Lemma \ref{lemma_10}, there exists some
$\epsilon\in\{\pm 1\}$ such that 
$\ad^m(L_{\epsilon\alpha_{k+1}})(X)\neq 0$ for any $m\leq a
\vert l(\mu)\vert$. Set $\epsilon_{k+1}=\epsilon$.
It  follows that 

\centerline{$\ad^{m_{k+1}}(L_{\epsilon_{k+1}\alpha_{k+1}})\dots
\ad^{m_1}(L_{\epsilon_1\alpha_1}) (L_\lambda)\neq 0$}

\noindent   Therefore Assertion $(\mathcal{H}_{k+1})$
is proved.

Since $(\mathcal{H}_0)$ is trivial, the assertion $(\mathcal{H}_n)$
is proved and the lemma follows. 
\end{proof}

\begin{lemma}\label{lemma_12} 
The set $\Pi$ generates $\Lambda$, and  $\Sigma$ is not empty.
\end{lemma}
\begin{proof}
Let $M$ be the sublattice generated by
$\Pi$ and let $N$ be its complement.  Set $\mathcal{A}=\oplus_{\lambda\in
M} \mathcal{L}_{\lambda}$,
$\mathcal{B}=\oplus_{\lambda\in N} \mathcal{L}_{\lambda}$
and $\mathcal{I}=\mathcal{B}+ [\mathcal{B},\mathcal{B}]$.
Since $\Lambda=M\cup N$ and $M+N\subset N$, we have
$\mathcal{L}=\mathcal{A}+\mathcal{B}$ and 
$[\mathcal{A},\mathcal{B}]\subset \mathcal{B}$. Therefore, 
it follows from Lemma \ref{lemma_3} that $\mathcal{I}$ is
an ideal.

By hypothesis, $N$ does not contain $0$ nor any element in
$\Pi$ and thus $\mathcal{I}\cap\mathcal{L}_0=\{0\}$. Since  
$\mathcal{I}\neq\mathcal{L}$, it follows that $\mathcal{I}=0$. Therefore
the support of $\mathcal{L}$ lies in $M$. Since
$\Supp \mathcal{L}$ generates $\Lambda$, it is proved that $\Pi$ generates
$\Lambda$. 

The proof of the second assertion is similar.
Set $\mathcal{A}=\oplus_{l(\lambda)=0}\, \mathcal{L}_{\lambda}$,
$\mathcal{B}=\oplus_{l(\lambda)\neq 0}\, \mathcal{L}_{\lambda}$
and $\mathcal{I}=\mathcal{B}+ [\mathcal{B},\mathcal{B}]$. We have
$\mathcal{L}=\mathcal{A}+\mathcal{B}$ and by Lemma \ref{lemma_2} we have
$[\mathcal{A},\mathcal{B}]\subset \mathcal{B}$. Therefore $\mathcal{I}$ is an ideal.
Since $L_0$ is not central, this ideal is not trivial and so 
$\mathcal{I}=\mathcal{L}$. It follows easily that $L_0$ belongs to
$[\mathcal{B},\mathcal{B}]$, thus there is $\alpha\in\Lambda$ such that
$l(\alpha)\neq 0$ and $[\mathcal{L}_\alpha,\mathcal{L}_{-\alpha}]\neq 0$.
By definition $\alpha$ belongs to $\Sigma$, so $\Sigma$ is
not empty. 
\end{proof}

A subset $B\in\Lambda$ is called a \textit{$\Q$-basis} if 
it is a basis of $\Q\otimes \Lambda$. Equivalently, it means that
$B$ is a basis of a finite index sublattice in $\Lambda$.
From the previous lemma,  $\Pi$ contains some
$\Q$-basis $B$. 
Then  define the additive map $L_B:\Lambda\rightarrow \C$ by
the requirement that $L_B(\beta)=l(\beta)$ for all $\beta\in B$.

\begin{lemma}\label{lemma_13} Assume that $l$ is an  unbounded
function. Then
we have $L_B(\alpha)=l(\alpha)$, for any $\alpha\in \Pi$.
\end{lemma}
\begin{proof}
For clarity, the proof is divided into 
four steps. 

\textit{Step one:} Additional notations are now introduced. 

Let $\vert\vert\,\,\vert\vert$ be a  euclidean norm on $\Lambda$, i.e.
the restriction of a usual norm on $\R \otimes\Lambda$.
For any positive real number $r$, the ball of radius $r$ is the set
$B(r)=\{\lambda\in\Lambda\vert\,
\vert\vert\lambda\vert\vert\leq r\}$.
There is a positive real numbers  $v$ such that
Card $B(r)\leq v r^n$ for all $r\geq 1$.

Fix $\alpha\in \Pi$. Set $B=\{\alpha_1,\dots,\alpha_n\}$, where $n$ is
the rank  of $\Lambda$ and set $\alpha_{n+1}=\alpha$.
Let $a=a(\alpha_1,\dots,\alpha_{n+1})$ be the constant of 
Lemma \ref{lemma_11} and set 
$b=\sum_{1\leq i\leq n+1}\,\vert\vert\alpha_i\vert\vert$. 
Also for $s\geq 0$, let 
$A(s)$ be the set of all $(n+1)$-uples $\mathbf{m}=(m_1,\dots,m_{n+1})$
of integers with $0\leq m_i\leq s$ for any $i$.

\textit{Step two:} There exists $r_0>1/ab$ such that 
$(a r)^{n+1}>v  (abr)^n$ for any $r>r_0$. Equivalently, we
have:
\[ \Card\, A(ar)>\Card\, B(abr)\]
for any $r>r_0$.

\textit{Step three:} Let $\lambda\in\Supp\mathcal{L}$. 
Define two maps:
$\epsilon_{\lambda}: A(a\vert l(\lambda)\vert )\rightarrow 
\{\pm 1\}^{n+1}$ and
$\Theta_\lambda:A(a\vert l(\lambda)\vert )
\rightarrow B(ab\vert l(\lambda)\vert )$, by the following requirement.
By Lemma \ref{lemma_11}, for each $(n+1)$-uple
$\mathbf{m}=(m_1,\dots,m_{n+1})\in 
A(a\vert l(\lambda)\vert )$, there exists
$(\epsilon_1,\dots,\epsilon_{n+1})\in\{\pm 1\}^{n+1}$ such that

\centerline{$\ad^{m_{n+1}}(L_{\epsilon_{n+1}\alpha_{n+1}})\dots
\ad^{m_1}(L_{\epsilon_1\alpha_1}) (L_\lambda)
\neq0$.}

\noindent Thus set

$\epsilon_{\lambda}(\mathbf{m})=(\epsilon_1,,\dots,\epsilon_{n+1})$, and
$\Theta_\lambda(\mathbf{m})=
\sum_{1\leq i\leq n+1}\, \epsilon_i m_i\alpha_i$. 

\noindent From the definition of the maps $\epsilon_\lambda$ and
$\Theta_\lambda$, it follows  that 

\noindent($\ast$)\hskip1.4cm $\ad^{m_{n+1}}(L_{\epsilon_{n+1}\alpha_{n+1}})\dots
\ad^{m_1}(L_{\epsilon_1\alpha_1})
(L_\lambda)=c L_{\lambda+\Theta_\lambda(\mathbf{m})}$,

\noindent where c is a non-zero scalar. \\
Moreover, we have
\begin{align*}
\vert\vert \Theta_\lambda({\bf m})\vert\vert\leq
&
\sum_{1\leq i\leq n+1}\, m_i\vert\vert\alpha_i\vert\vert \\
\leq 
&a\vert l(\lambda)\vert
\sum_{1\leq i\leq n+1}\, \vert\vert\alpha_i\vert\vert \\
\leq 
&ab \vert l(\lambda)\vert,
\end{align*}
\noindent and therefore $\Theta_\lambda$ takes value in $B(ab\vert
l(\lambda)\vert)$.

\textit{Step four:} Since the function $l$ is unbounded, one can choose 
$\lambda\in\Supp\mathcal{L}$ such that $\vert l(\lambda)\vert  >r_0$.

It follows from Step two that $\Theta_{\lambda}$ is not
injective. Choose two distinct elements $\mathbf{m}$, $\mathbf{m'}\in A(a\vert
l(\lambda)\vert)$ with  $\Theta_\lambda(\mathbf{m})=\Theta_\lambda(\mathbf{m'})$.
Set $\mathbf{m}=(m_1,\dots,m_{n+1})$, $\mathbf{m'}=(m_1',\dots,m_{n+1}')$,
$\epsilon_{\lambda}(\mathbf{m})=(\epsilon_1,\dots,\epsilon_{n+1})$ and 
$\epsilon_{\lambda}(\mathbf{m'})=(\epsilon_1',\dots,\epsilon_{n+1}')$. 
Using Identity ($\ast$), we have:
\begin{align*}
l(\lambda+\Theta_\lambda(\mathbf{m}))= 
&l(\lambda)+\sum_{1\leq k\leq n+1}\,
\epsilon_i m_i l(\alpha_i)\\
=
&l(\lambda)+L_B(\mu)+\epsilon_{n+1}m_{n+1} l(\alpha_{n+1}), 
\end{align*}
where $\mu=\sum_{1\leq k\leq n}\, \epsilon_i m_i\alpha_i$.
Similarly, we get 

\centerline{$l(\lambda+\Theta_\lambda({\bf m'}))=
l(\lambda)+L_B(\mu')+\epsilon_{n+1}'m_{n+1}' l(\alpha_{n+1}),$}

\noindent where $\mu'=\sum_{1\leq k\leq n}\, \epsilon_i' m_i' \alpha_i$.
Therefore, we get:
\begin{align*}
&(\epsilon_{n+1}m_{n+1}-\epsilon_{n+1}'m_{n+1}') l(\alpha_{n+1})
+L_B(\mu-\mu')=0. \\
&(\epsilon_{n+1}m_{n+1}-\epsilon_{n+1}'m_{n+1}') \alpha_{n+1}
+\mu-\mu'=0.
\end{align*}
Since $\{\alpha_1,\dots,\alpha_n\}$ is a $\Q$-basis, 
we have $\epsilon_{n+1}m_{n+1}\neq \epsilon_{n+1}'m_{n+1}'$.
It follows from the previous two identities that
$l(\alpha_{n+1})=L_B(\alpha_{n+1})$, which proves the lemma.
\end{proof}

A function $m:\Supp\mathcal{L}\rightarrow \C$ is called
\textit{additive} if there is an additive function
$\widetilde{m}:\Lambda\rightarrow \C$ whose restriction
to $\Supp \mathcal{L}$ is $m$. Since $\Supp \mathcal{L}$ generates
$\Lambda$, the function $M$ is uniquely determined by $m$.

\begin{lemma}\label{lemma_14} Assume that $l$ is an unbounded  function. 
Then the function $l$ is additive.
\end{lemma}
\begin{proof}
It follows from the previous lemma that there exists
an additive function $L:\Lambda\rightarrow \C$ such that
$l(\alpha)=L(\alpha)$ for any $\alpha \in \Pi$.

 Set $M=\{\lambda\in \Supp\mathcal{L}\vert\, l(\lambda)=L(\lambda)\}$  
and let $N$ be its complement.  Set 
$\mathcal{A}=\oplus_{\lambda\in M}\, \mathcal{L}_{\lambda}$,
$\mathcal{B}=\oplus_{\lambda\in N}\, \mathcal{L}_{\lambda}$
and $\mathcal{I}=\mathcal{B}+ [\mathcal{B},\mathcal{B}]$.
We have $\mathcal{L}=\mathcal{A}+\mathcal{B}$. Since $l(\lambda+\mu)=
l(\lambda)+l(\mu)$ whenever $[L_\lambda,L_\mu]\neq 0$
(Lemma \ref{lemma_2}), we also have 
$[\mathcal{A},\mathcal{B}]\subset \mathcal{B}$. Therefore, 
it follows from Lemma \ref{lemma_3} that $\mathcal{I}$ is
an ideal.

By assumption, $N$ does not contains $0$, nor any element in
$\Pi$ and thus $\mathcal{I}\cap\mathcal{L}_0=\{0\}$. Since  
$\mathcal{I}\neq\mathcal{L}$, it follows that $\mathcal{I}=0$.
Therefore $N=\emptyset$, which implies that $l$ is additive. 
\end{proof}

An algebra $\mathcal{L}\in \mathcal{G}'$ is called \textit{integrable} if
the function $l:\Supp\mathcal{L}\rightarrow \C$ is bounded
(it is similar to the definition of integrability in \cite{M1},\cite{M2}).
Otherwise the Lie algebra
$\mathcal{L}$ will be called \textit{non-integrable}. 

Recall that $\Sigma$ is the set of all
$\alpha$ such that the Lie algebra
$\mathfrak{s}(\alpha)=\mathcal{L}_{-\alpha}\oplus \mathcal{L}_0\oplus \mathcal{L}_{\alpha}$ is isomorphic to $\mathfrak{sl}(2)$.

\begin{lemma}\label{lemma_15} Assume that $\mathcal{L}\in \mathcal{G'}$ is integrable. For any $\beta\in\Sigma$, we have:

\centerline{$l(\Lambda)\subset \Z.l(\beta)/2$.}
\end{lemma}
\begin{proof} Since $\beta\in\Sigma$, there is a
$\mathfrak{sl}(2)$-triple $(e,f,h)$ with $e\in\mathcal{L}_\beta$,
$f\in\mathcal{L}_{-\beta}$ and $h=2L_0/l(\beta)$.
Since $\mathcal{L}$  is integrable, it is a direct sum of
$\mathfrak{s}(\beta)$-modules of finite dimension and the eigenvalues
of $h$ are integers. Thus the eigenvalues of $L_0$ are integral 
multiple of $l(\beta)/2$. 
\end{proof}

\begin{lemma}\label{lemma_16}  Assume that $\mathcal{L}\in \mathcal{G'}$ is integrable.
There exists $\alpha\in\Sigma$, and an integer $N\in \Z_{>0}$ such
that: 

\centerline{$l(\Lambda)=[-N,N].l(\alpha)$.}
\end{lemma}
\begin{proof}
First we claim that
$\Im l\subset \Z.l(\alpha)$ for some $\alpha\in\Sigma$.

By Lemma \ref{lemma_12}, $\Sigma\neq\emptyset$. Choose
any $\beta\in \Sigma$. If $\Im l\subset \Z.l(\beta)$, 
the claim is proved. 

Otherwise, set $A=\{\lambda\in \Supp\mathcal{L}\vert\,l(\lambda)\in
\Z.l(\alpha)\}$ and let
$B$ be its complement. Also set 
$\mathcal{A}=\oplus_{\lambda\in A}\, \mathcal{L}_\lambda$ 
and
$\mathcal{B}=\oplus_{\lambda\in B}\,\mathcal{L}_\lambda$.
It is clear that $\mathcal{L}=\mathcal{A}\oplus \mathcal{B}$ and
$[\mathcal{A},\mathcal{B}]\subset \mathcal{B}$. By Lemma \ref{lemma_3}, 
$\mathcal{B}+[\mathcal{B},\mathcal{B}] $ is an ideal, therefore
there exists $\alpha\in B$ such that
$[L_\alpha,L_{-\alpha}]$ is
a non-zero multiple of $L_0$. Since $l(\alpha)\notin
\Z.l(\beta)$, we have $l(\alpha)\neq 0$, and so
$\alpha$ lies in $\Sigma$. 

By Lemma \ref{lemma_15}, there are integers $m$ and $n$ such that
$l(\beta)=ml(\alpha)/2$ and $l(\alpha)=nl(\beta)/2$, which
implies $mn=4$. Moreover, 
$n/2$ is not an integer. Thus $n=\pm 1$ and
$l(\alpha)=\pm l(\beta)/2$. By Lemma \ref{lemma_15},
$\Im l\subset \Z.l(\beta)/2=\Z.l(\alpha)$, and therefore the
claim is proved.

Since $l$ is bounded, there a finite set $X\subset \Z$ such that 
$\Im l=X.l(\alpha)$. Since $\mathfrak{s}(\alpha)$ is isomorphic to
$\mathfrak{sl}(2)$, $\mathcal{L}$ is a direct sum of finite dimensional
simple $\mathfrak{s}(\alpha)$-modules.
Thus it follows that $X$ is necessarily a
symmetric interval $[-N, N]$.
\end{proof}

Let $\mathcal{L}\in \mathcal{G'}$ be integrable. The \textit{type} of $\mathcal{L}$
is the integer $N$ such that $l(\Lambda)=[-N,N].l(\alpha)$ for
some $\alpha\in\Sigma$.

\begin{thm}\label{theorem_1} (Alternative for the class $\mathcal{G}$) 
Let $\mathcal{L}$ be a  Lie
algebra in the class $\mathcal{G}$. Then $\mathcal{L}$ satisfies  one of the
following two assertions:
\begin{enumerate}
\item[(i)] The function $l:\Lambda\rightarrow \C$
is additive, or
\item[(ii)] there exists $\alpha\in\Sigma$ such
that $l(\Lambda)=[-N,N].l(\alpha)$, for some positive integer $N$.
\end{enumerate}
\end{thm}
\begin{proof}
Theorem \ref{theorem_1} follows from Lemmas \ref{lemma_14} and \ref{lemma_16}.
\end{proof}

\part{Classification of integrable Lie algebras
in ${\cal G}$}\label{chapter_II}

\section{Notations and conventions for chapter \ref{chapter_II}}\label{sect_4}
Let $\Lambda$ be a lattice.
Let $\mathcal{L}$ be an integrable Lie algebra
in the class  $\mathcal{G'}$ of type $N$. By definition,
the spectrum of $\ad(L_0)$ is 
$[-N,N].x$, for some scalar $x\in \C^\ast$. 
After a renormalization of $L_0$, 
it can be assumed that $x=1$. 

Set $\mathcal{L}^i=\oplus_{l(\beta)=i}\, \mathcal{L}_\beta$. There is a
decomposition:

\centerline{$\mathcal{L}=\oplus_{i\in [-N,+N]}\, \mathcal{L}^i$.}

\noindent Relative to this decomposition, $\mathcal{L}$ is weakly 
$\Z$-graded: in
general, the homogeneous components $\mathcal{L}^i$ are infinite dimensional.
For any integer $i$, set $\Sigma_i=\{\beta\in\Sigma\vert \,l(\beta)=i\}$.
Similarly, there is a decomposition

\centerline{$\Sigma=\cup_{i\in [-N,+N]}\,\Sigma_i$.}

\section{The Lie subalgebra $\mathcal{K}$}\label{sect_5}

Let $e,f,h$ be the standard basis of $\mathfrak{sl}(2)$.
For a finite dimensional $\mathfrak{sl}(2)$-module $V$,
set $V^i=\{v\in V\vert h.v=2i v\}$. A simple finite dimensional
$\mathfrak{sl}(2)$-module is called \textit{spherical} if 
$V^0\neq 0$, or, equivalently if the eigenvalues of 
$h$ are even integers,  or, equivalently if
$\dim V$ is odd. For a spherical module $V$, the elements
of  $V^0$ are called the \textit{spherical vectors} of $V$.

Let $\pi: \mathfrak{sl}(2)\rightarrow \C$
defined by $\pi(h)=1$, $\pi(e)=\pi(f)=0$.
Let $U,V$ be two spherical simple finite dimensional
$\mathfrak{sl}(2)$-modules, and let
$b:U\otimes V\rightarrow \mathfrak{sl}(2)$ be a non-zero
$\mathfrak{sl}(2)$-morphism. Define $B:U\times V\rightarrow \C$
by $B(u,v)=\pi(b(u\otimes v))$. Denote respectively by $K(U)\subset U$ and 
$K(V)\subset V$
the left and the right kernel of the bilinear map $B$.

\begin{lemma}\label{lemma_17} (with the previous notations).
Assume that $b\neq 0$ and $\dim V\geq \dim U$.
Then one of the following statements holds:
\begin{enumerate}
\item[(i)] We have $\dim U=\dim V$, 
$K(U)=U^0$ and $K(V)=V^0$, or
\item[(ii)] we have $\dim U=2m-1$,  $\dim V=2m+1$, $K(U)=0$ and
$K(V)=V^m\oplus V^{-m}$ for some positive integer $m$.
\end{enumerate}
\end{lemma}
\begin{proof}
It follows easily from the tensor product decomposition 
formula  for $\mathfrak{sl}(2)$-modules 
that $U\otimes V$ contains the adjoint module iff
$\dim U=\dim V$ or if $\dim U=\dim V-2$. 
These dimensions
are odd because $U$ and $V$ are spherical, so we have
\begin{enumerate}
\item[case 1:] $\dim U=\dim V$, or
\item[case 2:] $\dim U=2m-1$,  $\dim V=2m+1$, for some positive integer $m$.
\end{enumerate}
Moreover, the adjoint module
appears with multiplicity one, so $b$ is uniquely defined up to a scalar
multiple.

In the first case, identify $V$ with $U^\ast$ and 
$\mathfrak{sl}(2)$ with its dual. Up to a non-zero scalar multiple, 
$b$ can be identified with the  map 
$U\otimes V\rightarrow \mathfrak{sl}(2)^\ast$ defined as follows:

\centerline{$b(u\otimes v)$ is the linear map 
$x\in \mathfrak{sl}(2)\mapsto <x.u\vert v>$,}

\noindent for any $u\in U$, $v\in V$,
where $<\vert>$ is the duality pairing between $U$ and $V$,
and where $x.u$ denotes the action of $x$ on $u$.
So, up to a non-zero scalar multiple, we have 
$B(u,v)=<h.u\vert v>$. It follows easily that the left kernel
of $B$ is the kernel of $h\vert_U$, namely $U^0$. By symmetry,
the right kernel of $B$ is $V^0$.

In the second case, let $E=\C X\oplus \C Y$ be the two dimensional
representation of $\mathfrak{sl}(2)$, where $X$ and $Y$ denotes
two eigenvectors for $h$. 
Identify $U$ with $S^{2(m-1)}E$, $V$ with $S^{2m}E$
and $\mathfrak{sl}(2)$ with the dual of $S^2 E$,
where $S^lE$ denotes  the space of degree $l$ homogeneous polynomials in
$X$ and $Y$. Up to a non-zero scalar multiple, $b$ can be identified with
the  map  $S^{2(m-1)}E \otimes S^{2(m+1)}E\rightarrow (S^2 E)^\ast$ defined
as follows:

\centerline{$b(F\otimes G)$ is the linear map 
$H\in S^2E \mapsto <H.F\vert G>$,}

\noindent for any $F\in S^{2(m-1)}E $, $G\in S^{2(m+1)}E$, 
where $<\vert>$ is the $\mathfrak{sl}(2)$-invariant
pairing on
$S^{2(m+1)}E$, and where $H.F$ denotes product of polynomials
$H$ and $F$.
So, up to a non-zero scalar multiple, we have 
$B(F,G)=<XYF\vert G>$. Since the multiplication by $XY$ is injective,
it follows that $K(U)=0$. Since $B$ is $h$-invariant, it follows that
the right kernel is generated by $X^{2m}$ and 
$Y^{2m}$, i.e. $K(V)=V^m\oplus V^{-m}$. 
\end{proof}

Set $\Lambda_e=\Lambda\times \Z$. 
Then $\mathcal{L}$ admits the natural 
$\Lambda_e$-gradation:
\[ \mathcal{L}=\bigoplus_{(\beta,i)\in\Lambda_e}\,
\mathcal{L}_\beta^i,\]
where 
$\mathcal{L}_\beta^i=\mathcal{L}_\beta\cap \mathcal{L}^i$. Set
$\Supp_e \mathcal{L}=\{(\beta,i)\in\Lambda_e\vert\,
\mathcal{L}_\beta^i\neq 0\}$.

For
$\alpha\in\Sigma_1$, define $\pi_\alpha:\Lambda_e\rightarrow\Lambda$
by $\pi_\alpha(\beta,i)=\beta-i\alpha$. For $\gamma\in\Lambda$,
set 
$\mathcal{M}(\alpha,\gamma)=\oplus_{\pi_\alpha(\beta,i)=\gamma}\,
\mathcal{L}_\beta^i$. It is clear that
 $[\mathcal{M}(\alpha,\gamma_1),\mathcal{M}(\alpha,\gamma_2)]
\subset \mathcal{M}(\alpha,\gamma_1+\gamma_2)$, for any
$\gamma_1,\,\gamma_2\in\Lambda$. Since 
$\Supp_e \mathcal{L}\subset\Lambda\times[-N,N]$, it follows that
$\dim \mathcal{M}(\alpha,\gamma)\leq 2N+1$ for all $\gamma$.

Therefore, the decomposition
\[ \mathcal{L}=\bigoplus_{\gamma\in\Lambda}\,\mathcal{M}(\alpha,\gamma) \]
provides a new  $\Lambda$-gradation of $\mathcal{L}$.
It is clear that each $\mathcal{M}(\alpha,\gamma)$ is an 
$\mathfrak{s}(\alpha)$-module.

\begin{lemma}\label{lemma_18} Let $\alpha\in\Sigma_1$. 
\begin{enumerate}
\item[(i)] $\mathcal{M}(\alpha,\gamma)$ is not zero iff $\gamma$ belongs to $\Supp
\mathcal{L}^0$.
\item[(ii)] For any $\gamma\in \Supp \mathcal{L}^0$, 
the   $\mathfrak{s}(\alpha)$-module $\mathcal{M}(\alpha,\gamma)$
is simple and it is generated by the spherical vector $L_\gamma$.
\item[(iii)] In particular, we have $\mathcal{M}(\alpha,0)=\mathfrak{s}(\alpha)$.
\end{enumerate}
\end{lemma}
\begin{proof}
Since $\alpha$ belongs to $\Sigma_1$, any simple
component of the $\mathfrak{s}(\alpha)$-module 
$\mathcal{L}$ is spherical. Assume that $\mathcal{M}(\alpha,\gamma)$ 
is not zero. Then its spherical part is $\mathcal{L}_\gamma$, and so
$\gamma$ belongs to $\Supp \mathcal{L}^0$. Moreover its spherical
part has dimension one, so $\mathcal{M}(\alpha,\gamma)$ is simple.
Thus Points (i) and (ii) are proved.

Since $\mathcal{M}(\alpha,0)$ contains $\mathfrak{s}(\alpha)$,
it follows  that 
$\mathcal{M}(\alpha,0)=\mathfrak{s}(\alpha)$. 
\end{proof}

Let $B:\mathcal{L}\times\mathcal{L}\rightarrow \C$ be the 
skew-symmetric bilinear form 
defined by $B(X,Y)=L^\ast_0([X,Y])$, for any $X,Y\in\mathcal{L}$. 
Its kernel, denoted  $\mathcal{K}$, is
a Lie subalgebra.
Set $\mathcal{K}^i=\mathcal{K}\cap\mathcal{L}^i$ for any integer $i$.

\begin{lemma}\label{lemma_19} 
Let $\alpha\in\Sigma_1$ and  $\gamma\in\Lambda$.
Assume that:

\centerline{ $[\mathcal{M}(\alpha,\gamma),\mathcal{M}(\alpha,-\gamma)]
\neq 0$.}

\noindent Then one of the following two assertions holds:
\begin{enumerate}
\item[(i)]  $\dim \mathcal{M}(\alpha,\gamma)=\dim \mathcal{M}(\alpha,-\gamma)$, or
\item[(ii)] $\dim \mathcal{M}(\alpha,\epsilon\gamma)=3$ and 
$\dim \mathcal{M}(\alpha,-\epsilon\gamma)=1$, for some
 $\epsilon=\pm1$.
\end{enumerate}
\end{lemma}
\begin{proof}
Assume that neither Assertion (i) nor (ii) holds. Moreover, 
it can be assumed, without loss of generality, that
$\dim \mathcal{M}(\alpha,\gamma)>
\dim \mathcal{M}(\alpha,-\gamma)$.

By Lemma \ref{lemma_18}, the modules $\mathcal{M}(\alpha,\pm\gamma)$ are spherical and
$\mathcal{M}(\alpha,-\gamma)$ is not the trivial representation.
Thus the hypotheses imply that
  
\centerline {$\dim \mathcal{M}(\alpha,-\gamma)\geq 3$
and $\dim \mathcal{M}(\alpha,\gamma)\geq 5$.}

\noindent Set $\beta=\gamma+\alpha$ and $\delta=\gamma+2\alpha$.
Since $L_\gamma$ is a spherical vector of 
$\mathcal{M}(\alpha,\gamma)$, it follows that
$\gamma$, $\beta$ and $\delta$ belongs to $\Supp \mathcal{M}(\alpha,\gamma)$.
Similarly, $-\gamma$ and $-\beta$ belong to
$\Supp{\cal M}(\alpha,-\gamma)$.

It follows from Lemma \ref{lemma_17} that 
$L_{-\beta}$ is not in the
kernel of $B$. Thus $\beta$ belongs to $\Sigma_1$. 
Since $\gamma+\delta=2\beta$ and $l(\gamma)+l(\delta)=2$,
the element $[L_\gamma,L_\delta]$ is homogeneous
of degree $(2\beta,2)$ relative to the $\Lambda_e$-gradation.
Thus $[L_\gamma,L_\delta]$ belongs to
$\mathcal{M}(\beta,0)$. By Lemma \ref{lemma_18}, we have
$\mathcal{M}(\beta,0)=\mathfrak{s}(\beta)$,
and therefore

\centerline{ $[L_{\gamma},L_{\delta}]=0$.}

Similarly, $[L_{-\gamma},L_{\delta}]$ is homogeneous
of degree $(2\alpha,2)$ relative to the $\Lambda_e$-gradation.
Thus it belongs to  $\mathcal{M}(\alpha,0)$.
Since $\mathcal{M}(\alpha,0)=\mathfrak{s}(\alpha)$, it follows that

\centerline{$[L_{-\gamma},L_{\delta}]=0$.}

 However, these two relations $[L_{\pm\gamma},L_{\delta}]=0$ are
impossible.  Indeed it follows  from Lemma \ref{lemma_17} that 
$L_{-\gamma}$ is not in the
kernel of $B$. Thus $[L_{-\gamma},L_\gamma]=cL_0$ for some $c\neq 0$,
and thus  $[[L_{-\gamma},L_\gamma],L_\delta]=2cL_\delta\neq 0$,
which is a contradiction. 
\end{proof}

\begin{lemma}\label{lemma_20} Let  $\alpha\in \Sigma_1$. Then we have:
\begin{enumerate}
\item[(i)] $[\mathcal{L}_{-\alpha},\mathcal{K}^i]\subset \mathcal{K}^{i-1}$ for any $i>1$,
\item[(ii)] $[\mathcal{L}_{\alpha},\mathcal{K}^i]\subset \mathcal{K}^{i+1}$ for any $i\geq
1$.
\end{enumerate}
\end{lemma}
\begin{proof}
For any $\gamma\in\Lambda$ and any $i\in \Z$, set
$\mathcal{K}(\gamma)=\mathcal{M}(\alpha,\gamma)\cap \mathcal{K}$ and
$\mathcal{K}^i(\gamma)=\mathcal{K}^i\cap \mathcal{K}(\gamma)$.
Since $\mathcal{K}$ is a graded subspace of $\mathcal{L}$,
we have $\mathcal{K}=\oplus_\gamma\,\mathcal{K}(\gamma)$. Therefore, it is 
enough to prove for any $\gamma\in\Lambda$ the following assertion:

($\mathcal{A}$) \hskip3mm $[\mathcal{L}_{-\alpha},\mathcal{K}^{i+1}(\gamma)]\subset  \mathcal{K}^i(\gamma)$ and  
$[\mathcal{L}_{\alpha},\mathcal{K}^i(\gamma)]\subset \mathcal{K}^{i+1}(\gamma)$, 
for any $i\geq 1$

\noindent Four cases are required to check the assertion.

\textit{First case:} Assume 
$[\mathcal{M}(\alpha,\gamma),\mathcal{M}(\alpha,-\gamma)]=0$.

\noindent  In such a case,
$\mathcal{K}(\gamma)=\mathcal{M}(\alpha,\gamma)$ is a $\mathfrak{s}(\alpha)$-module, and Assertion ($\mathcal{A}$) is obvious.

From now on, it can be assumed that

\centerline{$[\mathcal{M}(\alpha,\gamma),\mathcal{M}(\alpha,-\gamma)]\neq 0$}

\textit{Second  case:} Assume 
$\dim \mathcal{M}(\alpha,\gamma)=
\dim \mathcal{M}(\alpha,-\gamma)$.

\noindent In such a case, it follows from
Lemma \ref{lemma_17} that $\mathcal{K}(\gamma)=\mathcal{K}^0(\gamma)$, and Assertion
($\mathcal{A}$) is clear.

\textit{Third case:} Assume  
 $\dim \mathcal{M}(\alpha,\gamma)< \dim \mathcal{M}(\alpha,-\gamma)$. 

\noindent It follows from Lemma \ref{lemma_17}
that $\mathcal{K}(\gamma)=0$ and Assertion ($\mathcal{A}$) is obvious.

\textit{Fourth case:} Assume  
$\dim \mathcal{M}(\alpha,\gamma)> \dim \mathcal{M}(\alpha,-\gamma)$. 

\noindent It follows
from Lemma \ref{lemma_19} that $\dim \mathcal{M}(\alpha,\gamma)=3$ and
that $\dim \mathcal{M}(\alpha,-\gamma) =1$. Thus
$\mathcal{K}(\gamma)=\mathcal{K}^1(\gamma)\oplus \mathcal{K}^{-1}(\gamma)$.
Thus Assertion ($\mathcal{A}$) follows from the fact that $1$ is the highest
eigenvalue of $L_0$ on $\mathcal{M}(\alpha,\gamma)$ and therefore
$[\mathcal{L}_\alpha, \mathcal{K}^1(\gamma)]=0$. 
\end{proof}

\begin{lemma}\label{lemma_21} We have:
\begin{enumerate}
\item[(i)] $[\mathcal{L}^{-1},\mathcal{K}^i]\subset \mathcal{K}^{i-1}$ for any $i>1$,
\item[(ii)] $[\mathcal{L}^1,\mathcal{K}^i]\subset \mathcal{K}^{i+1}$ for any $i\geq 1$.
\end{enumerate}
\end{lemma}

\begin{proof}
We have
$\mathcal{L}^{-1}=\mathcal{K}^{-1}\oplus[\oplus_{\alpha\in\Sigma_{1}}
\,\mathcal{L}_{-\alpha}]$ and 
$\mathcal{L}^1=\mathcal{K}^1\oplus[\oplus_{\alpha\in\Sigma_1}\,\mathcal{L}_\alpha]$. Therefore the lemma follows from the previous lemma and
the fact that $\mathcal{K}$ is a Lie subalgebra.
\end{proof}

\section{Types of integrable Lie algebras in the class $\mathcal{G}'$}\label{sect_6}

In this section, it is proved that an integrable Lie algebra
$\mathcal{L}\in \mathcal{G}'$ is of type $1$ or $2$. In the terminology
of root graded Lie algebras, it corresponds with $A_1$ and $BC_1$ 
Lie algebras, see in particular \cite{BZ}.

\begin{lemma}\label{lemma_22} We have $\Sigma_1\neq\emptyset$ and
$\Sigma_i=\emptyset$ for any $i>2$.
\end{lemma}
\begin{proof}
The fact that $\Sigma_1\neq\emptyset$
follows from Lemma \ref{lemma_16}.

Let $\beta\in \Sigma_i$ with $i>0$. Fix $\alpha\in\Sigma_1$.
By Lemma \ref{lemma_15}, we have
$1=l(\alpha)\in \Z.l(\beta)/2$. It follows that
$i$ divides $2$, and therefore
$i= 1$ or $2$. 
\end{proof}
In what follows, it will be convenient to set
$\mathcal{L}^{> 0}=\oplus_{i> 0}\,\mathcal{L}^i$. Similarly
define $\mathcal{L}^{\geq 0}$ and
$\mathcal{L}^{< 0}$.

\begin{lemma}\label{lemma_23} Let $\mathcal{L}\in \mathcal{G}'$ be integrable of type $N$.
\begin{enumerate}
\item[(i)]  The Lie algebra $\mathcal{L}^{>0}$ (respectively
$\mathcal{L}^{<0}$) is generated by 
$\mathcal{L}^{1}$ (respectively by $\mathcal{L}^{-1}$).
\item[(ii)] The commutant of $\mathcal{L}^1$ is $\mathcal{L}^N$.
\item[(iii)] The $\mathcal{L}^0$-module $\mathcal{L}^N$ is simple
$\Lambda$-graded.
\item[(iv)] $\mathcal{L}^k=[\mathcal{L}^{-1},\mathcal{L}^{k+1}]$, and
$\mathcal{L}^{k+1}=[\mathcal{L}^{1},\mathcal{L}^{k}]$,
for any $k\in[-N,N-1]$.
\end{enumerate}
\end{lemma}

\begin{proof}
First prove  Assertion (i).  Let $\alpha\in\Sigma_1$. Since 
$\mathcal{L}$ is a direct sum of finite dimensional spherical
$\mathfrak{s}(\alpha)$-modules, we have
$\mathcal{L}^{k+1}=\ad^k(\mathcal{L}_\alpha)(\mathcal{L}^1)$
for any $k>0$. Since $\mathcal{L}_\alpha\subset \mathcal{L}^1$,
the Lie algebra $\mathcal{L}^{>0}$ is generated by $\mathcal{L}^1$.
Similarly, the Lie algebra $\mathcal{L}^{<0}$ is generated by 
$\mathcal{L}^{-1}$.
Assertion (i) is proved.

Now prove  Assertion (ii). Let $\mathcal{M}=\oplus \mathcal{M}^i$ be
the commutant of $\mathcal{L}^1$, where  $\mathcal{M}^i
=\mathcal{L}^i\cap \mathcal{M}$.

Let $i<N$ be any integer, and set
$\mathcal{I}=\Ad (U(\mathcal{L}))(\mathcal{M}^i)$.
By Point (i), $\mathcal{L}^{\geq 0}$ is generated by
$\mathcal{L}^1$ and $\mathcal{L}^0$, thus $\mathcal{M}^i$ is
a $\mathcal{L}^{\geq 0}$-submodule. Hence by
PBW theorem we get
$\mathcal{I}=\Ad (U(\mathcal{L}^{<0}))(\mathcal{M}^i)$. So any homogeneous
component  of $\mathcal{I}$ has degree $\leq i$ and 
$\mathcal{I}$ intersects
trivially $\mathcal{L}^N$. 
Therefore $\mathcal{I}$ is a proper ideal. Since $\mathcal{I}$ is
clearly $\Lambda$-graded, it follows that
$\mathcal{I}=0$ 
which implies that $\mathcal{M}^i=0$. Therefore
$\mathcal{M}$ is homogeneous of degree $N$. Since
$\mathcal{L}^N$ is obviously in the commutant of 
$\mathcal{L}^1$, Assertion (ii) is proved.

The proof of Assertion (iii) is similar. Choose any
$\Lambda$-graded $\mathcal{L}^0$-sub\-mo\-du\-le $\mathcal{N}$ in 
$\mathcal{L}^N$ and set $\mathcal{I}=\Ad (U(\mathcal{L}))(\mathcal{N})$.
By the same argument, it follows that $\mathcal{I}^N=\mathcal{N}$.
Since the graded ideal $\mathcal{I}$ is not zero,
it follows that $\mathcal{N}=\mathcal{L}^N$, i.e. the
$\mathcal{L}^0$-module $\mathcal{L}^N$ is simple
$\Lambda$-graded, which proves Assertion (iii).

It follows from the previous 
considerations that 
$\mathcal{L}=\Ad(U(\mathcal{L}^{<0}))(\mathcal{L}^N)$. 
Since $\mathcal{L}^{<0}$ is generated by
$\mathcal{L}^{-1}$,  we have

\centerline{$\mathcal{L}=\oplus_{k\geq 0}\,\ad^k(\mathcal{L}^{-1})(\mathcal{L}^N)$.}

\noindent Similarly, we have

\centerline{$\mathcal{L}=\oplus_{k\geq 0}\,\ad^k(\mathcal{L}^{1})(\mathcal{L}^{-N})$.}

 Assertion (iv) follows from these two identities. 
\end{proof}

\begin{lemma}\label{lemma_24} One of the following assertions holds:
\begin{enumerate}
\item[(i)] $\mathcal{L}$ is of type $1$, or
\item[(ii)] $\mathcal{L}$ is of type $2$, and $\Sigma_2\neq\emptyset$.
\end{enumerate}
\end{lemma}
\begin{proof}
Let $N$ be the type of $\mathcal{L}$. Also
set $M=2$ if $\Sigma_2\neq\emptyset$ and set
$M=1$ otherwise. Thus Lemma \ref{lemma_24} is equivalent to $N=M$.

Assume otherwise, i.e.
$M<N$. It follows from the hypothesis that: 

\centerline{  $\mathcal{L}^{M+1}=\mathcal{K}^{M+1}$
and $\mathcal{L}^{M}\neq \mathcal{K}^{M}$.}

However by Lemma \ref{lemma_23}, we get
$\mathcal{L}^M=[\mathcal{L}^{-1},\mathcal{L}^{M+1}]$, and so
$\mathcal{L}^M=[\mathcal{L}^{-1},\mathcal{K}^{M+1}]$. From Lemma \ref{lemma_21},
we have $[\mathcal{L}^{-1},\mathcal{K}^{M+1}]\subset \mathcal{K}^M$. So 
we obtain $\mathcal{L}^M\subset \mathcal{K}^M$ which contradicts
$\mathcal{L}^{M}\neq \mathcal{K}^{M}$.
\end{proof}

\begin{lemma}\label{lemma_25} (Main Lemma for type $2$ integrable Lie algebras)

Let $\mathcal{L}\in \mathcal{G}'$ be integrable of type $2$.
Then any $\beta\in\Supp \mathcal{L}$ with $l(\beta)\neq 0$
belongs to $\Sigma$. 
\end{lemma}
\begin{proof}
\textit{First step:}
By Lemma \ref{lemma_23}, we get
$\mathcal{L}^0=[\mathcal{L}^1,\mathcal{L}^{-1}]$. Since 
$\mathcal{L}^3=0$, we deduce that
$[\mathcal{L}^0,\mathcal{K}^2]=
[\mathcal{L}^{1},[\mathcal{L}^{-1},\mathcal{K}^2]]$.
It follows from Lemma \ref{lemma_21} that

\centerline{$
[\mathcal{L}^{1},[\mathcal{L}^{-1},\mathcal{K}^2]]
\subset [\mathcal{L}^{1},\mathcal{K}^1]
\subset \mathcal{K}^2$}

\noindent Hence $\mathcal{K}^2$ is an $\mathcal{L}^0$-submodule
of $\mathcal{L}^2$. By the previous lemma, we have $\Sigma_2\neq\emptyset$,
so $\mathcal{K}^2$ is a proper submodule of
$\mathcal{L}^2$. It follows from Lemma \ref{lemma_23} (iii) that
$\mathcal{K}^2=0$. Equivalently, any $\beta\in \Supp \mathcal{L}$
with $l(\beta)=2$ belongs to $\Sigma$.

\textit{Second step:} By Lemma \ref{lemma_21}, we have 
$\ad(\mathcal{L}^{1})(\mathcal{K}^1)\subset \mathcal{K}^2$. Since 
$\mathcal{K}^2=0$, it follows that $\mathcal{K}^1$ lies inside
the commutant of $\mathcal{L}^1$. However by Lemma \ref{lemma_23} (ii), this commutant
is $\mathcal{L}^2$. Hence we have $\mathcal{K}^1=0$, or,
equivalently, any $\beta\in \Supp \mathcal{L}$
with $l(\beta)=1$ belongs to $\Sigma$.

\textit{Final step:} Similarly, it is proved that any 
$\beta\in \Supp \mathcal{L}$
with $l(\beta)=-1$ or $l(\beta)=-2$ belongs to $\Sigma$.
Since $l$ takes value in $[-2,2]$, the Lemma is proved.
\end{proof}

The Main Lemma also holds  for type $1$ integrable Lie algebras.
Unfortunately, the proof requires more computations. The simplest 
approach, based on Jordan algebras, is developed in the next two
sections. Another approach (not described in the paper) is
based on tensor products of Chari-Pressley loop modules \cite{C}, 
\cite{CP}  for $A^{(1)}_1$.

\section{Jordan algebras of the class $\mathcal{J}$}\label{sect_7}

Let $J$ be a Jordan algebra. In what follows, it is assumed that any
Jordan algebra is unitary.  For any
$X\in J$, denote by $M_X\in\End(J)$ the multiplication by $X$ and set 
$[X,Y]=[M_X,M_Y]$. It follows from the Jordan identity
that $[X,Y]$ is a derivation of $J$, see \cite{J}. Any linear combinations
of such expressions is called an \textit{inner derivation}, and the space 
of inner derivations of $J$ is denoted by $\Inn J$.

\begin{lemma}\label{lemma_26}  Let $X,Y\in J$.
\begin{enumerate}
\item[(i)] $[X^2,Y]=2[X,XY]$,
\item[(ii)] $X^2Y^2 + 2 (XY)(XY)= Y(X^2Y)+2 X(Y(YX))$.
\end{enumerate}
\end{lemma}
\begin{proof}
To obtain both identities \cite{J}, apply $\left. d\dfrac{d}{dt}\right\vert_{t=0}$
to the Jordan identities:
\begin{align*}
&[(X+tY)^2, X+tY]=0, \\
&(X+tY)^2 (Y(X+tY))=((X+tY)^2 Y)(X+tY).
\end{align*}
\end{proof}

Let $Q$ be a finitely generated abelian group,
let $J=\oplus_\lambda\, J_\lambda$ be a $Q$-graded Jordan
algebra, and set $\mathcal{D}=\Inn J$. There is a decomposition
$\mathcal{D}=\oplus_\lambda\,\mathcal{D}_\lambda$, where 
$\mathcal{D}_{\lambda}$ is the space of  linear combinations
$[X,Y]$, where $X\in J_\mu$, $Y\in J_\nu$ for some
$\mu,\,\nu\in Q$ with $\mu+\nu=\lambda$. Relative to this
decomposition,  
$\mathcal{D}_\lambda$ is a weakly $Q$-graded Lie algebra.

   The Jordan algebra $J$ is called of \textit{class $\mathcal{J}(Q)$} if 
the following requirements hold:
\begin{enumerate}
\item[(i)]  $J$ is a simple $Q$-graded Jordan algebra, 
\item[(ii)] $\dim J_\lambda+\dim \mathcal{D}_\lambda\leq 1~\forall \, \lambda\in Q$, and
\item[(iii)] $\Supp J$ generates $Q$.
\end{enumerate}
If $\pi:Q\rightarrow \overline{Q}$ is a surjective morphism of
abelian groups and if $\overline{J}$ is a $\overline{Q}$-graded 
Jordan algebra, then $\overline{J}$ is in the class $\mathcal{J}(\overline{Q})$ if and only if $\pi^\ast\overline{J}$ is in $\mathcal{J}(Q)$.
This follows from the functoriality  of inner derivations.
For now on, fix an abelian group $Q$ and set 
$\mathcal{J}=\mathcal{J}(Q)$.

An element $X$ of a Jordan algebra $J$ is called 
\textit{strongly invertible} if  $XY=1$ and $[X,Y]=0$, for
some $Y\in J$. Obviously, a strongly invertible element 
is invertible in the sense of \cite{J}, Definition 5.

An element $Z\in J$ is
called \textit{central} iff $[Z,X]=0$ for any $X\in J$. The subalgebra of
central  elements is called the center of $J$. Since $J$ is unitary,
the map $\psi\mapsto \psi(1)$ identifies the centroid  of $J$ with its
center. Therefore both algebras will be denoted by $C(J)$.

In  a Jordan algebra of the class $\mathcal{J}$, we have
$ab=0$ or $[a,b]=0$ for any two homogeneous elements $a$, $b$.
For example, if we have
$XY=1$ for two homogeneous elements $X$ and $Y$, then $X$ is
strongly invertible.

\begin{lemma}\label{lemma_27} Let $J\in \mathcal{J}$ be a Jordan algebra, and let
$X\in J$ be a homogeneous and strongly invertible element. Then
$X^2$ is central.
\end{lemma}
\begin{proof}
In order to show that $[X^2,Z]=0$ for all $Z$, it can be assumed
that $Z$ is a non-zero homogeneous element.

First assume that $XZ=0$. In such a case, we have

\centerline{$[X^2,Z]=2[X,XZ]=0$,}

\noindent by Point (i) of the previous lemma.

Assume otherwise. By
hypothesis, there exists a homogeneous element $Y\in J$ such
that $XY=1$. Then we get:

$Y(XZ)=[Y,X].Z + X(YZ)=[Y,X].Z + [X,Z].Y+ Z(XY)$.

\noindent We have $[X,Y]=0$ and since $XZ\neq 0$, we also have $[X,Z]=0$.
Therefore $Y(XZ)=Z(XY)=Z$. Since $M_X$ and $M_Y$ commutes, we have
$Y(Y(X(XZ)))=Y(X(Y(XZ)))=Y(XZ)=Z$ and therefore
$X(XZ)\neq 0$. Since $[X^2,Z]$ and $X(XZ)$ are homogeneous of the same
degree, we have $[X^2,Z]=0$ by Condition (ii) of the definition
of the class $\mathcal{J}$. 
\end{proof}

\begin{lemma}\label{lemma_28} Let $J$ be a Jordan algebra in the class $\mathcal{J}$. Then 

\centerline{$\mathcal{D}.J\cap J_0 =\{0\}$.}
\end{lemma}
\begin{proof}
Assume otherwise. Then there are three homogeneous
elements $X,\,Y,\,Z\in J$ such that
$[X,Y].Z=1$.   

It can be assumed that $J$ is generated by $X,\, Y$ and
$Z$. Indeed let $J'$ be the Jordan algebra generated by
$X,\, Y$ and $Z$. Since $J'$ is
unitary, it contains a unique maximal graded ideal $M$.
Then the same relation  $[X,Y].Z=1$ holds in $J/M$ and
$J/M$ is again in the class $\mathcal{J}$.

Also, it can be assumed that $C(J)=\C$. 
Indeed, by Lemma \ref{lemma_9}, there exists an
abelian group
$\overline{Q}$, a surjective morphism $\pi:Q\rightarrow \overline{Q}$
and a $\overline{Q}$ graded simple Jordan algebra $\overline{J}$ such
that $J=\pi^\ast\,\overline{J}$. It has been noted that
$\overline{J}$ is in $\mathcal{J}(\overline{Q})$ and by simplicity
its centroid $C(\overline{J})$ is reduced to $\C$.
Since the same relation $[X,Y].Z=1$ holds in $\overline{J}$,
it can be assumed that $C(J)=\C$.

Let $\alpha,\,\beta,\,\gamma$ be the degree of
$X$, $Y$ and $Z$. Note that $\alpha+\beta+\gamma=0$.

We have $X(YZ)-Y(XZ)=1$. By symmetry, one can assume that
$X(YZ)\neq 0$.

We have $X(YZ)=[X,Z].Y + Z(XY)$. Since $[X,Y]$ is not zero,
we have $XY=0$ and $0\neq X(YZ)=[X,Z].Y$. Thus we have
$[X,Z]\neq 0$ and $XZ=0$.

Thus $X(YZ)=1$. It follows from Lemma \ref{lemma_27} and the hypothesis 
$C(J)=\C$ that
$X^2$ is a non-zero scalar $c$, and therefore $2\alpha=0$ and
$YZ=X/c$.

 By Lemma \ref{lemma_26} (ii),  there is the identity

\centerline{$Y^2 Z^2 + 2 (YZ)^2=Z(Y^2 Z)+ 2 Y(Z(ZY))$.}

\noindent We have

\centerline{$\deg Y^2 Z=2\beta+\gamma=\beta-\alpha=\beta+\alpha=\deg
[X,Y]$, and} 

\centerline{$\deg Z(ZY)=\beta+2\gamma=\gamma-\alpha=\gamma+\alpha=\deg
[X,Z]$.}

By hypothesis we have  $[X,Y]\neq 0$ and it has been proved that
$[X,Z]\neq 0$. Therefore $Y^2 Z=0$ and $Z(ZY)=0$ and
the right side of the identity is zero.
Thus we get $Y^2 Z^2+ 2(YZ)^2=0$. Since $YZ=X/c$, it follows that
$Y^2 Z^2=-2/c$. Hence $Y^2$ is strongly invertible. By Lemma \ref{lemma_27},
$Y^4$ is a non-zero scalar, and therefore $4\beta=0$.

Thus the subgroup generated by $\alpha$ and $\beta$ has order
$\leq 8$, and it contains $\gamma=-(\alpha+\beta)$. Therefore $\Supp J$ is
finite, i.e. $\dim J<\infty$.

It follows from Jordan identity that
$M_{[X,Y].Z}=[[M_X,M_Y],M_Z]$, see \cite{J}, Formula (54). Set
$P=[M_X,M_Y]$ and 
$Q=M_Z$. Thus $P$ and $Q$ belongs to the associative
algebra $\End (J)$, and they satisfy
$[P,Q]=1$, where $[,]$ denotes the ordinary Lie bracket.
Since $J$ is finite dimensional, this identity is impossible.
\end{proof}

\begin{lemma}\label{lemma_29} Let $J$ be a simple graded Jordan algebra of 
the class $\mathcal{J}$. Then any non-zero homogeneous element is 
strongly invertible.
\end{lemma}
\begin{proof} Let $\lambda:J\rightarrow \C$ be defined by
$\lambda(1)=1$ and $\lambda(X)=0$ if $X$ is homogeneous of degree $\neq 0$.

It follows from the previous lemma that
$\lambda([X,Y].Z)=0$, or, equi\-va\-lently
$\lambda((XY)Z)=\lambda(X(YZ))$ for all $X,\,Y,\,Z\in J$. Thus the kernel
of the bilinear map $B:J\times J\rightarrow \C, 
(X,Y)\mapsto \lambda(XY)$ is a graded ideal. Hence
$B$ is non-degenerated, which implies that any non-zero homogeneous element
is strongly invertible.  
\end{proof}
Recall that the underlying grading group is denoted by $Q$.

\begin{lemma}\label{lemma_30} Let $J$ be a  Jordan algebra of the class
$\mathcal{J}$. Then 

\centerline{$\Supp\, C(J) \supset 2Q$.}
\end{lemma}
\begin{proof}
Let $\alpha\in \Supp J$, and let $X\in
J_\alpha\setminus\{0\}$. By Lemma \ref{lemma_29}, $X$ is 
strongly invertible and by Lemma \ref{lemma_27},
$X^2$ is central. Therefore
$2\alpha$ belongs to $\Supp\, C(J)$. Since $\Supp J$ generates $Q$  
and $\Supp\, C(J)$ is a subgroup, it follows
that $\Supp\, C(J)$ contains $2Q$.
\end{proof}

\section{Kantor-Koecher-Tits  construction}\label{sect_8}

Let us recall the Kantor-Koecher-Tits  construction,
which first appeared in \cite{T}, and then in 
\cite{Kan}, \cite{Ko}. It contains two parts.

To any Jordan algebra
$J$  is associated a Lie algebra, denoted by 
$\mathfrak{sl}(2,J)$, whose definition is as follows.
As a vector space, $\mathfrak{sl}(2,J)=\mathfrak{sl}(2)\otimes J
\oplus \Inn J$, and the bracket is defined by the following
requirements:
\begin{enumerate}
\item[(i)] its restriction to $\Inn J$ is the Lie algebra structure
of $\Inn J$,
\item[(ii)] we have: $[\delta,x\otimes a]=x\otimes \delta.a$, 
\item[(iii)] we have:
$[x\otimes a,y\otimes b]= [x,y]\otimes ab + k(x,y) [a,b]$,
for any
$\delta\in\Inn J$,  $x,\,y\in\mathfrak{sl}(2)$, 
$a,\,b\in J$, where $k(x,y)=1/2 \,\tr (xy)$ and where $[a,b]$ denotes
the inner derivation $[M_a,M_b]$.
\end{enumerate}
Conversely, to  certain Lie algebras are associated with Jordan
algebras. Indeed assume that the Lie algebra
$\mathcal{L}$ contains a subalgebra $\mathfrak{sl}(2)$ with its
standard basis $e,f,h$ and moreover  that 

\centerline{$\mathcal{L}=\mathcal{L}^{-1}\oplus\mathcal{L}^0\oplus\mathcal{L}^1$,}

\noindent where $\mathcal{L}^i=\{x\in\mathcal{L}\vert\, [h,x]=2i x\}$.
Define an algebra $J$ by the following
requirements:
\begin{enumerate}
\item[(i)] as a vector space, $J=\mathcal{L}^1$,
\item[(ii)] the product of two elements $x$, $y$ is given by the formula
$xy=1/2 [[f,x],y]$.
\end{enumerate}

\begin{lemma}[Tits \cite{T}] \label{lemma_31} With the previous hypothesis:
\begin{enumerate}
\item[(i)] $J$ is a Jordan algebra.
\item[(ii)] If moreover $\mathcal{L}^0=[\mathcal{L}^1,\mathcal{L}^{-1}]$ and 
if the center of $\mathcal{L}$ is trivial, then we have
$\mathcal{L}=\mathfrak{sl}(2,J)$.
\end{enumerate}
\end{lemma}

An \textit{admissible datum} is a triple
$(J,\Lambda',\alpha)$ with the following conditions:
\begin{enumerate}
\item[(i)] $\Lambda'$ is a subgroup of $\Lambda$, $\alpha$ is an
element of $\Lambda$, and the group $\Lambda$ is generated 
by $\Lambda'$ and $\alpha$,
\item[(ii)] $J$ is a Jordan algebra in the class $\mathcal{J}(\Lambda')$,
\item[(iii)] the four subsets $\Supp J$, $\pm\alpha+\Supp J$ and $\Supp \Inn J$
are disjoint.
\end{enumerate}
Then a $\Lambda$-gradation of
the  Lie algebra $\mathfrak{sl}(2,J)$ is defined as follows:
\begin{enumerate}
\item[(i)] On $\Inn J$, the gradation is the natural $\Lambda'$ gradation
\item[(ii)] For any homogeneous element $x\in J$, set

\centerline{$\deg h\otimes x=\deg x$, $\deg e\otimes
x=\alpha+\deg x$ and
$\deg f\otimes x=-\alpha+\deg x$.}
\end{enumerate}
The condition (iii) of an admissible triple ensures that
the associated $\Lambda$-gradation of $\mathfrak{sl}(2,J)$ is mutiplicity free. 
Since the $\Lambda$-graded Lie algebra $\mathfrak{sl}(2,J)$ is clearly simple graded, it is a type $1$ integrable
Lie algebra of the class $\mathcal{G}'$.

Conversely, let $\Lambda$ be a finitely generated abelian
group, and let $\mathcal{L}\in \mathcal{G}'(\Lambda)$ be a type
$1$ integrable Lie algebra.
 Let  $\Lambda'$ be the subgroup generated by 
$\beta-\gamma$, where $\beta$ and $\gamma$ run over 
$\Supp\mathcal{L}^1$. 

By Lemma \ref{lemma_22}, the set $\Sigma_1$ is not empty. Choose $\alpha\in\Sigma_1$.
Let $J(\alpha)$ be the Jordan algebra defined by
the following requirements:
\begin{enumerate}
\item[(i)] As a vector space, $J(\alpha)=\mathcal{L}^1$,
\item[(ii)] the gradation of $J(\alpha)$ is given by

\centerline{$J(\alpha)_{\mu}=\mathcal{L}^1_{\mu+\alpha}$,}
\item[(iii)] the product $xy$ of two elements $x,y\in J(\alpha)$ is defined by

\centerline{$xy=[[L_{-\alpha},x],y]$.}
\end{enumerate}

\begin{lemma}\label{lemma_32} Let $\mathcal{L}\in \mathcal{G}'$ be a type 1 
integrable Lie algebra, and let $\alpha\in
\Sigma_1$. 
\begin{enumerate}
\item[(i)] The triple $(J(\alpha),\Lambda',\alpha)$ is an admissible
datum, and
\item[(ii)] as a $\Lambda$-graded Lie algebra, $\mathcal{L}$ is
isomorphic to $\mathfrak{sl}(2,J(\alpha))$.
\end{enumerate}
\end{lemma}

\begin{proof} First check that
$\mathcal{L}$ satisfies the hypothesis of the previous lemma.
By simplicity of $\mathcal{L}$, its center is trivial. Moreover 
the identity  $\mathcal{L}^0=[\mathcal{L}^1,\mathcal{L}^{-1}]$
holds  by Lemma \ref{lemma_23} (iv). 

By the previous lemma, the Lie algebra
$\mathcal{L}$ is isomorphic to $\mathfrak{sl}(2,J(\alpha))$. 
It follows from
the definition that $\Supp J$ generates $\Lambda'$
and that $\Lambda=\Lambda'+\Z \alpha$. 
It is clear that $J$ is a simple graded Jordan algebra.
Since
the gradation of $\mathfrak{sl}(2,J(\alpha))$ is multiplicity
free, the four subsets $\Supp J(\alpha)$, $\pm\alpha+\Supp J(\alpha)$ and
$\Supp \Inn J(\alpha)$ are disjoint. In particular
$J(\alpha)$ is of the class ${\cal J}(\Lambda')$ and
 $(J(\alpha),\Lambda',\alpha)$ is an admissible datum.
\end{proof}

\begin{lemma}\label{lemma_33} (Main Lemma for type 1 integrable Lie algebras)

Let $\mathcal{L}\in \mathcal{G}'$ be integrable of type $1$.
Then any $\beta\in\Supp \mathcal{L}$ with $l(\beta)\neq 0$
belongs to $\Sigma$.
\end{lemma}

\begin{proof} Fix $\alpha\in \Sigma_1$. By the
previous lemma, $\mathcal{L}$ is
isomorphic to $\mathfrak{sl}(2,J(\alpha))$.

Since $l(\beta)\neq 0$, we have $l(\beta)=\pm 1$. Without
loss of generality, it can be assumed that $l(\beta)=1$.
Under the previous isomorphism, $L_\beta$ is identified
with $e\otimes X$, where $X$ is a homogeneous element
of $J(\alpha)$. By Lemma \ref{lemma_29}, $X$ is strongly
invertible. Let $Y$ be its inverse.  
Up to a scalar multiple, 
$L_{-\beta}$ is identified with $f\otimes Y$. Since
$[e\otimes X, f\otimes Y]=h$, it follows that
$[L_\beta,L_{-\beta}]\neq 0$, and so $\beta$ belongs to $\Sigma$.
\end{proof}

\section{Connection between the Centroid and the Weyl group}\label{sect_9}

In this section, 
$\mathcal{L}\in \mathcal{G}'(\Lambda)$ is an integrable simple graded
Lie algebra of type $N$, where $N=1$ or $N=2$. 
The Main Lemma is now established for
integrable Lie algebras of both types 1 and 2. 
The goal of the section is to prove that
$C(\mathcal{L})$ is very large, namely
$\mathcal{L}$ is finitely generated as 
a $C(\mathcal{L})$-module.

Recall the decomposition
$\mathcal{L}=\oplus_{i\in \Z}\,\mathcal{L}^i$, where
$\mathcal{L}^i=\Ker(\ad\, L_0-i)$. 

\begin{lemma}\label{lemma_34} There is a natural algebra isomorphism

\centerline{$C(\mathcal{L})\simeq \End_{\mathcal{L}^0}(\mathcal{L}^{N})$.}
\end{lemma}

\begin{proof} Any $\psi\in C(\mathcal{L})$ commutes with $L_0$,
and therefore $\psi$ stabilizes $\mathcal{L}^N$. This induces a natural
algebra morphism  $\theta:C(\mathcal{L})\rightarrow \End_{\mathcal{L}^0}(
\mathcal{L}^{N})$.  By Lemma \ref{lemma_23} (iv), 
$\mathcal{L}^{N}$ generates the adjoint module, so $\theta$ is injective.

Set
$V=\Ind_{\mathcal{L}^{\geq 0}}^\mathcal{L}\,\mathcal{L}^{N}$, 
where $\mathcal{L}^{\geq 0}=\oplus_{i\geq 0}\, \mathcal{L}^i$.
Since $\mathcal{L}^{N}$ generates the adjoint module,
the natural $\mathcal{L}$-equivariant map 
$\eta:V\rightarrow \mathcal{L}$
is onto. 

The $\mathcal{L}$-module $V$ is a weakly $\Lambda\times \Z$-graded 
$\mathcal{L}$-module. In particular,
there is a decomposition $V=\oplus_{i\leq N} V^i$, where 
$V^i=\{v\in V\vert L_0 v=iv\}$. Let $K$ be the biggest 
$\mathcal{L}$-module lying in $V^{<N}$,
where $V^{<N}=\oplus_{i<N}\,V^i$. It is clear that $K$ is
graded relative to the $\Lambda\times \Z$ gradation.
Since $V^N\simeq \mathcal{L}^N$ and $\mathcal{L}$ is simple graded,
it is clear that $K$ is precisely the kernel of $\eta$.

Since $\End_{\mathcal{L}^0}(\mathcal{L}^{N})=
\End_{\mathcal{L}^{\geq 0}}(\mathcal{L}^{N})$,  it follows that any $\psi\in
\End_{\mathcal{L}^0}(\mathcal{L}^{N})$ extends to a $\mathcal{L}$-endomorphism 
$\hat\psi$ of $V$. Moreover $\hat\psi$ stabilizes each $V^i$, therefore
it stabilizes $K$. Hence $\hat\psi$ determines  an 
$\mathcal{L}$-endomorphism
$\overline\psi$ of  $\mathcal{L}$. 
Since $\overline\psi\in \End_\mathcal{L}(\mathcal{L})=C(\mathcal{L})$ extends 
$\psi\in \End_{\mathcal{L}^0}(\mathcal{L}^{N})$,
the natural algebra morphism $\theta:C(\mathcal{L})\rightarrow \End_{\mathcal{L}^0}(\mathcal{L}^{N})$ is onto. 

Therefore $\theta$ is an isomorphism, which proves the lemma. 
\end{proof}

From now on, denote by $\Lambda'$ the subgroup of $\Lambda$
generated by $\alpha-\beta$ when $\alpha$, $\beta$ run over $\Sigma_1$.

\begin{lemma}\label{lemma_35} Let $\alpha\in\Sigma_1$.
We have $\Supp\mathcal{L}^i\subset i\alpha+\Lambda'$ for all $i$.
\end{lemma}
\begin{proof}
By the Main Lemmas \ref{lemma_25} and \ref{lemma_33}, we have 
$\Supp\mathcal{L}^1=\Sigma_1$, thus we have $\Supp\mathcal{L}^1\subset
\alpha+\Lambda'$. Similarly, we have
$\Supp\mathcal{L}^{-1}\subset -\alpha+\Lambda'$. By Assertion (iv) of
Lemma \ref{lemma_23}, we have $\mathcal{L}^0=[\mathcal{L}^{-1},\mathcal{L}^1]$,
$\mathcal{L}^2=[\mathcal{L}^{1},\mathcal{L}^1]$ and
$\mathcal{L}^{-2}=[\mathcal{L}^{-1},\mathcal{L}^{-1}]$. Thus the inclusion 
$\Supp\mathcal{L}^i\subset i\alpha+\Lambda'$  easily follows. 
\end{proof}

Let $\beta\in\Sigma_1$. Since $[L_\beta,L_{-\beta}]\neq 0$, it can be
assumed that $[L_\beta,L_{-\beta}]=2L_0$. Let $s_\beta$ be the
automorphism of the Lie algebra $\mathcal{L}$ defined by

\centerline{$s_\beta=\exp(-\ad(L_{\beta}))\circ\exp(\ad( L_{-\beta}))\circ
\exp(-\ad(L_{\beta}))$.}

\begin{lemma}\label{lemma_36} Let $\beta\in\Sigma_1$ and let
$\lambda\in\Supp\mathcal{L}$.
\begin{enumerate}
\item[(i)] We have $s_\beta\,\mathcal{L}_\lambda
\subset\mathcal{L}_{\lambda-2l(\lambda)\beta}$.
\item[(ii)] If $l(\lambda)=0$, then $s_\beta(L_\lambda)
=\pm L_\lambda$.
\end{enumerate}
\end{lemma}

\begin{proof} Set $\gamma=\lambda-l(\lambda)\beta$.
The action of 
$\mathfrak{s}(\beta):=\C L_{-\beta}
\oplus \C L_0\oplus \C L_\beta\simeq \mathfrak{sl}(2)$ on the 
spherical module $\mathcal{M}(\beta,\gamma)$ integrates to
an action of the group $PSL(2,\C)$, and
$s_\beta$ is the action of the group element
$\pm \begin{pmatrix} 0 & -1 \\ 1 & 0 \end{pmatrix}$, from which
both assertions follow. 
\end{proof}

For $\alpha,\beta\in \Sigma_1$, set
$t_{\alpha,\beta}= s_\alpha\circ s_\beta\circ s_\alpha\circ s_\beta$.

\begin{lemma}\label{lemma_37}
We have 
\begin{enumerate}
\item[(i)] $t_{\alpha,\beta}\,\mathcal{L}_\lambda\subset
\mathcal{L}_{\lambda +4l(\lambda)(\alpha-\beta)}$, for any
$\lambda\in \Supp\mathcal{L}$.
\item[(ii)] $t_{\alpha,\beta}(x)=x$ for any $x\in\mathcal{L}^0$.
\end{enumerate}
\end{lemma}
\begin{proof}
The first point follows from the previous lemma .
Let $\lambda\in\Supp\mathcal{L}^0$. By the previous lemma,
there are $\epsilon_\alpha,\,\epsilon_\beta\in\{\pm1\}$
such that $s_\alpha\,(L_\lambda)
=\epsilon_\alpha\, L_\lambda$ and $s_\beta\,(L_\lambda)
=\epsilon_\beta\, L_\lambda$. Therefore 
$t_{\alpha,\beta}\,(L_\lambda)=
\epsilon_\alpha^2\,\epsilon_\beta^2 \,L_\lambda=L_\lambda$.
Thus $t_{\alpha,\beta}$ acts trivially on 
$\mathcal{L}^0$.
\end{proof}

\begin{lemma}\label{lemma_38} Set $M=\Supp C(\mathcal{L})$. Then
we have $8\Lambda'\subset M\subset\Lambda'$.
\end{lemma}
\begin{proof}
By Lemma \ref{lemma_35}, the support of each $\mathcal{L}^i$ is contained in
one $\Lambda'$-coset. It follows easily that $M\subset\Lambda'$.

Let $\alpha,\,\beta\in \Sigma_1$, and let $t$ be
the restriction of $t_{\alpha,\beta}$ to $\mathcal{L}^N$. It follows from
the previous lemma that $t$  is an $\mathcal{L}^0$-morphism
of $\mathcal{L}^N$ of degree $4N(\alpha-\beta)$. By Lemma \ref{lemma_34}, 
there exists a morphism $\psi\in C(\mathcal{L})$ whose restriction to
$\mathcal{L}^N$ is $t$. Thus $4N(\alpha-\beta)$ belongs to $M$.
Since $N=1$ or $2$, it is always true that $8(\alpha-\beta)$ belongs to
$M$. Thus we get $8\Lambda'\subset M$. 
\end{proof}

\noindent{\textbf{Remark}.}
\textit{
If $\mathcal{L}$ is of type 1, there is a simpler
way to derive Lemma \ref{lemma_38}. Indeed, Lemma
\ref{lemma_30}  easily implies  that $M$ contains $2\Lambda'$.}

\section{Classification of integrable Lie algebras 
of the class $\mathcal{G}$}\label{sect_10}

Recall that $\Lambda$ is a lattice. We first state a classification result
for the integrable Lie algebras of the class $\mathcal{G}'$. \\

\noindent{\textbf{Theorem $\mathbf{2'}$}~
\textit{Let 
$\mathcal{L}$ be an integrable Lie algebra  in the class $\mathcal{G}'(\Lambda)$
and let $M$ be the support of $C(\mathcal{L})$. 
There exists a Lie algebra
$\g\in \mathcal{G}'(\Lambda/M)$ such that:
\begin{enumerate}
\item[(i)] $\g$ is a simple  finite dimensional Lie algebra,
\item[(ii)] $\mathcal{L}\simeq \pi^*\g$ as a $\Lambda$-graded Lie algebra, \\
where $\pi$ is the natural map $\Lambda\rightarrow\Lambda/M$.
\end{enumerate}
} 

\begin{proof}
 By Lemma \ref{lemma_9}, there exists
a  simple Lie algebra
$\g$,  and a  $\Lambda/M$ gradation 
$\g=\oplus_\mu\,\g_\mu$ of $\g$ such that
$\mathcal{L}\simeq \pi^\ast\g$ as a
$\Lambda$-graded algebra. Since 
$\mathcal{L}\in\mathcal{G}'(\Lambda)$, it is clear that 
$\g$ belongs to $\mathcal{G}'(\Lambda/M)$.

It remains to prove that $\g$ is finite dimensional.
By Lemma \ref{lemma_35}, $\Supp\mathcal{L}$ lies in a most five $\Lambda'$-cosets. Therefore, $\Supp \g$ lies in at most five
$\Lambda'/M$-cosets. By Lemma \ref{lemma_38}, $\Lambda'/M$ is finite. 
Therefore $\Supp \g$ is finite, which implies that
$\g$ is finite dimensional. 
\end{proof}

By the previous theorem, the classification of all
integrable Lie algebras in the class $\mathcal{G}'$ follows from the
classification of finite dimensional simple Lie algebras
of the class $\mathcal{G}'$. 

For the class $\mathcal{G}$, the
classification will be  explicit.
Let $\mathcal{L}=\pi^\ast\g$ as in Theorem 2'.

\begin{lemma}\label{lemma_39} Assume moreover that $\mathcal{L}$ belongs to
the class $\mathcal{G}$. Then we have:
\begin{enumerate}
\item[(i)] $\dim\g_\mu= 1$ for any $\mu\in\Lambda/M$,
\item[(ii)] $\dim\g=a2^n$ for some $a\in\{1,\,3,\,5\}$ and some $n\geq 0$.
\end{enumerate}
\end{lemma}
\begin{proof}
It is obvious that $\dim\g_\mu= 1$ for any $\mu\in\Lambda/M$.
Since $\Lambda=\Supp\mathcal{L}$, it follows from Lemma \ref{lemma_35} that
$\Lambda$ is an union of at most five $\Lambda'$-cosets, therefore
the index $[\Lambda:\Lambda']$ is $\leq 5$. By Lemma \ref{lemma_38},
$M$ contains $8\Lambda'$, so $[\Lambda':M]$ is a power of
$2$. Thus the index  $[\Lambda:M]$ can be written as
$a'2^{n'}$, where $a'\leq 5$. Thus this index is of the form $a 2^n$,
with $a=1$, $3$ or $5$. Therefore

\centerline{$\dim\g=[\Lambda:M]=a2^n$,}

\noindent  with $a=1$, $3$ or $5$. 
\end{proof}

\begin{lemma}\label{lemma_40} Let $\g$ be a finite dimensional simple
Lie algebra of dimension $a2^n$ for some $a\in\{1,\,3,\,5\}$ and some
$n\geq 0$. Then $\g$ is of type $A$, or $\g$ is isomorphic
to $\mathfrak{sp}(4)\simeq \mathfrak{so}(5)$.
\end{lemma}
\noindent{\textbf{Remark.}}~
\textit{
More precisely, a type $A$ Lie algebra  
of dimension $a2^n$ for some $a\in\{1,\,3,\,5\}$ and some $n\geq 0$ is
isomorphic to $\mathfrak{sl}(l)$ for $l=2,\,3,\,4,\,5,\,7$ or $9$.
However this remark is not essential.}
\begin{proof}
Assume that $\g$ is not of type $A$.

The dimension of the exceptional Lie algebras
$G_2$, $F_4$, $E_6$, $E_7$ and $E_8$ are
respectively $14=2.7$, $52=4.13$, $78=6.13$, $133=7.19$ and
$248=8.31$, and therefore $\g$ is not exceptional.

Thus $\g$ is of type $B$, $C$ or $D$, and its dimension is 
 $m(m+1)/2$ for some integer $m$. 
Consider the equation
$m(m+1)/2=a2^n$, for some $n\geq 0$ and some  $a\in\{1,\,3,\,5\}$.
Obviously $m=1$ is a solution. For $m>1$, $m$ or $m+1$ is a odd factor of
$m(m+1)/2$, and therefore $m$ or $m+1$ should be $3$ or $5$. 
The case $m=5$ being not a solution, the only solutions are: $m=1, m=2,
m=3$ and $m=4$.  

However the following values
should be excluded:
\begin{enumerate}
\item[(i)] $m=1$ because $\mathfrak{so}(2)$ is abelian,
\item[(ii)] $m=2$ because $\mathfrak{so}(3)\simeq \mathfrak{sp}(2)$
is of type $A$,
\item[(iii)] $m=3$ because $\mathfrak{so}(4)$ is not simple.
\end{enumerate}
The remaining case $m=4$ is precisely the dimension  of
$\mathfrak{sp}(4)\simeq \mathfrak{so}(5)$. 
\end{proof}
Let $\g$ be a simple Lie algebra and let $F$ be an abelian group. A 
gradation  $\g=\oplus_{\alpha\in F}\,\g_{\alpha}$ of $\g$ is called
\textit{simple} if $\dim \g_\alpha=1$ for any $\alpha\in F$. This implies
that $F$ is finite, of order $\dim\g$.

\begin{lemma}\label{lemma_41} 
\begin{enumerate}
\item[(i)] The Lie algebra $\mathfrak{sp}(4)\simeq \mathfrak{so}(5)$
does not admit a simple gradation.
\item[(ii)] The Lie algebra $\mathfrak{sl}(n)$ does not admit admit a simple
gradation  for $n>3$.
\end{enumerate}
\end{lemma}
\begin{proof}
\textit{Point (i):}
Set  $\g=\mathfrak{sp}(4)$ and let
$\g=\oplus_{\alpha\in F}\,\g_{\alpha}$ be a simple gradation of 
$\g$. Since $\dim\g=10$, the group $F$ is isomorphic to 
$\Z/2\Z\times \Z/5\Z$. This gradation induces a 
$\Z/2\Z$-gradation $\g=\g_0\oplus \g_1$, where each component
$\g_i$ has dimension $5$. Since $\g_0$ is reductive of dimension
$5$, it is isomorphic with $\mathfrak{sl}(2)\oplus \C^2$
or $\C^5$ and its rank is $\geq 3$. Since the rank of
$\mathfrak{sp}(4)$ is two, this is impossible.

\textit{Point (ii):}
Set  $\g=\mathfrak{sl}(n)$, let
$\g=\oplus_{\alpha\in F}\,\g_{\alpha}$ be a simple gradation of 
$\g$ by a finite abelian group $F$. Let $X=\Hom(F,\C^\ast)$ be its
character group. The gradation induces a natural action of $X$ on $\g$: an
element $\chi\in X$ acts on $\g_{\alpha}$ by multiplication by 
$\chi(\alpha)$. 

Let $\rho:X\rightarrow \Aut(\g)$ be the corresponding morphism.
Since the adjoint group $PSL(n)$ has index $\leq 2$
in $\Aut(\g)$, there is a subgroup 
$Y$ of $X$ of index $\leq 2$ such that 
$\rho(Y)\subset PSL(n)$. Let $\psi:Y\rightarrow
PSL(n)$ be the corresponding morphism. 
The group $F$ has
order  $dim\g=n^2-1$, hence the order of $Y$ is prime to $n$.
Thus the map $\psi$ can be lifted to a morphism
$\hat{ \psi}: Y\rightarrow SL(n)$. 

Let $K$ be the commutant of $\hat\psi(Y)$ and let
$\mathfrak{k}$ be its Lie algebra. Let
$F'=\{\alpha\in F\vert\,\chi(\alpha)=1\,\forall
\chi\in Y\}$.  Since $[X:Y]\leq 2$, the group
$F'$ has at most two elements. 
It is clear that $\mathfrak{k}$ is the subalgebra of fixed points
under $Y$, hence $\mathfrak{k}=\oplus_{\beta\in F'}\,\g_\beta$,
from which it follows that

\centerline {$\dim K=\dim \mathfrak{k}\leq 2$ .}

However $Y$ being commutative,
$\hat{ \psi}( Y)$ lies inside a maximal torus, and thus we have

\centerline{$\dim K \geq n-1$. }

\noindent It follows that $n\leq 3$. 
\end{proof}

The Lie algebra $\mathfrak{sl}(2)$ has a simple
$\Z/3\Z$-gradation $\Gamma_3$ defined as follows

\centerline{$\deg e=1$, $\deg h =0$ and $\deg f=-1$,}

\noindent where
$e,\,f,\,h$ is the standard basis.

The Lie algebra $\mathfrak{sl}(3)$ has a simple
$\Z/8\Z$-gradation $\Gamma_8$ defined as follows:

\centerline{$\deg f_1+f_2 =-1$, $\deg h_1+h_2 =0$,
$\deg e_1+e_2 =1$, $\deg [f_1,f_2]=2$,}

\centerline{ $\deg f_1-f_2 =3$, $\deg h_1-h_2 =4$,
$\deg e_1-e_2 =5$, and $\deg [e_1,e_2]=6$,}

\noindent where $e_1,\,e_2,\,h_1,\, h_2,\,f_1,\,f_2$ are
the standard Chevalley generators of $\mathfrak{sl}(3)$.

\begin{lemma}
\label{lemma_42} Any simple finite dimensional Lie algebra
with a simple gradation is isomorphic to 
$(\mathfrak{sl}(2),\Gamma_3)$ or
$(\mathfrak{sl}(3),\Gamma_8)$.
\end{lemma}
\begin{proof}
By Lemma \ref{lemma_41}, it is enough to prove that 
$\Gamma_3$ is the unique simple gradation of
$\mathfrak{sl}(2)$ and that
$\Gamma_8$ is the unique simple gradation of
$\mathfrak{sl}(3)$, up to isomorphism. The first assertion is clear.

Set $\g=\mathfrak{sl}(3)$ and let
$\g=\oplus_{\gamma\in\Gamma}\,\g_{\gamma}$ be a simple
gradation of $\g$ by a group $\Gamma$ of order $8$. Since $\g \in \mathcal{G}(\Gamma)$, we have $\Sigma_1\neq \emptyset$ by Lemma \ref{lemma_22}.
Let $\alpha\in \Sigma_1$ and $V\cong \C^3$ be the natural representation of $\mathfrak{g}$. Since the $\mathfrak{s}(\alpha)$-module
$\g$ is spherical, it follows that
$\mathfrak{s}(\alpha)$ is a principal $\mathfrak{sl}(2)$ subalgebra. 
So we can assume that
$L_0=h_1+h_2$, 
$L_\alpha= e_1+ e_2$ and $L_{-\alpha}=  f_1+ f_2$. Since
$\g^{-2}=\C [f_1,f_2]$ is one dimensional,  $[f_1,f_2]$ is $\Gamma$-homogeneous. 
Let $\beta$ be its degree.

So up to a scalar multiple, we have 
$L_\beta=[f_1,f_2]$, $L_{\beta+\alpha}=f_1-f_2$,
$L_{\beta+2\alpha}=h_1-h_2$, 
$L_{\beta+3\alpha}=e_1-e_2$ and 
$L_{\beta+4\alpha}=[e_1,e_2]$. Since 
$L_\beta$ and $L_{\beta+4\alpha}$ are not proportional,
we have $4\alpha\neq 0$. Thus $\Gamma$ is cyclic and it is generated
by $\alpha$. Since the 8 elements of $\Gamma$ are
$0$, $\pm\alpha$ and $\beta+i\alpha$ for $0\leq i\leq 4$, it follows 
easily that $\beta=2\alpha$. Thus the gradation is isomorphic to
$\Gamma_8$. 
\end{proof}
\noindent{\textbf{Remark.}}~
\textit{ 
Of course, Lemma \ref{lemma_42} can be proved by general arguments on
simple finite dimensional Lie algebras. The only interest in our proof, using the class $\mathcal{G}(\Gamma)$ and Lemma \ref{lemma_22}, is that we do not need new notations. \\
Also note that Lemma \ref{lemma_42} can be seen as a particular case of the general results of the recent paper \cite{E}.}

\begin{thm}\label{theorem_2} Let $\Lambda$ be a lattice and
let $\mathcal{L}$ be an integrable primitive Lie algebra of
the class $\mathcal{G}$. Then $\Lambda=\Z$ and
$\mathcal{L}$ is isomorphic to
$A^{(1)}_1$ or $A^{(2)}_2$.
\end{thm}
\begin{proof}
By Theorem 2', there exists a finite
abelian group $F$, a simple Lie algebra $\g$ with a simple
$F$-gradation and a
surjective morphism
$\pi:\Lambda\rightarrow F$ such that
$\mathcal{L}\simeq\pi^\ast\g$. By the previous lemma,
$F$ is cyclic. Since $\mathcal{L}$ is primitive,
$\Ker\pi$ contains no primitive vectors. Hence
$\Lambda=\Z$. 

Moreover the graded simple Lie algebra $\g$ is isomorphic to
$(\mathfrak{sl}(2),\Gamma_3)$ or
$(\mathfrak{sl}(3),\Gamma_8)$.It is clear that 
$\pi^\ast((\mathfrak{sl}(2),\Gamma_3))\simeq A^{(1)}_1$,
for any surjective morphism $\pi:\Z\rightarrow \Z/3\Z$
and $\pi^\ast((\mathfrak{sl}(3),\Gamma_8))=A^{(2)}_2$
for any surjective morphism $\pi:\Z\rightarrow \Z/8\Z$.
It follows that
$\mathcal{L}$ is isomorphic to $A^{(1)}_1$ or $A^{(2)}_2$. 
\end{proof}
\part{Classification of non-integrable Lie algebras
in $\mathcal{G}$}\label{chapter_III}
\section{Rank 1 subalgebras}\label{sect_11}
From now on, let $\Lambda$ be a lattice and let
$\mathcal{L}\in \mathcal{G}(\Lambda)$ be
a non-integrable  Lie algebra.  The goal of this section is Lemma \ref{lemma_46},
namely that any  graded subalgebra isomorphic to $\mathfrak{sl}(2)$
lies in a Witt algebra. 

For any $\lambda\in \Lambda$, set

\centerline{$\Omega(\lambda)=\Supp [\mathcal{L},L_\lambda]$.}

\begin{lemma}\label{lemma_43} Let $\lambda\in\Lambda$. Then we have
$\Lambda=F+\Omega(\lambda)$,
for some finite subset $F$ of $\Lambda$.
\end{lemma}
\begin{proof}
By Theorem \ref{theorem_1}, the function
$l$ is additive, and moreover $l\not\equiv 0$
(otherwise $L_0$ would be
central). Since
$\Omega(0)=\{\mu\in\Lambda\vert \,l(\mu)\neq 0\}$, we have 
$\Lambda=\Omega(0)\cup\,\alpha+\Omega(0)$, where
$\alpha$ is any element with $l(\alpha)\neq 0$.

However by Lemma \ref{lemma_4}, we have $\Omega(\lambda)\equiv\Omega(0)$.
Therefore we have
$\Lambda=F+\Omega(\lambda)$
for some  finite subset $F$ of $\Lambda$.
\end{proof}

For $\alpha\in\Lambda \setminus \{0\}$, set
$\mathcal{L}^{(L_\alpha)}=\{x\in\mathcal{L}\vert\,
\ad^n(L_{\alpha})(x)=0,\,\forall \, n>>0\}$.

\begin{lemma}\label{lemma_44} 
Let $\alpha\in\Sigma$. There are no
$\lambda\in\Lambda$ such that
$[\mathcal{L},L_\lambda]\subset \mathcal{L}^{(L_\alpha)}$.
\end{lemma}
\begin{proof}
\textit{Step 1:} A subset $X\subset \Lambda$ is called
$\alpha$-\textit{bounded} if for any $\beta\in\Lambda$ there exists
$n(\beta)\in \Z$ such that 
$\beta+n\alpha\notin X$ for any $n\geq n(\beta)$. 

Set  $\mathfrak{s}(\alpha)=
\C L_{-\alpha}\oplus \C L_0\oplus \C L_\alpha$. By
hypothesis, $\mathfrak{s}(\alpha)$ is isomorphic to $\mathfrak{sl}(2)$. For any  
$\beta\in\Lambda$, set 
$\mathcal{M}(\beta)=\oplus_{n\in \Z}\,  \,\mathcal{L}_{\beta+n\alpha}$. 
As a $\mathfrak{sl}(2)$-module,
$\mathcal{M}(\beta)$ is a weight module with weight multiplicities
$1$. It follows from Gabriel's classification (\cite{D}, 7.8.16) that 
$[L_\alpha,L_{\beta+n\alpha}]\neq 0$ for $n>>0$. Therefore
$\Supp \mathcal{L}^{(L_\alpha)}$ is $\alpha$-bounded. 

\textit{Step 2:} Since $\Supp \mathcal{L}^{(L_\alpha)}$  is 
$\alpha$-bounded, 
there is no finite subset $F$ of $\Lambda$ such that 
$F+\Supp \mathcal{L}^{(L_\alpha)}=\Lambda$. Therefore
by Lemma \ref{lemma_43}, we have: 

\centerline{$\Omega(\lambda)\not\subset \mathcal{L}^{(L_\alpha)}$, }

\noindent which proves the lemma. 
\end{proof}

The Witt algebra $W=\Der\,\C[z,z^{-1}]$ has 
basis $L_n=z^{n+1}\dfrac{\d}{\d z}$, where $n$ runs over 
$\Z$. We have $[L_n,L_m]=(m-n)\,L_{n+m}$.
It contains two subalgebras
$W^{\pm}=\Der\,\C[z^{\pm 1}]$. The Lie algebra
$W^+$ has basis $(L_n)_{n\geq -1}$ and 
$W^-$ has basis $(L_n)_{n\leq 1}$. Their intersection
$W^+\cap W^-$ is  isomorphic
to $\mathfrak{sl}(2)$. 

Set $V^\pm=W/W^\pm$. Then $V^\pm$ is a $W^\pm$-module,
$V^+\oplus V^-$ is a $\mathfrak{sl}(2)$-module.
Identify the elements $L_n\in W$ with their images in $V^{\pm}$.
Then $(L_n)_{n\leq -2}$ is a basis of $V^+$, and  
$(L_n)_{n\geq 2}$ is a basis of $V^-$.

Let $\mathcal{V}$ be the class of all Lie algebras 
$\mathcal{D}$ with a basis $(\d_n)_{n\in\Z}$ satisfying
$[\d_0,\d_n]=n\, \d_n$, for all $n\in \Z$. An algebra
$\mathcal{D}\in \mathcal{V}$ admits a $\Z$-gradation, relatively to which
$\d_n$ is homogeneous of degree $n$.

\begin{lemma}[\cite{M1}]\label{lemma_45} Let $\mathcal{D}\in \mathcal{V}$. Assume that
$[d_1,d_{-1}]\neq 0$. As a $\Z$-graded Lie algebra,
$\mathcal{D}$ is isomorphic to one of the  following four  Lie algebras:
\begin{enumerate}
\item[(i)] $W$,
\item[(ii)] $W^+\, \sd V^+$,
\item[(iii)] $W^-\, \sd V^-$,
\item[(iv)] $\mathfrak{sl}(2)\, \sd (V^+\oplus V^-)$,
\end{enumerate}
where, $V^{\pm}$ and $V^+\oplus V^-$ are abelian ideals.
\end{lemma}
\begin{proof}
See \cite{M1}, Lemma 16. In  \textit{loc. cit.}, the statement
is slightly more general, because it is only assumed that
$\dim \mathcal{D}_n\leq 1$. The assumption $\dim \mathcal{D}_n= 1$ for all $n$
corresponds with the following four types of the Lemma 16 of
\cite{M1}:
type (2,2), type (2,3) (with $q=1$),
type (3,2) (with $p=1$) and type (3,3) (with $p=q=1$).
\end{proof}

For any
$\alpha\in\Lambda$, set 
$\mathcal{L}(\alpha)=\oplus_{n\in \Z}\, \mathcal{L}_{n\alpha}$.
If $l(\alpha)\neq 0$, $\mathcal{L}(\alpha)$ belongs to the class
$\mathcal{V}$.

\begin{lemma}\label{lemma_46} Let $\alpha\in\Sigma$. Then
$\mathcal{L}(\alpha)$ is isomorphic to $W$.
\end{lemma}
\begin{proof}
\textit{Step 1:} Assume otherwise. By the previous
lemma, $\mathcal{L}(\alpha)$ contains an abelian ideal. Exchanging the role of
$\pm\alpha$ if necessary, we can assume that
$M=\oplus_{n\leq -2}\,\mathcal{L}_{n\alpha}$ is an abelian ideal of
$\mathcal{L}(\alpha)$.

Set $\mathcal{Q}=\oplus_{n\leq 1}\,\mathcal{L}_{n\alpha}$.
It follows that $\mathcal{Q}\simeq\mathfrak{sl}(2)\, \sd M$.
For any  $\Z\alpha$-coset $\beta\subset\Lambda$, 
set $\mathcal{M}(\beta)=\oplus_{\gamma\in \beta}\, 
\,\mathcal{L}_{\gamma}$ and $\mathcal{F}(\beta)
=\mathcal{M}(\beta)/\mathcal{M}(\beta)^{(L_\alpha)}$,
where $\mathcal{M}(\beta)^{(L_\alpha)}=\mathcal{M}(\beta)\cap 
\mathcal{L}^{(L_\alpha)}$. Since
$\ad(L_\alpha)$ is locally nilpotent on $\mathcal{Q}$,
$\mathcal{M}(\beta)^{(L_\alpha)}$ is a $\mathcal{Q}$-submodule 
and thus $\mathcal{F}(\beta)$ is a $\mathcal{Q}$-module.

In the next two points, we will prove the following assertion:

\centerline{($\mathcal{A}$)\hskip5mm \textit{for any $\Z\alpha$-coset
$\beta$, $L_{-2\alpha}$ acts trivially on 
$\mathcal{F}(\beta)$.}}

\textit{Step 2:} Assume first that $\mathcal{M}(\beta)^{(L_\alpha)}\neq
0$ and prove Assertion ($\mathcal{A}$) in this case. 

It follows from
$\mathfrak{sl}(2)$-theory  that 
$\Ker \ad(L_\alpha)\vert_{\mathcal{M}(\beta)}$ has dimension $\leq 2$
(see \cite{D} 7.8.16 or \cite{Mi}).
So there exists $\gamma\in {\beta}$ such that
$\mathcal{M}(\beta)^{(L_\alpha)}=\oplus_{n>0}
\mathcal{L}_{\gamma-n\alpha}$. Thus, 
$\mathcal{F}(\beta)$ is the free $\C [L_\alpha]$-module
of rank one generated by $L_\gamma$. Since 
$\mathcal{L}_{\gamma-2\alpha}\subset \mathcal{M}(\beta)^{(L_\alpha)}$, we
have:

 \centerline{
$[L_{-2\alpha},L_\gamma]=0$ modulo $\mathcal{M}(\beta)^{(L_\alpha)}$.}

\noindent So $L_{-2\alpha}$ acts trivially on 
the generator $L_\gamma$ of the $\C[L_\alpha]$-module $\mathcal{F}(\beta)$. Since $[L_{-2\alpha},L_\alpha]=0$, it follows 
that $L_{-2\alpha}$ acts trivially on $\mathcal{F}(\beta)$. 

\textit{Step 3:} Assume now that $\mathcal{M}(\beta)^{(L_\alpha)}=0$
and prove Assertion ($\mathcal{A}$) in this case.

After a renormalization of $L_0$, it can be assumed that
$l(\alpha)=1$. Also $M$ is isomorphic to an irreducible
Verma module, so we have $[L_\alpha,L_{-3\alpha}]\neq 0$. After a suitable
renormalization, it can be assumed that 
$[L_{-3\alpha},L_{\alpha}]=L_{-2\alpha}$.

Let $t$ be the action of $L_\alpha$ on 
$\mathcal{M}(\beta)$ and let $\gamma\in\beta$.
By hypothesis, $t$ acts injectively on 
 $\mathcal{M}(\beta)$. Since 
$t.\mathcal{M}(\beta)_{\gamma+n\alpha}\subset 
\mathcal{M}(\beta)_{\gamma+(n+1)\alpha}$ and 
$\dim \mathcal{M}(\beta)_{\gamma+n\alpha}=1$ for any $n$, it follows
that $t$ acts bijectively. Hence 
$\mathcal{M}(\beta)$ is the  free
$\C[t,t^{-1}]$-module of rank $1$ generated by $L_\gamma$.

Use this generator $L_\gamma$ to identify $\mathcal{M}(\beta)$ with 
$\C[t,t^{-1}]$ and denote by 
$\rho:\mathcal{Q}\rightarrow\End(\C[t,t^{-1}])$ the
corresponding action.

Set $d=\rho(L_0)$. Since  $[L_0,L_\alpha]=L_\alpha$
and $[L_0,L_\gamma]=l(\beta)L_\gamma$, we get that
$[d,t]=t$ and $d.1=l(\gamma)$. Therefore $d=t\d/\d t+l(\gamma)$.

Set $X=\rho(L_{-2\alpha})$. Since $[L_\alpha,L_{-2\alpha}]=0$
and $[L_0,L_{-2\alpha}]=-2 L_{-2\alpha}$, we get 
$[t,X]=0$ and $[t\d/\d t,X]=-2X$. So we have
$X=at^{-2}$ for some $a\in\C$.

Set $Y=\rho(L_{-3\alpha})$. Since 
$[L_{-3\alpha},L_{\alpha}]=L_{-2\alpha}$
and $[L_0,L_{-3\alpha}]=-3 L_{-3\alpha}$, we get 
$[Y,t]=at^{-2}$ and $[t\d/\d t,Y]=-3Y$. So we have
$Y=at^{-2}\d/\d t+b t^{-3}$ for some $b\in\C$.

 Since $M$ is an abelian ideal, we have 
$[\rho(L_{-3\alpha}),\rho(L_{-2\alpha})]=0$, i.e. 

\centerline{$[at^{-2}\d/\d t+b t^{-3},at^{-2}]=0$,}

\noindent
from which it
follows that $a=0$ and $X=0$. Therefore $L_{-2\alpha}$ acts trivially
on $\mathcal{M}(\beta)=\mathcal{F}(\beta)$.

\textit{Step 4:}  Assertion ($\mathcal{A}$) is equivalent to

\centerline{$[\mathcal{L},L_{-2\alpha}]\subset \mathcal{L}^{(L_\alpha)}$.}

\noindent This contradicts Lemma \ref{lemma_44}.  It follows that 
$\mathcal{L}(\alpha)$ is necessarily isomorphic to $W$.
\end{proof}

\section{Symbols of twisted pseudo-differential operators on the circle}\label{sect_12}
\subsection{Ordinary pseudo-differential operators on the
circle}

Denote by $A$ the Laurent polynomial ring
$\C[z,z^{-1}]$. By definition, $\Spec A$ is
called the \textit{circle}. Let $D^+$ be the algebra of differential
operators on the circle.  An element in
$D^+$ is a finite sum $a=\sum_{n\geq 0}\, a_n
\partial^n$, where $a_n\in A$ and where $\partial$ stands for
$\frac{\d}{\d z}$.
The product of the differential operator $a$ by a differential
operator $b=\sum_{m\geq 0}\, b_m \partial^m$ is described  by the
following formula:

\centerline{$a.b=
\sum_{k\geq 0} \sum_{n,\,m\geq 0}\,\binom{n}{k}
a_n(\partial^k\,b_m)\,\partial^{n+m-k}.$}

A \textit{pseudo-differential operator} on the circle is a formal 
series 

\centerline{$a=\sum_{n\in \Z}\, a_n \partial^n$,}

\noindent where $a_n\in A$
and $a_n=0$ for $n>>0$.  The space $D$ of all pseudo-differential
operators has a natural structure of algebra (see below for
a precise definition of the product).

\subsection{Twisted differential operators on the circle}
It is possible to enlarge the algebra $D^+$ by
including complex powers of $z$. 

Let $\mathcal{A}$ be the algebra with
basis $(z^s)_{s\in\C}$ and product 
$z^sz^t=z^{s+t}$. The derivation $\partial$ extends to
$\mathcal{A}$ by  $\partial z^s=s.z^{s-1}$. Set
$\mathcal{D}^+=\mathcal{A}\otimes_A D^+$. The
product on $D^+$ extends naturally to $\mathcal{D}^+$. The algebra
$\mathcal{D}^+$ will be called the \textit{algebra of twisted differential
operators} on the circle.

 As usual, set $\binom{s}{k}=(1/k!)s(s-1)\dots (s-k+1)$ for
any $s\in\C$ and any $k\in\Z_{\geq 0}$. 
When $s$ is a non-negative integer,  $(^s_k)$ is the usual binomial
coefficient. The  product in $\mathcal{D}^+$ is defined by the following
formula:

   \centerline{$z^s\partial^m.z^t\partial^n=
\sum_{k\geq 0}\,k!\binom{m}{k}\binom{t}{k}\,z^{s+t-k}\partial^{m+n-k}.$}

\noindent Here the sum is finite, because $\binom{m}{k}=0$ for $k>m$. 

Let
$\mathcal{D}^+_{\leq n}$ be the space of all differential operators of
order $\leq n$. Set $\mathcal{P}^{+}=\oplus_{n\geq 0}\,
\mathcal{D}^+_{\leq n}/\mathcal{D}^+_{\leq n-1}$. As usual we have
$\mathcal{D}^+_{\leq m}.\mathcal{D}^+_{\leq n}\subset \mathcal{D}^+_{\leq m+n}$ and
$[\mathcal{D}^+_{\leq m},\mathcal{D}^+_{\leq n}]\subset 
\mathcal{D}^+_{\leq m+n-1}$, therefore $\mathcal{P}^{+}$ has a natural structure
of Poisson algebra. As usual, an element 
$a\in \mathcal{D}^+_{\leq n}\setminus 
\mathcal{D}^+_{\leq n-1}$ has exactly \textit{order $n$}, and its image
$\sigma(a)=a$ mod $\mathcal{D}^+_{\leq n-1}$ in $\mathcal{P}^{+}$ is called its
\textit{symbol}. That is why $\mathcal{P}^{+}$
is called the \textit{algebra of symbols of twisted differential
operators}. In what follows, the Poisson bracket of symbols
will be denoted by  $\{,\}$. 

\subsection{Twisted pseudo-differential operators on the circle}
Similarly,  it is possible to enlarge the algebra $D$ by adding complex
powers of $z$ and $\partial$.  

Since the formula involves an infinite
series in powers of  $\partial$, a restriction
on the the support of the series is necessary to ensure
the convergence of
the series defining  the product.

 For any
$s\in\C$, set 
$]-\infty,s]=\{s-n\vert\,n\in\Z_{\geq 0}\}$. Say that a
subset $X$ of $\C$ is \textit{good} if there exists a finite set
$S$ such that $X\subset\cup_{s\in S}\,]-\infty,s]$. Let $\mathcal{D}$ be
the space of all formal series $\sum_x\,a_x\partial^x$, where
$a_x\in\mathcal{A}$ for any $x\in\C$ and where 
$a_x=0$ for all $x$ outside a good subset of $\C$. 
An element of $\mathcal{D}$ is called a \textit{twisted
pseudo-differential operator}. Then
one can define a product on $\mathcal{D}$ by the formula

 \centerline{$z^s\partial^x.z^t\partial^y=
\sum_{k\geq 0}\,k!\binom{x}{k}\binom{t}{k}\,z^{s+t-k}\partial^{x+y-k}.$}

\noindent
Thank to the restriction on the support, the product of general
twisted pseudo-differential operators is well-defined.

Unfortunately, the definition of the order of a twisted
pseudo-dif\-feren\-tial operator requires a non-natural choice
of a total ordering $<$ on $\C$ (viewed as an abelian group)
in a such way that its restriction to $\Z$ is the usual order. 
Then  say that the operator $a\in\mathcal{D}$ has \textit{order $x$}
if $a$ can be writen as $a=a_x\partial^x+\sum_{y<x}\,a_y\partial^y$,
where $a_x\in \mathcal{A}$ is not zero. Let
$\mathcal{D}_{\leq x}$ (respectively $\mathcal{D}_{<x}$)  be the subspace of all
twisted pseudo-differential operators of order $\leq x$ (respectively of
order $<x$) and set 
$\mathcal{P}=\oplus\, \mathcal{D}_{\leq x}/\mathcal{D}_{<x}$ 

We have 
$\mathcal{D}_{\leq x}.\mathcal{D}_{\leq y}\subset \mathcal{D}_{\leq x+y}$ and
$[\mathcal{D}_{\leq x},\mathcal{D}_{\leq y}]\subset \mathcal{D}_{\leq x+y-1}$
and therefore $\mathcal{P}$ has a natural structure of Poisson algebra.
It should be noted that any good subset of $\C$ has a maximal 
element, therefore the symbol $\sigma(a)$ of any $a\in\mathcal{D}$ is
well-defined. For $\lambda=(s,t)\in \C^2$, set 
$E_\lambda=\sigma(z^{s+1}\partial^{t+1})$. Then 
the commutative product of $\mathcal{P}$ is given by the formula
$E_\lambda.E_\mu=E_{\lambda+\mu+\rho}$ and the Poisson bracket by
$\{E_\lambda,E_\mu\}=<\lambda+\rho\vert\mu+\rho> E_{\lambda+\mu}$,
where $<,>$ is the standard symplectic form on $\C^2$ and
$\rho=(1,1)$. Indeed for any $s,\,t\,\in\C$ the symbol of 
$z^s\partial^t$ is independent on the  choice of a 
total order $<$. Thus the whole Poisson structure does not depend on 
this non-natural choice.

\subsection{Decomposition of $\mathcal{P}$ under the Witt algebra $W$}

For any integer $n$, set $L_n=\sigma(z^{n+1}\partial)$. We have
$\{L_n,L_m\}=(m-n)L_{n+m}$. Therefore, the Lie algebra $W=\oplus_n\,\C L_n$ is isomorphic  to the Witt algebra, i.e. the derivation algebra of
$\C[z,z^{-1}]$.

Fix $\delta\in\C$. Set  ${\Omega}^\delta=\mathcal{D}_{\leq
-\delta}/\mathcal{D}_{<-\delta}$. Thus $W$ is a Lie subalgebra of 
${\Omega}^{-1}$ and each $\Omega^\delta$ is a $W$-module.  
For any $x$, set  
$u_x^{\delta}=\sigma(z^{x-\delta}\partial^{-\delta})
=E_{x-\delta-1,-\delta-1}$. Note that

\centerline{$\{L_n,u_x^\delta\}=(x+n\delta)u_{x+n}^\delta$.}

\noindent For any coset $s\in\C/\Z$, set 

\centerline{$\Omega^\delta_s=\oplus_{x\in s}\,\C u
^\delta_x$.}

\noindent  There is a decomposition
$\Omega^\delta=\oplus_{s\in\C/\Z}\,\Omega^{\delta}_s$,
where  each $\Omega^\delta_s$ is a $W$-sub\-mo\-du\-le. These $W$-modules 
$\Omega^\delta_s$ are called the {\it tensor densities modules}. 

It is clear that $\Omega^0_0=A$ and $\Omega^1_0=\Omega^1_A$, where $\Omega_A^1$ is the module of K\"{a}hler differential of $A$.
As $W$-module, $\Omega^0_0$ and $\Omega^1_0$ have length two.
Indeed, set $\overline A=\C[z,z^{-1}]/\C$. 
Their  composition series are described by the following exact
sequences:

\centerline{
$0\rightarrow \C \rightarrow \Omega^0_0\rightarrow
\overline A\rightarrow 0$ and
$0\rightarrow \overline{A}\rightarrow \Omega^1_0\rightarrow
\C\rightarrow 0$.
}

\noindent
Otherwise, for $s\neq 0$ or for $\delta\notin\{0,\,1\}$, the 
$W$ module $\Omega^\delta_s$ is irreducible. It should
be noted that $W=\Omega^{-1}_0$.

\subsection{The $W$-equivariant bilinear maps $P^{\delta,\delta'}_{s,s'}$
and $B^{\delta,\delta'}_{s,s'}$}\label{sect_12.5}

Set $Par=\C\times\C/\Z$. There is a decomposition
$\mathcal{P}=\oplus_{(\delta,s)\in Par}\,\Omega^{\delta}_s$. 
It follows from the explicit description of the Poisson structure on
$\mathcal{P}$ that we have:

\centerline{
$\Omega^{\delta}_s.\Omega^{\delta'}_{s'}
\subset
\Omega^{\delta+\delta'}_{s+s'}$ and 
$\{\Omega^\delta_s,\Omega^{\delta'}_{s'}\}\subset
\Omega^{\delta+\delta'+1}_{s+s'}$. }

\noindent for any quadruple
$(\delta,s),\,(\delta',s')\in Par$. Accordingly, we get 
two $W$-equi\-va\-riant bilinear maps:

\centerline{
$P^{\delta,\delta'}_{s,s'}:\Omega^{\delta}_s\times\Omega^{\delta'}_{s'}
\rightarrow\Omega^{\delta+\delta'}_{s+s'}$ and
$B^{\delta,\delta'}_{s,s'}:\Omega^{\delta}_s\times\Omega^{\delta'}_{s'}
\rightarrow\Omega^{\delta+\delta'+1}_{s+s'}$.}

\noindent 
 In what follows, 
these morphisms $P^{\delta,\delta'}_{s,s'}$ and
$B^{\delta,\delta'}_{s,s'}$ will be called the \textit{
commutative product} and the  \textit{Poisson bracket product}. These
bilinear maps are always non-zero, except the Poisson bracket for
$\delta=\delta'=0$.

\subsection{The outer derivations $\log z$ and $\log\partial$}\label{sect_12.6}
Recall the following obvious fact:

\begin{lemma}\label{lemma_47} Let $R$ be a Poisson algebra, and let
$d\in R$ be an invertible element. Then the map
$r \in R\mapsto \{d,r\}/d \in R$ is a derivation.
\end{lemma}

In what follows, it will be convenient
to use the notation  
$\{\log d, r\}$ for $\{d,r\}/d$. 
Since the ordinary bracket of operators is denoted by $[,]$ and
the Poisson bracket of symbol is denoted by $\{,\}$, 
it will be convenient to write $z^s\partial^\delta$ for
$\sigma(z^s\partial^\delta)$. In an expression like
$\{z^s\partial^\delta, z^{s'}\partial^{\delta'}\}$ it is clear that the
arguments are symbols of operators.

In the Poisson algebra $\mathcal{P}$, both $z$ and $\partial$ are
invertible. Therefore $\{\log z, \cdot \}$ and $\{\log\partial, \cdot\}$ are derivations
of $\mathcal{P}$. Let $\mathcal{E}\subset \Der \mathcal{P}$ the vector
space generated by $\ad \,\Omega^0_0$, $\{\log z, \cdot\}$ and $\{\log\partial, \cdot\}$.
Also set $\overline{A}=A/\C=\Omega^0_0/\C$. As a vector space,
it  is clear that $\mathcal{E}=\overline{A}\oplus \C\{\log\,z, \cdot\} \oplus
\C\{\log\,\partial, \cdot\}$.

We have:

\centerline{$\{\log z, z^s\partial^\delta\}=-\delta
z^{s-1}\partial^{\delta-1}$,  \quad
$\{\log \partial,z^s\partial^\delta\}=
s z^{s-1}\partial^{\delta-1}$.}

\noindent It follows that $\mathcal{E}$ is a $W$-module,
with basis 
$\{\log z, \log \partial, (e_n) \vert\, n\neq 0\}$
and the $W$-module structure is given by:

$L_n.e_m=m e_{n+m}$ if $n+m\neq 0$ and $0$ otherwise,

$L_n.\log z=e_n$ if $n\neq 0$ and $0$ otherwise,

$L_n.\log\partial=-(n+1) e_n$ if $n\neq 0$ and $0$ otherwise,

\noindent where  $e_n$ is the image of $z^n$ in $\overline{A}$. 

\begin{lemma}\label{lemma_48} 
\begin{enumerate}
\item[(i)] As a $W$-module, there is an exact sequence:

\centerline{
$0\rightarrow \overline{A}\rightarrow \mathcal{E}
\rightarrow \C^2\rightarrow 0$
}

\item[(ii)] $\mathcal{E}.\Omega^{\delta}_s\subset \Omega^{\delta+1}_s$,
for all $(\delta,s)\in Par$.
\end{enumerate}
\end{lemma}
\begin{proof}
Both assertions follow from the previous
computations.
\end{proof}

\subsection{The Lie algebra $W_\pi$}\label{sect_12.7}
As a Lie algebra, $\mathcal{P}$ has a natural $\C^2$-gradation
$\mathcal{P}=\oplus_{\lambda\in\C^2}\,\mathcal{P}_\lambda$ where
$\mathcal{P}_\lambda=\C E_{\lambda}$. 
For any additive map $\pi:\Lambda\rightarrow\C^2$
set,  $W_\pi=\pi^\ast\mathcal{P}$. When  $\pi$ is one-to-one,
$W_\pi$ has been defined in the introduction. In general,
the notation $\pi^\ast$ has been defined in Section \ref{sect_2}.

\begin{lemma}\label{lemma_49} 
\begin{enumerate}
\item[(i)] The Lie algebra $W_\pi$ is  simple graded iff: 

\centerline{$(\mathcal{C})$ \quad $\Im\,\pi\not\subset\C\rho$ 
and $2\rho\notin \Im\,\pi$.}

\item[(ii)] Moreover if $\pi$ is one to one, then
$W_\pi$ is simple.
\end{enumerate}
\end{lemma}
\begin{proof}
We may assume that $\pi$ is one-to-one
and that $\Lambda$ is a subgroup of $\C^2$.
Thus $W_\pi$ has basis $(E_\lambda)_{\lambda\in\Lambda}$ and
the bracket is given by the formula \\
\centerline{$[E_\lambda,E_\mu]=<\lambda+\rho\vert\mu+\rho> E_{\lambda+\mu}$.}

If $\Lambda \subset \C \rho$, then $W_\pi$ is abelian. If $2\rho \in \Lambda$, then
$E_{2\rho} \in W_\pi \setminus [W_\pi, W_\pi]$. Thus, the condition $\mathcal{C}$ is
necessary for the simplicity of $W_\pi$.

Conversely, assume the condition $\mathcal{C}$. 
First prove that $W_\pi$ is simple as a graded Lie algebra. 
Let $\lambda$, $\mu\in\Lambda$. For $\theta\in\Lambda$, set

$g(\theta)=<\lambda+\rho\vert \theta+\rho>$

$h(\theta)=<\lambda+\theta+\rho\vert \mu+2\rho>$

\noindent Since $2\rho \not\in \Lambda$, we have $\lambda+\rho\neq 0$
and $\mu+2\rho\neq 0$, and so $g$ and $h$ are not constant affine functions. 
Since their zero sets are proper cosets of $\Lambda$, there exists
$\theta\in\Lambda$ with $g(\theta)h(\theta)\neq 0$.

Note that $h(\theta)=<\lambda+\theta+\rho\vert \mu-\lambda-\theta+\rho>$,
and therefore:

\centerline{
 $[E_{\mu-\theta-\lambda}, [E_\theta,E_\lambda]]=g(\theta)h(\theta)
E_\mu$. }

\noindent It follows that for any $\lambda,\,\mu\in\Lambda$,
$E_\mu$ belongs to $\Ad(U(W_\pi))(E_\lambda)$. So $W_\pi$ is simple as a
$\Lambda$-graded Lie algebra.

Next prove that $W_\pi$ is simple. Let $\psi$
be an homogeneous element of the centroid $C(W_\pi)$,
and let $\mu$ be its degree. Since 
$\psi$ is injective and since $\psi(L_0)$ commutes with 
$L_0$, we have $\mu=z\rho$ for some $z \in \C$.

Fix any $\lambda\in\Lambda$ with $<\lambda\vert\,\rho>\neq 0$.
Define the function $c:\Lambda\rightarrow \C$ by the
requirement:

\centerline{ $\ad(E_{-\lambda})\ad(E_\lambda) (E_\mu)=c(\mu) E_{\mu}$.}

\noindent Using the facts   that $\psi$
commutes with $\ad(E_{-\lambda})\ad(E_\lambda)$ and that  $\psi$ is
injective, it follows that
$c(\mu+z\rho)=c(\mu)$, $\forall \mu\in\Lambda$ and therefore the
function
$n\in\Z\mapsto c(\mu+nz\rho)$ is constant. However
we have $c(\mu)=<-\lambda+\rho\vert\lambda+\mu+\rho>
<\lambda+\rho\vert\mu+\rho>$, so 
$c(\mu+nz\rho)$ is a degree two polynomial in $n$ with highest term is
$-z^2<\lambda\vert\,\rho>^2 n^2$. Therefore $z=0$, which means that
$C(W_\pi)=\C$. 

It follows from Section \ref{sect_2}, in particular Lemma \ref{lemma_9}, that
$W_\pi$ is simple.
\end{proof}

\section{Tensor products of generalized tensor densities modules}
\label{sect_13}

This section is a review of the results of \cite{KS} and
\cite{IM} which are used later on. Indeed \cite{IM}
contain the whole list of all $W$-equivariant bilinear maps
$\mu: M^1 \times M^2\rightarrow N$ , 
where $M^1$, $M^2$ and $N$ are in a certain class $\mathcal{S}(W)$ (defined below). However 
the classification contains many cases. Here we will only
state the consequences which are of interest for this paper. 

\subsection{The Kaplansky Santharoubane Theorem}\label{sect_13.1}
As before, $W=\Der \C[z,z^{-1}]$ denotes the Witt Lie algebra.
Given a $W$-module $M$, set $M_x=\{m\in M\vert L_0.m=xm\}$.
Let $\mathcal{S}(W)$ be the class of all
$W$-modules such that  there exists 
$s\in\C/\Z$ satisfying the following
conditions:
\begin{enumerate}
\item[(i)] $M=\oplus_{x\in s}\, M_{x}$
\item[(ii)] $\dim M_{x}=1$ for all $x\in s$.
\end{enumerate}
All the tensor densities  modules $\Omega_s^{\delta}$, which have
been defined in Section \ref{sect_12},  are in the class $\mathcal{S}(W)$. Conversely,
the modules in the class $\mathcal{S}(W)$ are called
\textit{generalized tensor densities modules}.
For $(a,\,b)\in\C^2$, define the modules 
$A_{a,b}$ and $B_{a,b}$ as follows:
\begin{enumerate}
\item[(i)] the module $A_{a,b}$ has basis $(u_n)_{n\in \Z}$ and we have
\begin{align*}
&L_m.u_n=(m+n)u_{m+n}\quad  \text{if}~ n\neq 0, \\
&L_m.u_0=(am^2+bm) u_m.
\end{align*}
\item[(ii)] the module $B_{a,b}$ has basis $(v_n)_{n\in \Z}$ and we have
\begin{align*}
&L_m.v_n=nv_{m+n}\quad \text{if}~ m+n\neq 0, \\ 
&L_m.v_{-m}=(am^2+bm) v_0.
\end{align*}
\end{enumerate}
The family of modules $A_{a,b}$ with 
$(a,\,b)\in\C^2$
is called the \textit{$A$-family}. Similarly, the family of all modules 
$B_{a,b}$ is called the \textit{$B$-family}. The union of the two families
is called the \textit{$AB$-family}.

Set $A=\C[z,z^{-1}]$ and
$\overline{A}=A/\C$. 
If $M$ is in $A$-family, then there is
an exact sequence
$0\rightarrow \overline{A}\rightarrow M\rightarrow \C\rightarrow 0$. Similarly, if
$N$ is in $B$-family, then there is
an exact sequence
$0\rightarrow \C\rightarrow N\rightarrow
\overline{A}\rightarrow 0$. 
Since we have 
$\overline{A}\oplus \C \simeq
A_{0,0}\simeq B_{0,0}$, the $W$-module  $\overline{A}\oplus \C$
is in both  families.  Otherwise the modules $A_{a,b}$ and
$B_{a,b}$ are indecomposable. Since $A_{a,b}\simeq A_{xa,xb}$
(respectively $B_{a,b}\simeq B_{xa,xb}$)
for any non-zero scalar $x$, the indecomposable modules of 
the $A$-family (respectively of the $B$-family) are parametrized by 
$\P^1$.

Indeed I. Kaplansky and R. Santharoubane gave a
complete classification of all modules in $\mathcal{S}(W)$.

\textbf{Theorem \cite{KS}:} (Kaplansky-Santharoubane Theorem)  
\textit{
\begin{enumerate}
\item[(i)] Any irreducible module $M\in \mathcal{S}(W)$ is 
isomorphic to $\Omega^\delta_s$ for some
$\delta\in\C$, $s\in \C/\Z$ with the
condition $s\neq 0$ if $\delta=0$ or $\delta =1$. 
\item[(ii)] Any reducible module $M\in\mathcal{S}(W)$ is in the
$AB$-family.
\end{enumerate}
}
The theorem is due to Kaplansky and Santharoubane.
The original paper \cite{KS} is correct, but
the statement contains a little misprint (the indecomposable modules were
classified by the affine line instead of the projective line).
For the $A$-family, 
see \cite{MS} for a correct statement, and in general
see \cite{MP} (see also \cite{Kap}).

\subsection{Degree of modules in $\mathcal{S}(W)$}

The de Rham differential provides a $W$-equivariant map
$\d:\Omega^0_s\rightarrow\Omega^1_s$. For $s\notin\Z$,
the map $\d$ is an isomorphism. Otherwise, the modules
$\Omega^{\delta}_s$ are pairwise non-isomorphic. 
Therefore one can define \textit{the degree} $\deg M$ of any $M\in\mathcal{S}(W)$ as follows:

$\deg M=\{\delta\}$ if $M\simeq \Omega^\delta_s$ for
some $\delta\neq 0,\,1$, and

$\deg M=\{0,\,1\}$ otherwise.

\noindent Note that the degree is a multivalued function. 
By definition,
\textit{a degree} for $M$ is a value $\delta\in\deg M$. Let
$\mathcal{S}^\ast(W)$ be the class of pairs $(M,\delta)$ where
$M\in\mathcal{S}(W)$ and $\delta$ is a degree of $M$.
For example, $(\Omega^0_s,0)$ and $(\Omega^0_s,1)$ belong
to $\mathcal{S}^\ast(W)$. 
 Usually an element of   $(M,\delta)\in\mathcal{S}^\ast(W)$ will be 
simply denoted 
by $M$, and we will set
$\deg M=\delta$. So the degree is an ordinary function 
on $\mathcal{S}^\ast(W)$.

\subsection{Degree of bilinear maps in $\mathcal{S}(W)$}

Let $\mathcal{B}(W)$ (respectively 
$\mathcal{B}^\ast(W)$) 
be the set of all $W$-equivariant bilinear
maps $\mu: M^1 \times M^2\rightarrow N$ , 
where $M^1$, $M^2$ and $N$ are
in $\mathcal{S}(W)$ (respectively in $\mathcal{S}^\ast(W)$). 
Let $\mu\in \mathcal{B}(W)$ (respectively $\mu\in \mathcal{B}^\ast(W)$).
By definition the \textit{degree} of $\mu$ is the set
(respectively the number) $\deg\,\mu=\deg N-\deg M^1-\deg M^2$. 
As before, the degree is a multivalued map on 
$\mathcal{B}(W)$ and an ordinary map on
$\mathcal{B}^\ast(W)$.

\bigskip 
Let now $\mu:M^1 \times M^2\rightarrow N$ be a
non-zero bilinear map, where $M^1$, $M^2$, $N$ are in
$\mathcal{S}^\ast(W)$. Set $\delta_1=\deg M^1$,
$\delta_2=\deg M^2$ and $\gamma=\deg N$.

\begin{lemma}[\cite{IM}]\label{lemma_50}
Let $\mu$, $\delta_1,\,\delta_2$ and $\gamma$ as before. 
Then $\deg \mu\in [-2,3]$. Moreover the possible
triples $(\delta_1,\delta_2,\gamma)$ are the following:
\begin{enumerate}
\item[(i)] if $\deg\mu=3$: only $(-2/3,-2/3,5/3)$,
$(0,0,3)$, $(0,-2,1)$ or $(-2,0,1)$,
\item[(ii)]  if $\deg\mu=2$: any triple $(0,\delta,\delta+2)$,
$(\delta,\,0,\,\delta+2)$ or $(\delta,\,-1-\delta,\,1)$,
\item[(iii)]  if $\deg\mu=1$: any triple
$(\delta_1,\delta_2,\delta_1+\delta_2+1)$,
\item[(iv)]   if $\deg\mu=0$: any triple
$(\delta_1,\delta_2,\delta_1+\delta_2)$,
\item[(v)]  if $\deg\mu=-1$: any triple $(1,\delta,\delta)$,
$(\delta,1,\delta)$ or $(\delta,1-\delta,0)$,
\item[(vi)] if  $\deg\mu=-2$: only $(1,1,0)$.
\end{enumerate}
\end{lemma}
For the proof, see \cite{IM}, Corollary 3 in Section 11.

\subsection{Bilinear maps of degree $1$}
Recall from
Section \ref{sect_12.5} that the Poisson bracket of symbols induces
a $W$-equivariant bilinear map:

\centerline{$B^{\delta,\delta'}_{u,u'}:
\Omega^{\delta}_u\times\Omega^{\delta'}_{u'}
\rightarrow\Omega^{\delta+\delta'+1}_{u+u'}$,}

\noindent for any $(\delta,u),\,(\delta',u)\in P$.

\begin{lemma}\label{lemma_51} Let  $\mu:\Omega^{\delta}_u\times
\Omega^{\delta'}_{u'}
\rightarrow\Omega^{\delta+\delta'+1}_{u+u'}$ be a $W$-equivariant map.
\begin{enumerate}
\item[(i)] If $(\delta,\delta')\neq (0,0)$, then $\mu$ is
proportional to $B^{\delta,\delta'}_{u,u'}$.
\item[(ii)] If $(\delta,\delta')= (0,0)$, there exist two scalars
$a$,$b$ such that 

\centerline{$\mu(f, g)=af\d g + b g\d f$,
for any $(f,g)\in \Omega^{0}_u\times
\Omega^{0}_{u'}$,}

\noindent 
where $\d$ is the de Rham differential.
\end{enumerate}
\end{lemma}
\begin{proof}
By Corollary 4 of \cite{IM}, there are no degenerate maps $\pi:
\Omega^{\delta}_u\times\Omega^{\delta'}_{u'}
\rightarrow\Omega^{\delta+\delta'+1}_{u+u'}$.  Thus the lemma follows from Theorems 3.1 and 3.2 of \cite{IM}.
\end{proof}

\subsection{Degenerated bilinear maps with value in $W$}

Let $\mu:M^1 \times M^2\rightarrow N$ be a bilinear map
in the class $\mathcal{B}(W)$. The bilinear map
$\mu$ is called \textit{non-degenerated} if the set of all
$(x,y)$ such that $\mu(M_x\times N_y)\neq 0$ is Zarisky dense in
$\C^2$. Otherwise, it is called \textit{degenerated}.

\begin{lemma}\label{lemma_52} Let
$\mu:M^1 \times M^2\rightarrow W$ be a
non-zero degenerated bilinear map in $\mathcal{B}(W)$.
Then one of the following assertion holds:
\begin{enumerate}
\item[(i)] There is a non-zero morphism $\phi: M^1 \rightarrow \C$ of $W$-modules and
an isomorphism $\psi: M^2\rightarrow W$ of $W$-modules such that
$\mu(m_1, m_2)=  \phi(m_1)\psi(m_2)$ for all 
$m_1\in M^1,\,  m_2\in M^2$.
\item[(ii)] There is an isomorphism $\psi: M^1\rightarrow W$ of $W$-modules
and  a non-zero morphism
$\phi: M^2 \rightarrow \C$ of $W$-modules such that
$\mu(m_1, m_2)=  \phi(m_2)\psi(m_1)$ for all 
$m_1\in M^1,\,  m_2\in M^2$. 
\end{enumerate}
\end{lemma}
\begin{proof}
By Corollary 2 of \cite{IM}, either $M^1$ or $M^2$ is reducible. So, assume that $M^1$ is reducible. Then, Assertion (i) follows from Lemma 13 of \cite{IM}.
\end{proof}

\subsection{Non-degenerated bilinear maps with value in $W$}
Two families of $W$-equivariant
bilinear maps $\mu: M\times N\rightarrow W$ will be defined. 
Since $W\simeq \Omega^{-1}_0$, the product of symbols 
$\pi_s^\delta: \Omega^{\delta}_s\times \Omega^{-1-\delta}_{-s}\rightarrow
W$   and
the Poisson bracket of symbols
$\beta_s^\delta:  \Omega^{\delta}_s\times
\Omega^{-2-\delta}_{-s}\rightarrow W$ are $W$-equivariant bilinear maps.
With the notations of Section \ref{sect_12.5}, we have
 $\pi_s^\delta=P^{\delta,-1-\delta}_{s,-s}$
and $\beta_s^\delta=B^{\delta,-1-\delta}_{s,-s}$.

The \textit{ordinary family} is the set of all these maps
$\pi_s^\delta$ and  $\beta_s^\delta$.

The second family is called the  complementary family.
Its definition requires the
following obvious Lemma.

\begin{lemma}\label{lemma_53} The indecomposable modules of the $A$-family are
exactly the codimension one submodules of $\mathcal{E}$.
\end{lemma}
The lemma follows easily from the structure  constant 
of module $\mathcal{E}$ given in Section \ref{sect_12.6} and the 
the structure  constant of the modules $A_{a,b}$ 
given in Section \ref{sect_13.1}.

The  left complementary family consists of   bilinear maps
$\mu_M^l: M\times \Omega^{-2}_0\rightarrow W$,
where $M$ belongs to the
$AB$-family,  which are defined as follows.

If $M$ is an indecomposable module in  the $A$-family, the Lemma \ref{lemma_53} 
provides an  injection $j:M\rightarrow\mathcal{E}$. Otherwise $M$ contains
a trivial submodule and
$M/\C\simeq \overline{A}$. Thus there is a natural
map $j:M\rightarrow\mathcal{E}$ whose image is
$\overline{A}$.

Recall that 
$\ad(\mathcal{E})(\Omega^{\delta}_0)\subset\Omega^{\delta+1}_0$ 
and that $\Omega^{-1}_0\simeq W$. 
So define  $\mu_M^l:M\times \Omega^{-2}_0\rightarrow W$  by
$\mu^l_M(m,\omega)=\ad(j(m))(\omega)$. This family of maps
$\mu_M^l$ is called the \textit{left complementary family}.
The \textit{right complementary family}
is the family of maps $\mu^r_M: \Omega^{-2}_0\times M\rightarrow W$ 
obtained by exchanging the two factors. The \textit{complementary family}
is the union of the left and of the right complementary families.

It should be noted that these definitions contain the following
ambiguities.  First, when $s\neq 0$ mod $\Z$,
the de Rham operator $\d$ is an isomorphism
$\d: \Omega^0_s\rightarrow \Omega^1_s$, which provides the following
commutative diagram: 

\[ \UseTips
\newdir{ >}{!/-5pt/\dir{>}}
\xymatrix @=1pc @*[r]
{
    \Omega_s^0 \times \Omega_{-s}^{-2} \ar[rrr]^{\phantom{abcdefg}\beta_s^0} 
    \ar[dd]^{d \times \id}
    &&& \Omega_0^{-1}  \ar[dd]^{-\frac{1}{2}} \\
    &&& \\
    \Omega_s^1 \times \Omega_{-s}^{-2} \ar[rrr]^{\phantom{abcdef}\pi_s^1} &&& 
    \Omega_0^{-1} \\
} \]

So, up to conjugation, we have
$\beta^0_s=\pi^1_s$ and $\beta^{-2}_s=\pi^{-2}_s$.
Next, the maps 
 $\pi^1_0: \Omega^{1}_0\times \Omega^{-2}_{0}\rightarrow W$ and
$\beta^0_0: \Omega^{0}_0\times \Omega^{-2}_{0}\rightarrow W$
are in the left complementary family. Similarly,
$\beta^{-2}_0$ and $\pi^{-2}_0$ are in the right complementary family.

\begin{lemma}\label{lemma_54} Any  $W$-equivariant non-degenerated bilinear map
$\mu:M\times N\rightarrow W$ belongs to the ordinary or to the
complementary family.
More precisely, $\mu$ is conjugated to one of the following map:
\begin{enumerate}
\item[(i)] the map $\pi^\delta_s$ for some $\delta\neq 1,\,-2$,
\item[(ii)] the map $\beta^\delta_s$, for 
$(\delta,s)\neq (0,0),\,(-2,0)$,
\item[(iii)] or to a map of the complementary family. 
\end{enumerate}
\end{lemma}
\begin{proof}
See Theorem 3.1. or Table 3 in \cite{IM}. 
\end{proof}

\section{The Main Lemma (non-integrable case)}\label{sect_14}

Let $\mathcal{L}\in\mathcal{G}$ be a non-integrable  Lie algebra.
Recall that $\Sigma$ is the set of all
$\lambda\in \Lambda$ such that
$\mathcal{L}_\lambda\oplus\mathcal{L}_0\oplus\mathcal{L}_{-\lambda}$ is
isomorphic to $\mathfrak{sl}(2)$.

The aim of this section is the proof of the Main Lemma 
in the non-integrable case, Lemma \ref{lemma_62}. A similar statement had 
been already proved for integrable Lie algebras of type I (see Lemma \ref{lemma_33}) and of type II (see Lemma \ref{lemma_25}). However the proof in 
the three cases are quite different.

\begin{lemma}\label{lemma_55} 
Let $M,\,N\in \mathcal{S}(W)$ and let 
$\pi:M\times N\rightarrow W$ be a  non-zero 
$W$-equivariant bilinear  map.  Then one of the
following three assertions holds:
\begin{enumerate}
\item[(i)] $\pi(M_x\times N_{-x})\neq 0$ for all
$x\in \Supp M$ with $x\neq 0$.
\item[(ii)] There is a surjective morphism $\phi: M \rightarrow \C$ and
an isomorphism $\psi: N\rightarrow W$ such that
$\pi(m,n)=  \phi(m)\psi(n)$ for all $m\in M,\, n\in N$.
\item[(iii)] There is an isomorphism $\phi: M\rightarrow W$ and a surjective
morphism $\psi: N\rightarrow \C$ such that
$\pi(m, n)=  \psi(n)\phi(m)$ for all $m\in M,\, n\in N$.
\end{enumerate}
\end{lemma}
\begin{proof}
Assume that neither Assertion (ii) or (iii) hold. Then by
Lemma \ref{lemma_52}, the map $\pi$ is non-degenerated. Moreover, the
non-degenerated bilinear maps $\pi:M\times N\rightarrow W$ are classified
by Lemma \ref{lemma_54}. They are isomorphic to $\pi^\delta_s$, 
$\beta^\delta_s$ or
they are in the complementary family.

If $\pi=\pi^\delta_s$ for some $(\delta,s)\in P$, then $\pi$ is
the restriction of the commutative product of the Poisson algebra $\mathcal{P}$. Since  the commutative algebra $\mathcal{P}$
has no zero divisors, it follows that $\pi(m,n)\neq 0$ whenever
$m$ and $n$ are not zero, thus the assertion is clear if
$\pi=\pi^\delta_s$ for some $(\delta,s)\in P$.

If $\pi=\beta^\delta_s$ for some $(\delta,s)\in P$, then $\pi$ is
the restriction of the Poisson bracket of $\mathcal{P}$.
Let $x\in\Supp M$ and let $m\in M_x$, $n\in N_{-x}$ be non-zero vectors.
When $M$ and $N$ are realized as submodules of $\mathcal{P}$,
$m$ and $n$ are identified with vectors which
are homogeneous in $z$ and $\partial$. So, up to some non-zero
scalar, we have $m=E_\lambda$ and $n=E_{-\lambda}$ for some
$\lambda\in\C^2$
and $\pi(m,n)$ can be identified with 
$\{E_\lambda,E_{-\lambda}\}$. However we have
$\{E_\lambda,E_{-\lambda}\}=2<\lambda\vert\,\rho>E_0$, and therefore
we get:

\centerline{$\{E_\lambda,E_{-\lambda}\}=0$ implies that
$\{E_0,E_\lambda\}=0$.}

\noindent So $\pi(m,n)=0$ implies that $x=0$. 
Thus the assertion is proved if
$\pi=\beta^\delta_s$ for some $(\delta,s)\in P$.

Assume now that $\pi$ is in the left complementary family.
Since $\Omega_0^0/\C\simeq \overline{A}$, the  map
$\beta^0_0:\Omega^0_0\times \Omega^{-2}_0\rightarrow W$ induces a map
$\beta:{\overline A}\times \Omega^{-2}_0\rightarrow W$.
It follows from the previous computation that
$\beta(a,\omega)\neq 0$ if $a$ and $\omega$ are non-zero eigenvectors 
of $L_0$ with opposite eigenvalues.
Since any bilinear map of the left complementary family is a
lift, or an extension of $\beta$, it follows that
$\pi(M_x\times N_{-x})\neq 0$ for any $x\neq 0$. 
\end{proof}

A subset $S$ of $\Lambda$ is called \textit{quasi-additive} if it
satisfies the following three conditions:
\begin{enumerate}
\item[(i)] we have $l(\lambda)\neq 0$ for any $\lambda\in S$,
\item[(ii)] $S=-S$, and
\item[(iii)] we have $\lambda+\mu\in S$ for any
$\lambda,\,\mu\in S$ with $l(\lambda+\mu)\neq 0$.
\end{enumerate}
For a quasi-additive set $S$, let
$K(S)$ be the set of all elements of the form
$\lambda+\mu$ where $\lambda$ and $\mu$ belong to $S$ and
$l(\lambda+\mu)=0$. Set $Q(S)=S\cup K(S)$.

\begin{lemma}\label{lemma_56} 
Let $S$ be a quasi-additive
subset of $\Lambda$. Then $Q(S)$ and $K(S)$ are subgroups
of $\Lambda$.
\end{lemma}
\begin{proof}
 It is clear that 
$K(S)=Q(S)\cap\Ker l$. So it is enough to prove that
$Q(S)$ is a subgroup. Moreover $Q(S)=-Q(S)$, so it is enough to
show that $Q(S)$ is stable by addition.
By definition,
we have $Q(S)= S\cup (S+S)$. So it is enough to show that
$S+S+S\subset S+S$.

Let $\lambda_1,\,\lambda_2,\,\lambda_3\in S$. 
There are two distinct indices $1\leq i<j\leq 3$ such that
$l(\lambda_i+\lambda_j)\neq 0$. Otherwise we would have
$l(\lambda_1)=-l(\lambda_2)=l(\lambda_3)=-l(\lambda_1)$
which would contradict $l(\lambda_1)\neq 0$. Since all indices
play the same role,  we can assume that $l(\lambda_1+\lambda_2)\neq 0$.
It follows that  $\lambda_1+\lambda_2\in S$ and therefore
$\lambda_1+\lambda_2+\lambda_3=(\lambda_1+\lambda_2)+\lambda_3\in S+S$.
\end{proof}

As it is well-known, if $L=L_1\cup L_2$, where $L_1$, $L_2$ and $L$ are
subgroups of $\Lambda$, then $L=L_1$ or $L=L_2$. Indeed a similar 
property holds for quasi-additive subsets.

\smallskip
\begin{lemma}\label{lemma_57} Let $S,\, S_1,\,S_2$ be quasi-additive subsets of
$\Lambda$. If $S=S_1\cup S_2$, then $S=S_1$ or $S=S_2$.
\end{lemma}
\begin{proof} Assume otherwise. Choose
$\alpha\in S\setminus S_2$ and $\beta\in S\setminus S_1$.
Since $l(\beta)\neq 0$, there exists $\epsilon\in\{\pm 1\}$
with $l(\alpha)+l(\epsilon\beta)\neq 0$.

Since $\alpha+\epsilon \beta\in S$,
$\alpha+\epsilon\beta$ belongs to $S_1$ or to $S_2$, it can
be assumed that  $\alpha+\epsilon\beta\in S_1$. 
However we have:

\centerline{$\epsilon\beta=(\alpha+\epsilon\beta)-\alpha$, ~and~
$l(\alpha+\epsilon\beta)+l(-\alpha)\neq 0$,}

\noindent which implies 
that $\epsilon\beta=-\alpha+(\alpha+\epsilon\beta)$ belongs to $S_1$.
This  contradicts the hypothesis $\beta\in S\setminus S_1$.
\end{proof}

Let $\mathcal{L}\in \mathcal{G}$ be a non-integrable  Lie algebra.

\begin{lemma}\label{lemma_58} 
The set $\Sigma$ is quasi-additive.
\end{lemma}

\begin{proof}
It is obvious that $\Sigma$ satisfies the
conditions (i) and (ii) of the definition of quasi-additivity.

In order to prove condition (iii), consider
$\alpha,\,\beta\in \Sigma$ with 
$l(\alpha)+l(\beta)\neq 0$. Set
$\mathcal{L}(\alpha)=\oplus_{n\in\Z}\,\mathcal{L}_{n\alpha}$ and
$\mathcal{M}(\pm\beta)=\oplus_{n\in\Z}\,\mathcal{L}_{\pm\beta+n\alpha}$.
We have 

\centerline{$[\mathcal{M}(\beta),\mathcal{M}(-\beta)]\subset \mathcal{L}(\alpha)$,}

\noindent
therefore the Lie bracket induces a 
$\mathcal{L}(\alpha)$-equivariant bilinear map:

\centerline {$\pi:\mathcal{M}(\beta)\times\mathcal{M}(-\beta)\rightarrow\mathcal{L}(\alpha)$.}

By Lemma \ref{lemma_46}, the Lie algebra $\mathcal{L}(\alpha)$ is isomorphic to
$W$ and by hypothesis $\mathcal{M}(\pm\beta)$ belong to
$\mathcal{S}(W)$. Moreover we have
$\pi(L_\beta, L_{-\beta})\neq 0$. Therefore $\pi$ cannot satisfy 
Assertion (ii) or Assertion (iii) of the Lemma \ref{lemma_55}. Thus Assertion (i) of
Lemma \ref{lemma_55} holds.
Since
$l(\alpha+\beta)\neq 0$, it follows that  
$\pi(L_{\alpha+\beta}, L_{-\alpha-\beta})\neq 0$, and therefore
$\alpha+\beta$ belongs to $\Sigma$. 
\end{proof}

Set $Q=Q(\Sigma)$.

\begin{lemma}\label{lemma_59} Assume $Q\neq \Lambda$. Then there 
exists $\delta\in\Lambda\setminus Q$ such that

\centerline{$[L_\delta,L_{-\delta}]\neq 0$ and $l(\delta)=0$.}
\end{lemma}
\begin{proof}
By Lemma \ref{lemma_12}, the set $\Pi$ generates $\Lambda$. So there
is an element $\delta\in\Pi$ with $\delta\notin Q$. Since
$\Sigma$ lies in $Q$, $\delta$ lies in $\Pi\setminus\Sigma$.
Therefore,  we have 
$[L_\delta,L_{-\delta}]\neq 0$ and $l(\delta)=0$. 
\end{proof}

Recall that  $L^\ast_\lambda$ is the element of the graded dual 
$\mathcal{L}'$ defined by 

\centerline{$<L^\ast_{\lambda}\vert L_\mu>=\delta_{\lambda,\mu}$,}

\noindent where
$\delta_{\lambda,\mu}$ is Kronecker's symbol. Similarly, denote by 
$L_n^\ast$ the dual basis of the graded dual $W'$ of  $W$.
Assume that $Q\neq \Lambda$, and let 
$\delta\in\Lambda\setminus Q$ be the element defined in the previous
lemma.

\begin{lemma}\label{lemma_60} For each $\alpha\in \Sigma$, there 
exists a sign $\epsilon=\epsilon(\alpha)$ such that

\centerline{$L_{n\alpha}.L^\ast_{\epsilon\delta}=0$ and
$L_{n\alpha}.L^\ast_{-\epsilon\delta}\neq 0$}

\noindent for all non-zero integer $n$.
\end{lemma}

\begin{proof}
Set
$\mathcal{L}(\alpha)=\oplus_{n\in\Z}\,\mathcal{L}_{n\alpha}$ and
$\mathcal{M}(\pm\delta)=\oplus_{n\in\Z}\,\mathcal{L}_{\pm\delta+n\alpha}$.
 We have $[\mathcal{M}(\delta),\mathcal{M}(-\delta)]\subset 
\mathcal{L}(\alpha)$, therefore the Lie bracket induces a
$\mathcal{L}(\alpha)$-equi\-va\-riant  bilinear map:

\centerline {$\pi:\mathcal{M}(\delta)\times\mathcal{M}(-\delta)\rightarrow\mathcal{L}(\alpha)$.}

By Lemma \ref{lemma_46}, the Lie algebra $\mathcal{L}(\alpha)$ is isomorphic to
$W$, and therefore $\mathcal{M}(\delta)$ and $\mathcal{M}(-\delta)$ belong to
$\mathcal{S}(W)$.

Since  we have  $\Sigma\subset Q$,
it follows  that $\delta+\alpha\notin Q$ and therefore
we get $\delta+\alpha\notin \Sigma$.
Since $l(\delta+\alpha)=l(\alpha)\neq 0$, it follows that
$\pi(L_{\delta+\alpha},L_{-\delta-\alpha})=0$.  
Using again that $l(\delta+\alpha)\neq 0$,
$\pi$ satisfies Assertion (ii) or
Assertion (iii) of Lemma \ref{lemma_55}.

First, assume that $\pi$ satisfies Assertion (ii) of Lemma \ref{lemma_55}.
Thus there are a surjective morphism
$\phi:\mathcal{M}(\delta)\rightarrow \C$ and an isomorphism
$\psi:\mathcal{M}(-\delta)\rightarrow \mathcal{L}(\alpha)$ such that
$\pi(m, n)=\phi(m)\psi(n)$ for all 
$(m,n)\in\mathcal{M}(\delta)\times\mathcal{M}(-\delta)$.
Since $\phi$ is proportional to
$L_{\delta}^\ast$, it follows that $L_{\delta}^\ast$ is $\mathcal{L}(\alpha)$-invariant. In particular, we have $L_{n\alpha}.L^\ast_{\delta}=0$
for all integer $n$. Moreover, $\psi$ induces an isomorphism
$\mathcal{M}(-\delta)'\simeq W'$ under which
$L_{-\delta}^\ast$ is a non-zero multiple of
$L^\ast_0$. Since $L_n.L^\ast_0=2n L^\ast_{-n}$, we have
$L_{n\alpha}.L^\ast_{-\delta}\neq 0$
for all non-zero integer $n$. 
Hence, if $\pi$ satisfies Assertion (ii) of the Lemma,
then the Lemma holds for $\epsilon=1$.

Otherwise, $\pi$ satisfies Assertion (iii), and the same
proof shows that the Lemma holds
for  $\epsilon=-1$. 
\end{proof}

Assume again that $Q\neq \Lambda$. Let 
$\delta\in\Lambda \setminus Q$ be the element defined in Lemma \ref{lemma_59}
and let 
$\epsilon:\Sigma\rightarrow \{\pm 1\},
\alpha\mapsto \epsilon(\alpha)$ be the function defined by the
previous lemma.

\begin{lemma}\label{lemma_61} 
The function $\epsilon:\Sigma\rightarrow \{\pm 1\}$
is constant.
\end{lemma}
\begin{proof}
Let $S_{\pm}$ be the set of all $\alpha\in \Sigma$
with $\epsilon(\alpha)=\pm 1$. We claim that
the sets $S_\pm$ are quasi-additive. 

For simplicity, let consider $S_+$.
Let $\alpha,\beta\in S_+$ with $l(\alpha+\beta)\neq 0$. As before,
set $\mathcal{L}(\alpha)=\oplus_{n\in\Z}\,\mathcal{L}_{n\alpha}$ and
 $\mathcal{M}(\beta)=\oplus_{n\in\Z}\,\mathcal{L}_{\beta+n\alpha}$.
Also let  $\mathcal{N}$ be the $\mathcal{L}(\alpha)$-module generated by
$L_\beta$. Since  $\mathcal{L}(\alpha)$ is isomorphic to $W$,
the module $\mathcal{M}(\beta)$ belongs to $\mathcal{S}(W)$. Since
$[L_0,L_\beta]\neq 0$, the module $\mathcal{N}$ is not trivial.
Therefore, $\mathcal{N}=\mathcal{M}(\beta)$ or $\mathcal{M}(\beta)/\mathcal{N}
\simeq \C$. It follows that $\mathcal{N}$ contains $L_{\alpha+\beta}$.

 By definition of $S_+$, $L^\ast_{\delta}$ is $\mathcal{L}(\alpha)$-invariant.
Since it is also invariant by $L_\beta$,  we have
$\mathcal{N}.L^\ast_\delta=0$. In particular, we have
$L_{\alpha+\beta}.L^*_\delta=0$. This implies that
$\epsilon(\alpha+\beta)=1$, i.e. $\alpha+\beta\in S_+$.
So $S_+$ satisfies Condition (iii) of the definition of quasi-additivity.
Since the other conditions are obvious, $S_+$ is quasi-additive.

Similarly $S_-$ is quasi-additive. Since 
$\Sigma=S_+\cup S_-$, it follows from Lemma \ref{lemma_57}
that $\Sigma=S_+$ or $\Sigma=S_-$. So $\epsilon(\alpha)$ is
independant of $\alpha$ and the lemma is proved.
\end{proof}

\begin{lemma}\label{lemma_62} (Main Lemma non-integrable Lie algebras) 
Let $\mathcal{L}\in\mathcal{G}$ be a non-integrable  Lie algebra. Then 
$\Sigma$ is the set of all
$\lambda\in\Lambda$ such that 
$l(\lambda)\neq 0$.
\end{lemma}
\begin{proof}
Assume otherwise. By Lemma \ref{lemma_56}, the subgroup
$Q$ generated by $\Sigma$ is proper. 
Let  $\delta\in\Lambda\setminus Q$ be the element of Lemma \ref{lemma_59}.
By Lemma \ref{lemma_60}, the function $\epsilon:\Sigma\rightarrow \{\pm 1\}$
is constant. 
In order to get a contradiction, we can assume that 
$\epsilon(\gamma)=1$ for all $\gamma\in\Sigma$.

Set $\mathcal{A}=\oplus_{l(\lambda)=0}\,\mathcal{L}_\lambda$
and $\mathcal{B}=\oplus_{l(\lambda)\neq0}\,\mathcal{L}_\lambda$.
 We have $\mathcal{L}=\mathcal{A}\oplus\mathcal{B}$ and
$[\mathcal{A},\mathcal{B}]\subset \mathcal{B}$. By Lemma \ref{lemma_3},
$\mathcal{B} +[\mathcal{B},\mathcal{B}]$ is an ideal, and therefore
$\mathcal{L}=\mathcal{B} +[\mathcal{B},\mathcal{B}]$. Thus
$\mathcal{A}\subset [\mathcal{B},\mathcal{B}]$.

Since $l(\delta)=0$, we have $L_\delta\in[\mathcal{B},\mathcal{B}]$.
Therefore there are $\gamma_1,\gamma_2\in\Lambda$ with
$l(\gamma_1)\neq 0$, $l(\gamma_2)\neq 0$ such that
$[L_{\gamma_1},L_{\gamma_2}]$ is a non-zero multiple of $L_\delta$.
Since $[L_\delta, L_{-\delta}]\neq 0$,
$[[L_{\gamma_1},L_{\gamma_2}], L_{-\delta}]$ is a non-zero multiple
of $L_0$. Thus
$[L_{\gamma_1},[L_{\gamma_2}, L_{-\delta}]]$
or $[L_{\gamma_2},[L_{\gamma_1}, L_{-\delta}]]$ is a non-zero multiple
of $L_0$. By symmetry of the role of  $\gamma_1$ and of $\gamma_2$, we can
assume that $[L_{\gamma_1},[L_{\gamma_2}, L_{-\delta}]]=cL_0$, where $c$
is not zero. Thus $\gamma_1$ belongs to $\Sigma$.

Since $[L_{\gamma_1},L_{\gamma_2}]$ is a non-zero multiple
of $L_\delta$, we get $<L_\delta^\ast\vert[L_{\gamma_1},L_{\gamma_2}]>
\neq 0$, and therefore $L_{\gamma_1}.L_\delta^\ast\neq 0$. 
Thus we have $\epsilon(\gamma_1)=-1$, which contradicts the hypothesis 
$\epsilon(\gamma)=1$ for all $\gamma\in\Sigma$.
\end{proof}

For any $\lambda\in\Lambda$, set
$\Omega^\ast(\lambda)=\Supp \mathcal{L}.L_\lambda^\ast$. As a corollary of
the Main Lemma, we get:

\begin{lemma}\label{lemma_63} For any $\lambda\in\Lambda$, there is
a finite set $F$ such that $\Lambda=F+\Omega^\ast(\lambda)$.
\end{lemma}
\begin{proof}
By the Main Lemma \ref{lemma_62}, we have
$\Omega^\ast(0)\supset\{\mu\in\Lambda\vert l(\mu)\neq 0\}$. So we have 
$\Lambda=\Omega^\ast(0)\cup\,\alpha+\Omega^\ast(0)$, where
$\alpha$ is any element with $l(\alpha)\neq 0$.

However by Lemma \ref{lemma_4}, we have
$\Omega^\ast(\lambda)\equiv\Omega^\ast(0)$.
Therefore we have
$\Lambda=F+\Omega^\ast(\lambda)$,
for some a finite subset $F$ of $\Lambda$.
\end{proof}

\section{Local Lie algebras of rank $2$}\label{sect_15}

The aim of this section is Lemma \ref{lemma_66}, 
i.e. the fact that some local Lie algebras do
not  occur in a  Lie algebra $\mathcal{L}\in\mathcal{G}$.

First, start with some definitions. Let $L$, $S$  be Lie algebras.
The Lie algebra $S$ is called a \textit{section} of $L$ if there 
exists two Lie subalgebras $G$ and $R$ in $L$, such that
$R$ is an ideal of $G$ and $G/R\simeq S$.

Let $\mathfrak{a}$ be another Lie algebra. The relative notions
of a  Lie $\mathfrak{a}$-algebra, an $\mathfrak{a}$-subalgebra,
an $\mathfrak{a}$-ideal and an $\mathfrak{a}$-section are defined
as follows.
A \textit{Lie $\mathfrak{a}$-algebra} is a Lie algebra $L$ on which 
$\mathfrak{a}$ acts by
derivation. An \textit{$\mathfrak{a}$-subalgebra}  (respectively
an $\mathfrak{a}$-ideal ) of a  Lie $\mathfrak{a}$-algebra $L$
is a subalgebra (respectively an ideal) $G$ of $L$ which is
stable by $\mathfrak{a}$. Let  $L$, $S$  be 
Lie $\mathfrak{a}$-algebras. The Lie algebra $S$ is called an 
\textit{$\mathfrak{a}$-section} of $L$ if there 
exist an $\mathfrak{a}$-subalgebra $G$  of $L$, 
and an $\mathfrak{a}$-ideal  $R$ of $G$ such that
$G/R\simeq S$ as a  Lie $\mathfrak{a}$-algebra.

Following V.G. Kac \cite{Ka1}, a \textit{local Lie algebra} is a quadruple

\noindent $V=(\g,V^+, V^-,\pi)$, where:
\begin{enumerate}
\item[(i)] $\g$ is a Lie algebra and  $V^\pm$ are $\g$-modules
\item[(ii)]  $\pi: V^+\times V^-\rightarrow \g$ is a
$\g$-equivariant  bilinear map.
\end{enumerate}
As for Lie algebras, there are obvious notions of
\textit{local subalgebras}, \textit{local ideals} of $V$
and \textit{local sections} of $V$. By definition
a \textit{local Lie $\mathfrak{a}$-algebra}
is a local Lie algebra $V=(\g,V^+, V^-,\pi)$ such that  $\g$, $V^+$ and
$V^-$  are $\mathfrak{a}$-modules and all the  
products of the local structure are 
$\mathfrak{a}$-equivariant. As for Lie algebras, one defines the
notion of a  \textit{local $\mathfrak{a}$-subalgebra},
of a  \textit{local $\mathfrak{a}$-ideal} and of a 
\textit{local $\mathfrak{a}$-section} of $V$.

Let $\mathcal{L}=\oplus_{n\in \Z}\,\mathcal{L}_n$ be 
a  weakly $\Z$-graded Lie algebra. The subspace 
$\mathcal{L}_{loc}=\mathcal{L}_{-1}\oplus
\mathcal{L}_{0}\oplus\mathcal{L}_{1}$ is called its \textit{local part}.
It is clear that $\mathcal{L}_{loc}$ carries a structure of
a local Lie algebra. Indeed
$\g=\mathcal{L}_0$ is a Lie algebra,  $V^\pm=\mathcal{L}_{\pm1}$ are
$\mathcal{L}_0$-modules and the Lie bracket induces a bilinear map
$\pi: \mathcal{L}_{1}\times
\mathcal{L}_{-1}\rightarrow\mathcal{L}_{0}$. By definition, the Lie algebra
$\mathcal{L}$ is \textit{associated to the local Lie algebra $V$} if 
$\mathcal{L}$ is generated by its local part and if 
$\mathcal{L}_{loc}\simeq V$. 

Given a local Lie algebra $V$, there are two associated
Lie algebras $\mathcal{L}^{max}(V)$ and $\mathcal{L}^{min}(V)$
which satisfies the following conditions. For any
Lie algebra $\mathcal{L}$ associated to $V$,
there are morphisms of Lie algebras
$\mathcal{L}^{max}\rightarrow \mathcal{L}$ and 
$\mathcal{L}\rightarrow \mathcal{L}^{min}$. Here it should be understood that
the local parts of these morphisms are just the identity.

Let $F(V^{\pm})$ be the free Lie algebra generated by
$V^{\pm}$. Indeed it is shown in \cite{Ka1} that 
$F(V^+)\oplus \g\oplus F(V^{-})$ has a natural structure of a Lie algebra.
It follows that $\mathcal{L}^{max}(V) =F(V^+)\oplus \g\oplus F(V^-)$.
The weak $\Z$-gradation of the Lie algebra
$\mathcal{L}:=\mathcal{L}^{max}(V)$ is described as follows: $\mathcal{L}_0=\g$, 
$\mathcal{L}_{\pm1}=V^{\pm}$, so $F(V^+)$ is positively weakly graded and
$F(V^-)$ is negatively weakly graded.

Let $R$ be the maximal graded ideal of $\mathcal{L}^{max}(V)$ with
the property that its local part
$R_{-1}\oplus R_0\oplus R_{1}$ is zero. Then we have
$\mathcal{L}^{min}(V)=\mathcal{L}^{max}(V)/R$.

Let $\mathfrak{a}$ be an auxiliary Lie algebra. It should be
noted that $\mathcal{L}^{max}(V)$ and $\mathcal{L}^{min}(V)$ are
 Lie $\mathfrak{a}$-algebras whenever $V$ is a
local Lie $\mathfrak{a}$-algebra.

\begin{lemma}\label{lemma_64} Let $\mathcal{L}$ be a weakly $\Z$-graded algebra,
and let $V$ be an $\mathcal{L}_0$ local Lie algebra.
 If $V$ is an $\mathcal{L}_0$-section of $\mathcal{L}_{loc}$, then 
$\mathcal{L}^{min}(V)$ is an $\mathcal{L}_0$-section of $\mathcal{L}$.
\end{lemma}

\begin{proof}
By definition, there is a 
local $\mathcal{L}_0$-subalgebra $U$ of $\mathcal{L}_{loc}$
and a  local $\mathcal{L}_0$-ideal $R$ of $U$ such that
$U/R\simeq V$ as a  local Lie $\mathcal{L}_0$-algebra.
Let $\mathcal{L}(U)$ be the subalgebra of $\mathcal{L}$ generated
by $U$.

By definition, there is a surjective morphism
$\mathcal{L}^{max}(U)\rightarrow \mathcal{L}(U)$ which is the identity on
the local part.
Since the positively graded part and the negatively graded part
of $\mathcal{L}^{max}(V)$ are free Lie algebras,  the local map
$U\rightarrow V$ extends to  a surjective map
$\mathcal{L}^{max}(U)\rightarrow \mathcal{L}^{min}(V)$. It follows that
$\mathcal{L}(U)=\mathcal{L}^{max}(U)/I$ and
$\mathcal{L}^{min}(V)=\mathcal{L}^{max}(U)/J$, where $I$ and $J$ are
graded ideals. We have $I_{-1}\oplus I_{o}\oplus I_{1}=0$,
and $J$ is the maximal graded ideal such that
$J_{-1}\oplus J_{o}\oplus J_{1}\subset R$.

It follows that $I\subset J$. Thus 
$\mathcal{L}^{min}(V)$ is the quotient of the
$\mathcal{L}_0$-subalgebra $\mathcal{L}(U)$ by its
$\mathcal{L}_0$-ideal $J/I$. Hence
$\mathcal{L}^{min}(V)$ is an $\mathcal{L}_0$-section of $\mathcal{L}$. 
\end{proof}

Recall that the tensor densities $W$-module $\Omega_s^\delta$
have been defined in Section \ref{sect_12}. Its elements  are symbols 
$\sigma(f\partial^{-\delta})$, where 
$f\in  z^{s-\delta}\C[z,z^{-1}]$. To simplify the notations, this
symbol will be denoted by  $f\partial^{-\delta}$. 
Let $\delta,\,\eta,\,s,\,t$ be four scalars.
Define three maps:

\centerline{$\pi:\Omega_s^\delta\otimes\Omega_t^\eta
\rightarrow\Omega_{s+t}^{\delta+\eta}$,
$f\partial^{-\delta}\otimes g\partial^{-\eta}
\mapsto fg\partial^{-\delta-\eta}\,$,}

\centerline{$\beta_1:\Omega_s^\delta\otimes\Omega_t^\eta
\rightarrow\Omega_{s+t}^{\delta+\eta+1}$,
$f\partial^{-\delta}\otimes g\partial^{-\eta}
\mapsto fg'\partial^{-\delta-\eta-1}\,$,}

\centerline{$\beta_2:\Omega_s^\delta\otimes\Omega_t^\eta
\rightarrow\Omega_{s+t}^{\delta+\eta+1}$,
$f\partial^{-\delta}\otimes g\partial^{-\eta}
\mapsto f'g\partial^{-\delta-\eta-1}$.}

\noindent The map $\pi$ is the product of symbols $P^{\delta,\eta}_{s,t}$. Denote
by $\mathcal{K}$ its kernel. Since $\pi$ is a morphism
of $W$-modules, $\mathcal{K}$ is a $W$-submodule.

\begin{lemma}\label{lemma_65} 
\begin{enumerate}
\item[(i)] We have $\beta_1(\omega)+\beta_2(\omega)=0$ for any
$\omega\in\mathcal{K}$.
\item[(ii)] The restriction of $\beta_1$ to $\mathcal{K}$ is surjective,
and it is a morphism of $W$-modules.
\end{enumerate}
\end{lemma}
\begin{proof}
For  $f\partial^{-\delta}\otimes g\partial^{-\eta}
\in  \Omega_s^\delta\otimes\Omega_t^\eta$ we have
$(\beta_1+\beta_2)(f\partial^{-\delta}\otimes g\partial^{-\eta})
=(fg)'\partial^{-\delta-\eta-1}$, and therefore
$\beta_1(\omega)+\beta_2(\omega)=0$ for any
$\omega\in\mathcal{K}$.

Note that $\eta\beta_2-\delta\beta_1$ is the Poisson bracket
$B^{\delta,\eta}_{s,t}$ of
symbols, thus $\eta\beta_2-\delta\beta_1$ is a morphism of
$W$-modules. Since $\beta_1+\beta_2=0$ on $\mathcal{K}$, the restriction
of $(\delta+\eta)\beta_1$ to $\mathcal{K}$ is a morphism of $W$-module.
Thus $\beta_1\vert_\mathcal{K}$ is a $W$-morphism when 
$\delta+\eta\neq 0$. By extension of polynomial identities, 
it is always true that 
$\beta_1\vert_\mathcal{K}$ is a morphism of $W$-modules.

Let $w\in (s+t-\delta-\eta)+\Z$.
Choose $a\in (s-\delta)+\Z$, $b\in (t-\eta)+\Z$ with 
$a+b=w$. Set
$\omega=z^a\partial^{-\delta}\otimes z^b\partial^{-\eta}-
z^{a+1}\partial^{-\delta}\otimes z^{b-1}\partial^{-\eta}$.
Then $\omega$ belongs to $\mathcal{K}$ and 
$\beta_1(\omega)=z^w\partial^{-\delta-\eta-1}$.
Since the symbols $z^w\partial^{-\delta-\eta-1}$ form
a basis  $\Omega_{s+t}^{\delta+\eta+1}$, the restriction
of $\beta_1$ to $\mathcal{K}$ is onto. 
\end{proof}

Let $V=(\g ,V^+,V^-)$ be a local Lie algebra
with the following properties:
$\g=\oplus_x\g_x$ is a $\C$-graded Lie algebra,
$V^\pm=\oplus_x\,V^\pm_x$ are $\C$-graded
$\g$-modules and the bilinear map $\pi:V^+\times V^-\rightarrow\g$
is homogeneous of degree zero. Then the
Lie algebras  $\mathcal{L}^{max}(V)$ and $\mathcal{L}^{min}(V)$ are naturally
weakly $\Z\times\C$-graded. Set $\mathcal{L}=\mathcal{L}^{min}(V)$
and denote by $\mathcal{L}=\oplus\,\mathcal{L}_{n,x}$ the corresponding
decomposition. 
With the previous notations, let $\mathcal{F}_{loc}$ be the class of $\C$-graded local Lie algebras $V$ such that

\centerline{$\dim\,\mathcal{L}_{n,x}\leq 1$, for any $(n,x)\in\Z\times\C$.}

Since $W\simeq \Omega_0^{-1}$,
the commutative product on
$\mathcal{P}$ induces a $W$-equi\-va\-riant bilinear map
$\pi_s^{-\delta}:\Omega_s^{-\delta}\times
\Omega_{-s}^{\delta-1}\rightarrow W$.

\begin{lemma}\label{lemma_66} Assume that the local Lie algebra
$(W, \Omega_s^{-\delta}, \Omega_{-s}^{\delta-1},\pi_s^{-\delta})$
is in $\mathcal{F}_{loc}$. Then $\delta=-1$ or $\delta=2$.
\end{lemma}

\begin{proof}
For clarity, the proof is divided into four steps.
It is assumed, once for all, that  $\delta\neq-1$ and $\delta\neq 2$.

\textit{Step 1:}
Let $V$ be any local Lie algebra.
Set $\mathcal{L}= \mathcal{L}^{max}(V)$ and let $R$ be the kernel of the
morphism $\mathcal{L}\rightarrow \mathcal{L}^{min}(V)$.

In order to compute by induction, for $n\geq 1$, the
homogeneous components $\mathcal{L}_n/R_n$ of $\mathcal{L}^{min}(V)$,
it should be noted that:
\begin{enumerate}
\item[(i)] $\mathcal{L}_{\geq 1}:=\oplus_{n\geq 1}\,\mathcal{L}_n$ is the Lie algebra
freely generated by
$V^+=\mathcal{L}_1$, and therefore 
$\mathcal{L}_{n+1}=[\mathcal{L}_{1},\mathcal{L}_{n}]$ for all $n\geq 1$,
\item[(ii)] $R_1=0$ and  $R_{n+1}=\{x\in \mathcal{L}_{n+1}\vert\,
[\mathcal{L}_{-1},x]\subset R_n\}$ for any $n\geq 1$ . 
\end{enumerate}
More precisely, the following procedure will be used.
Assume by induction that: 
\begin{enumerate}
\item[(i)] the $\mathcal{L}_0$-modules $\mathcal{L}_i/R_i$,
\item[(ii)] the brackets 
$[~,~]:\mathcal{L}_1\times \mathcal{L}_{i-1}\rightarrow \mathcal{L}_i$,
\item[(iii)]  the brackets 
$[~,~]:\mathcal{L}_{-1}\times \mathcal{L}_i\rightarrow \mathcal{L}_{i-1}$
\end{enumerate}
have been determined
for all $1\leq i\leq n$. 

Thus define the bilinear map:
$B_n:\mathcal{L}_{-1}\times \mathcal{L}_1\otimes\mathcal{L}_n/R_n
\rightarrow \mathcal{L}_n/R_n$ by the formula
$B_n(x,y\otimes z)=[[x,y],z]+[y,[x,z]]$,
for any $x\in\mathcal{L}_{-1}, y\in\mathcal{L}_1$ and 
$z\in\mathcal{L}_{n}/R_n$. Let 
$K\subset \mathcal{L}_1\otimes\mathcal{L}_n/R_n$ 
be the right kernel of $B_n$. It follows from Jacobi identity that
the following diagram is commutative:

\[ \UseTips
\newdir{ >}{!/-5pt/\dir{>}}
\xymatrix @=1pc @*[r]
{
    \mathcal{L}_{-1}\times \mathcal{L}_1 \otimes \mathcal{L}_n/R_n 
    \ar[dd]^{\id \times [~,~]}
    && \ar[ddrr]^{B_n}
    &&&& \\
    &&&& \\
    \mathcal{L}_{-1}\times \mathcal{L}_{n+1}/R_{n+1}  
    &&\ar[rr]^{[~,~]} && 
    \mathcal{L}_n/R_n  \\
} \]

By definition of $\mathcal{L}^{min}(V)$, the right kernel
of the horizontal bilinear  map:
$[,]: \mathcal{L}_{-1}\times\mathcal{L}_{n+1}/R_{n+1}
\rightarrow \mathcal{L}_n/R_n$ is zero. Moreover the natural
map $[,]: \mathcal{L}_{1}\otimes\mathcal{L}_{n}/R_{n}
\rightarrow \mathcal{L}_{n+1}/R_{n+1}$ is onto.
Therefore, we have:

$$\mathcal{L}_{n+1}/R_{n+1} = (\mathcal{L}_1\otimes\mathcal{L}_n/R_n)/K$$

\noindent This isomorphism determines the structure of
$\mathcal{L}_0$-module of $\mathcal{L}_{n+1}/R_{n+1}$ as well as the 
bracket $[,]:\mathcal{L}_1\times \mathcal{L}_{n}\rightarrow \mathcal{L}_{n+1}$.
The bracket $[,]:\mathcal{L}_{-1}\times \mathcal{L}_{n+1}\rightarrow 
\mathcal{L}_{n}$ comes from $B_n$.

\textit{Step 2:} The local Lie algebras $V:=(W, \Omega_s^{-\delta},
\Omega_{-s}^{\delta-1},-\pi_s^{-\delta})$ and $(W, \Omega_s^{-\delta},
\Omega_{-s}^{\delta-1},\\ \pi_s^{-\delta})$ are obviously isomorphic.
In order to minimize the use of minus sign, it will be more convenient to 
work with $V$. \\
Set 
$\mathcal{L}=\mathcal{L}^{max}(V)$ and set $A=\C[z,z^{-1}]$,
 $A_k=z^{k(s+\delta)}A$ for any $k\in\Z$. We will use the map $B_1$
of step 1 to compute $\mathcal{L}_2/R_2$.

The elements of 
$\mathcal{L}_1$ (respectively of $\mathcal{L}_{-1}$) are symbols
$f\partial^\delta$ (respectively $g\partial^\gamma$), where
$f\in A_1$ (respectively where $g\in A_{-1}$) and where
$\gamma=1-\delta$.  In the local Lie
algebra $V$, we have:

\centerline{$[g\partial^\gamma, f\partial^\delta]=
\pi(f\partial^\delta, g\partial^\gamma)=fg\partial\in W$.}

\noindent 
Let $f\partial^\delta,\,g\partial^\delta\in \mathcal{L}_1$ and let
$h\partial^\gamma\in\mathcal{L}_{-1}$. It follows that:
\begin{align*}
B_1(h\partial^\gamma,f\partial^\delta\otimes
g\partial^\delta)
=
&[hf\partial,g\partial^\delta]+[f\partial^\delta,hg\partial] \\
=
&(hfg' -\delta g(hf)'+\delta f(hg)'-f'hg)\partial^\delta \\
=
&(1+\delta) h(fg'-f'g)\partial^\delta
\end{align*}
Define  
$\beta: \mathcal{L}_1\otimes \mathcal{L}_1\rightarrow \Omega^{1-2\delta}_{2s}$
by the formula $\beta (f\partial^\delta\otimes\,g\partial^\delta)
=(1+\delta) (fg'-f'g)\partial^{2\delta-1}$. In terms
of Poisson brackets, we have $\beta(b\otimes c)=(1+\delta)/\delta
\{b,c\}$ for any $b,c\in\Omega^\delta_s$, and therefore $\beta$ is a
morphism of $W$-modules (for $\delta=0$, this follows by extension of a polynomial identity).
Moreover it is easy to show  that $\beta$ is surjective.

Since $\delta\neq -1$,
it is clear from  the previous formula that the right
kernel of $B_1$ is precisely the kernel of $\beta$. So we get
$\mathcal{L}_2/R_2\simeq\Omega^{1-2\delta}_{2s}$.

One can choose such an isomorphism in a way that:

$[f\partial^\delta,\,g\partial^\delta]=
(1+\delta) (fg'-f'g)\partial^{2\delta-1}$,

$[h\partial^\gamma,k\partial^{2\delta-1}]=
hk\partial^\delta$,

\noindent for any
$f\partial^\delta,\,g\partial^\delta\in \mathcal{L}_1$,
$h\partial^\gamma\in\mathcal{L}_{-1}$ and
$k\partial^{1-2\delta}\in \mathcal{L}_2/R_2$.

\textit{Step 3:} We will use the map $B_2$
of step 1 to compute $\mathcal{L}_3/R_3$.

Let $f\partial^\delta\in\mathcal{L}_1$,
$k\partial^{2\delta-1}\in \mathcal{L}_2/R_2$ and
let $h\partial^\gamma\in\mathcal{L}_{-1}$.
We get
\begin{align*}
B_2(h\partial^\gamma,f\partial^\delta
\otimes k\partial^{2\delta-1})
=
&[hf\partial,k\partial^{2\delta-1}]+
[f\partial^\delta,hk\partial^{\delta}] \\
=
&(hfk'-(2\delta-1)(hf)'k+
(1+\delta)(f(hk)'-f'hk))\partial^{2\delta-1}\\
=
&((2-\delta) h'fk-3\delta hf'k
+(2+\delta)hfk')\partial^{2\delta-1} \\
=
&((2-\delta) fk)h'\partial^{2\delta-1} +
(-3\delta f'k +(2+\delta)fk')h\partial^{2\delta-1}.
\end{align*}
Since $\delta\neq 2$, it is clear that the right kernel
of $B_2$ is 

\centerline{$\Ker \pi\cap\Ker ((2+\delta)\beta_1-3\delta\beta_2)$,}

\noindent where 
$\pi:\Omega_s^{-\delta}\otimes\Omega_{2s}^{1-2\delta}
\rightarrow\Omega_{3s}^{1-3\delta}$ and
$\beta_1, \beta_2:\Omega_s^{-\delta}\otimes\Omega_{2s}^{1-2\delta}
\rightarrow\Omega_{3s}^{2-3\delta}$ are defined in Lemma \ref{lemma_65}.
However we have 
$((2+\delta)\beta_1-3\delta\beta_2)(\omega)=(4\delta+2)\beta_1(\omega)$
for all $\omega\in \Ker \pi$.

Assume now that $\delta\neq-1/2$, i.e. $4\delta+2\neq 0$. It follows that
$\mathcal{L}_3/R_3\simeq 
[\mathcal{L}_1\otimes\mathcal{L}_2/R_2]/\Ker \pi\cap \Ker\beta_1$
and there is an
exact sequence:

$0\rightarrow \Ker \pi/\Ker \beta_1\cap\Ker\pi\rightarrow
\mathcal{L}_3/R_3\rightarrow [\mathcal{L}_1\otimes\mathcal{L}_2/R_2]/\Ker \pi
\rightarrow 0$. Using Lemma \ref{lemma_65}, we get an exact sequence:

\centerline{$0\rightarrow \Omega_{3s}^{2-3\delta}\rightarrow
\mathcal{L}_3/R_3\rightarrow\Omega_{3s}^{1-3\delta}\rightarrow 0$.}

\noindent Thus, the homogeneous components of 
$\mathcal{L}_3/R_3$ have dimension two.
It follows that the local Lie
algebra $V$ is not in $\mathcal{F}_{loc}$ if $\delta\neq -1/2$.

\textit{Step 4:} Assume now that $\delta=-1/2$. The opposite
local Lie algebra is
$V'=(W, \Omega_{-s}^{-3/2},
\Omega_{s}^{1/2},\pi_{-s}^{-3/2})$. It follows from the
previous step that the homogeneous
components of
$\mathcal{L}_{-3}/R_{-3}$ have dimension two.
Thus the local Lie
algebra $V$ is not in $\mathcal{F}_{loc}$. 
\end{proof}

\section{The degree function $\delta$}\label{sect_16}

Let $\mathcal{L}\in\mathcal{G}$ be non-integrable. 

By Lemma \ref{lemma_62}, we have $\Sigma=\{\alpha\vert\,l(\alpha)\neq 0\}$.
Therefore $\Sigma$ contains primitive elements of $\Lambda$. So fix, once
for all, a primitive  $\alpha\in\Lambda$ which is in 
$\Sigma$. Set 
$\mathcal{L}(\alpha)=\oplus_{n\in\Z}\,\mathcal{L}_{n\alpha}$
and, for any $\beta\in\Lambda/\Z\alpha$, set $\mathcal{M}(\beta)=
\oplus_{n\in\Z}\,\mathcal{L}_{\beta+n\alpha}$.

By Lemma \ref{lemma_46}, the Lie algebra $\mathcal{L}(\alpha)$ is isomorphic 
to $W$. Since
$\mathcal{M}(\beta)$ belongs to the class
$\mathcal{S}(W)$, the map  $\beta\in\Lambda/\Z\alpha\mapsto \deg \mathcal{M}(\beta)$ is  a multivalued function. In this section, Lemma \ref{lemma_70} is
used to define an ordinary function
$\delta:\Lambda/\Z\alpha\rightarrow \C$, which
has the property that $\delta(\beta)\in\deg \mathcal{M}(\beta)$
for all $\beta$. Additional properties of $\delta$ will be proved in this
section and the next one.

A weakly $\Z$-graded Lie algebra $L=\oplus_{n\in\Z}\,
L_n$ is called \textit{minimal} iff it satisfies the following conditions:
\begin{enumerate}
\item[(i)] $L$ is generated by its local part 
$L_{loc}:=L_{-1}\oplus L_0\oplus L_1$
\item[(ii)] Any graded ideal $I$ such that
$I\cap L_{loc}=0$ is trivial.
\end{enumerate}
Equivalently, $L=\mathcal{L}^{min}(V)$ where $V$ is the local part of $L$.

For any integer $a$, the notation 
$\Z_{>a}$ has been defined in Section \ref{sect_1.1}. It will
be convenient to extend this notation for $a=-\infty$. In this case,
set $\Z_{>-\infty}=\Z$.

\begin{lemma}\label{lemma_67} 
 Let $a\in\Z\cup\{-\infty\}$ and let
$L=\oplus_{n\in\Z_{>a}}\, L_n$ be a weakly graded Lie
algebra. 
Assume that $L_n$ is a simple $L_0$-module, 
$[L_{-1},L_{n+1}]\neq 0$  and $[L_1,L_n]\neq 0$, for any $n>a$. 
Then $L$ is minimal.
\end{lemma}
\begin{proof}
Let $L'$ be the subalgebra generated by $L_{loc}$. 

Let $R$ be any non-zero graded $L'$-submodule of $L$ and let 
$n>a$. By simplicity of the $L_0$-module $L_n$, we have 
$[L_{1},L_n]=L_{n+1}$. Similarly, we have $[L_{-1}, L_n]=L_{n-1}$ 
(for $n=a+1$, this follows from $L_a=0$).
So, for $n\in\Supp R$, we have $R_n=L_n$ and $n+1 \in \Supp R$. 
Moreover if $n-1>a$, we also have 
$n-1\in\Supp R$. Thus $\Supp R=\Z_{>a}$ and
$R=L$. 

It follows that $L=L'$, i.e. $L$ is generated by its local part, and 
that $L$ is a simple  graded Lie algebra. Therefore
$L$ is minimal.
\end{proof}
Recall that $E_{-\rho}\in\mathcal{P}$ represents the symbol of $1$.
The center of the Lie algebras  
$\mathcal{P}^+$ and $\mathcal{P}$ is $\C E_{-\rho}$. Thus set
$\widetilde{\mathcal{P}}^+=\mathcal{P}^+/\C.E_{-\rho}$ and 
$\widetilde{\mathcal{P}}=[\mathcal{P},\mathcal{P}]/\C E_{-\rho}$.
Since $\mathcal{P}=[\mathcal{P},\mathcal{P}]\oplus\C E_{-2\rho}$,
the Lie algebra $\widetilde{\mathcal{P}}$ has basis
$(E_\lambda)$ when $\lambda$ runs over 
$\C^2\setminus\{-\rho,\,-2\rho\}$, and the bracket is defined
as before by: 

$[E_\lambda,E_\mu]=0$ if $\lambda+\mu=-\rho$ or $-2\rho$

$[E_\lambda,E_\mu]=<\lambda+\rho\vert\,\mu+\rho> E_{\lambda+\mu}$
otherwise, 

\noindent for any $\lambda,\,\mu\in \C^2\setminus\{-\rho,\,-2\rho\}$.

Set $Par^+=\Z_{\geq 0}\times \C/\Z$,
$Par=\C\times \C/\Z$. As a
$W$-module, there are decompositions:

\centerline {$\mathcal{P}^+=\oplus_{(\delta,u)\in Par^+}\,\Omega^{\delta}_u$
and $\mathcal{P}=\oplus_{(\delta,u)\in Par}\,\Omega^{\delta}_u$ }

\noindent Accordingly, there are  decompositions of
$\widetilde{\mathcal{P}}^+$ and $\widetilde{\mathcal{P}}$:

\centerline {
$\widetilde{\mathcal{P}}^+=\oplus_{(\delta,u)\in
Par^+}\,\widetilde{\Omega}^{\delta}_u$ and
$\widetilde{\mathcal{P}}=\oplus_{(\delta,u)\in Par}\,\widetilde{\Omega}^{\delta}_u$,
}

\noindent where:
 $\widetilde{\Omega}^{0}_0=\Omega^0_0/\C$,
$\widetilde{\Omega}^{1}_0=d\Omega^0_0\simeq \Omega^0_0/\C$, and 
$\widetilde{\Omega}^{\delta}_u=\Omega^\delta_u$ for
$(\delta,u)\neq (0,0)$ or $(1,0)$. 

In $\mathcal{P}$, we have:
 
\centerline{$\{\Omega^{\delta}_u,\Omega^{\delta'}_{u'}\}
\subset \Omega^{\delta+\delta'+1}_{u+u'}$.}

\noindent  Similarly, in
$\widetilde{\mathcal{P}}$ we have

\centerline{$\{\widetilde\Omega^{\delta}_u,\widetilde\Omega^{\delta'}_{u'}\}
\subset \widetilde\Omega^{\delta+\delta'+1}_{u+u'}$.}

Let $(\delta,u)\in Par$. Set 
$\widetilde{\mathcal{P}}^n_{\delta,u}=\tilde\Omega_{nu}^{n(\delta+1)-1}$ and
$\widetilde{\mathcal{P}}_{\delta,u}=
\oplus_{n\in\Z}\,\widetilde{\mathcal{P}}^n_{\delta,u}$. It follows that
$\widetilde{\mathcal{P}}_{\delta,u}$ is a Lie subalgebra of $\widetilde{\mathcal{P}}$.
Moreover the decomposition
$\widetilde{\mathcal{P}}_{\delta,u}=\oplus_{n\in\Z}\,\widetilde{\mathcal{P}}^n_{\delta,u}$ is a weak 
$\Z$-gradation of the Lie algebra.

Similarly, for  $u\in\C/\Z$, set
$\widetilde{\mathcal{P}}^+_n(u)=\widetilde{\Omega}_{nu}^{-n-1}$ and
$\widetilde{\mathcal{P}}^+(u)=\oplus_{n\geq -1}\,\widetilde{\mathcal{P}}^+_n(u)$.
It is clear that $\widetilde{\mathcal{P}}^+(u)$ is a weakly $\Z$-graded
subalgebra of $\widetilde{\mathcal{P}}^+$.

\begin{lemma}\label{lemma_68} 
\begin{enumerate}
\item[(i)] Let $(\delta,u)\in Par$ with $\delta\neq 0$ or $-2$.
Then the weakly $\Z$-graded Lie algebra $\widetilde{\mathcal{P}}_{\delta,u}$
is minimal,
\item[(ii)] Let $u\in\C/\Z$.
Then the weakly $\Z$-graded Lie algebra $\widetilde{\mathcal{P}}^+(u)$
is minimal.
\end{enumerate}
\end{lemma}
\begin{proof}
It should be noted that
\begin{enumerate}
\item[(i)] $\widetilde\Omega^{\delta}_u$ is a simple $W$-module for all 
$(\delta,u)\in Par$,
\item[(ii)] $\{\widetilde\Omega^{\delta}_u,\widetilde\Omega^{\delta'}_{u'}\}
\neq 0$, except when $\delta=\delta'=0$.
\end{enumerate}
It follows that the Lie algebra 
$\widetilde{\mathcal{P}}_{\delta,u}=
\oplus_{n\in\Z}\,\widetilde{\mathcal{P}}^n_{\delta,u}$
satisfies the hypothesis of the previous lemma for $a=-\infty$.
Similarly, the Lie algebra
$\widetilde{\mathcal{P}}^+(u)=\oplus_{n\geq -1}\,\widetilde{\mathcal{P}}^+_n(u)$
satisfies the hypothesis of the previous lemma for $a=-2$. 
Thus these Lie algebras are minimal. 
\end{proof}

For $\beta\in\Lambda/\Z\alpha$,  denote by $V(\beta)$ the local
Lie algebra  $\mathcal{M}(-\beta)\oplus\mathcal{L}(\alpha)\oplus 
\mathcal{M}(\beta)$.

\begin{lemma}\label{lemma_69} There exists $(\delta,s)\in Par$ such that
the local Lie algebra 
$V(\beta)$ admits $\widetilde{\mathcal{P}}_{\delta,s}^{loc}$ as a
$W$-section.
\end{lemma}

\begin{proof}
Let $\mu: \mathcal{M}(\beta)\times  \mathcal{M}(-\beta)\rightarrow 
\mathcal{L}(\alpha)\simeq W$
be the $W$-equivariant bilinear map induced by the bracket.

Choose any $\gamma\in\beta$ with
$l(\gamma)\neq 0$. By Lemma \ref{lemma_46}, $\gamma$ belongs to
$\Sigma$, hence $\mu(L_\gamma,L_{-\gamma})\neq 0$.
It follows from Lemma \ref{lemma_52} that $\mu$ is non-degenerate.
However all $W$-equivariant non-degenerate bilinear
maps $\mu:M\times N\rightarrow W$ are classified by Lemma \ref{lemma_54}.
Indeed $\mu$ is either conjugated to the
product of symbols 
$\pi_s^{\delta}:\Omega^\delta_s\times 
\Omega^{-1-\delta}_{-s}\rightarrow \Omega^{-1}_0\simeq W$,
for $\delta\neq 1$ or $-2$, 
to the Poisson bracket of symbols
$\beta^{\delta}_s: \Omega^\delta_s\times 
\Omega^{-2-\delta}_{-s}\rightarrow \Omega^{-1}_0\simeq W$,
for $(\delta,s)\neq (0,0)$ or $(-2,0)$
or $\mu$ is in the complementary family.

By Lemma \ref{lemma_64}, $\mathcal{L}^{min}(V(\beta))$ is a section of 
$\mathcal{L}$.
Therefore $V(\beta)$ should belong to the class
$\mathcal{F}_{loc}$.
It follows from Lemma \ref{lemma_66} that $\mu$ cannot be proportional
to the map $\pi_s^{\delta}$, for $\delta\neq 1$ or $-2$.

When $\mu$ is proportional to $\beta^{\delta}_s$, then
$V(\beta)$ is  isomorphic to $\widetilde{\mathcal{P}}_{\delta,s}^{loc}$
if $(\delta, s)\neq (0,0),\,(-2,0),\,(1,0),\,(-3,0)$.
Assume now  that  $\mu$ is proportional to $\beta^{1}_0$ (respectively
$\beta^{-3}_0$). Note that $\tilde\Omega^1_0=d\Omega^0_0$
is a codimension one $W$-submodule of $\Omega^1_0$, therefore
$\widetilde{\mathcal{P}}_{1,0}^{loc}$ 
(respectively $\widetilde{\mathcal{P}}_{-3,0}^{loc}$) is a codimension one ideal 
of $V(\beta)$. 
Thus $V(\beta)$ admits
a $W$-section isomorphic to some $\widetilde{\mathcal{P}}_{\delta,s}^{loc}$,
whenever $\mu$ is proportional to some $\beta^{\delta}_s$,
with $(\delta,s)\neq (0,0)$ or $(-2,0)$.

When $\mu$ is in the left complementary family, the $W$-module
$\mathcal{M}(\beta)$ belongs to the AB-family and
then there is an exact sequence 

\centerline{$0\rightarrow \widetilde\Omega^0_0\rightarrow \mathcal{M}(\beta)
\rightarrow \C\rightarrow 0$, or 
$0\rightarrow \C\rightarrow \mathcal{M}(\beta)
\rightarrow \widetilde\Omega^0_0\rightarrow 0$.}

\noindent In the first case, $V(\beta)$ has a codimension one ideal
isomorphic to $\widetilde{\mathcal{P}}_{0,0}^{loc}$. In the second case, 
the trivial submodule of $\mathcal{M}(\beta)$ is the center
$\mathfrak{Z}$ of the local Lie algebra $V(\beta)$,
and $V(\beta)/\mathfrak{Z}$ is isomorphic to
$\widetilde{\mathcal{P}}_{0,0}^{loc}$. In both cases, $\widetilde{\mathcal{P}}_{0,0}^{loc}$
is a $W$-section of $V(\beta)$. 

Similarly when $\mu$ is in the right complementary family,
$\widetilde{\mathcal{P}}_{-2,0}^{loc}$ is a $W$-section of
 $V(\beta)$.

Thus the lemma is proved in all cases.
\end{proof}

\begin{lemma}\label{lemma_70} 
Let $\beta\in\Lambda/\Z\alpha$.
There exists a unique scalar $\delta(\beta)$ such that
\begin{enumerate}
\item[(i)] $\delta(\beta)\in \deg \mathcal{M}(\beta)$, and
\item[(ii)] $-\delta(\beta)-2\in \deg \mathcal{M}(-\beta)$.
\end{enumerate}
\end{lemma}
\begin{proof}
Let $\beta\in\Lambda/\Z\alpha$.
Note that for $\beta=0$, then $\delta(0)=-1$ is 
the unique degree of $\mathcal{M}(0)\simeq W$ and it satisfies (ii).
From now on, it can be  assumed that $\beta\neq 0$.

First prove the existence of $\delta(\beta)$.

By Lemma \ref{lemma_69}, 
there exists $(\delta,u)\in Par$ such that
the local Lie algebra $V(\beta)$ admits 
$\widetilde{\mathcal{P}}_{\delta,u}^{loc}$ as $W$-section.

Hence the $W$-module $\widetilde\Omega^{\delta}_u$ is a subquotient of
$\mathcal{M}(\beta)$ and the $W$-module $\widetilde\Omega^{-\delta-2}_{-u}$ is a
subquotient of
$\mathcal{M}(-\beta)$. Thus
$\delta\in\deg \mathcal{M}(\beta)$ and
$-\delta-2\in\deg \mathcal{M}(-\beta)$.
The scalar $\delta(\beta)=\delta$ satisfies (i) and (ii), and
the existence is proved.

Next prove the unicity. If $\deg \mathcal{M}(\beta)$ is single
valued, then $\delta(\beta)$ is uniquely determined. Assume
otherwise. Then $\deg \mathcal{M}(\beta)=\{0,1\}$.
It follows from the
previous point that $-2$ belongs to
$\deg \mathcal{M}(\beta)+\deg \mathcal{M}(-\beta)$. Therefore 
$-2$ or $-3$ is a degree of $\mathcal{M}(-\beta)$. So
$\deg \mathcal{M}(-\beta)$ is single valued, and $-2-\delta(\beta)$ is
uniquely determined. In both case, $\delta(\beta)$ is uniquely determined.
\end{proof}

It follows from the previous lemma that there is a well determined
function $\delta:\Lambda/\Z\alpha\rightarrow \C,
\beta\mapsto\delta(\beta)$ with the property that
$\delta(\beta)\in \deg \mathcal{M}(\beta)$ and
$\delta(\beta)+\delta(-\beta)=-2$. This function will be called the \textit{degree function}.

\begin{lemma}\label{lemma_71}  
Let $\beta\in \Lambda/\Z\alpha$. Then we have:

\centerline{ $\delta(n\beta)=n(\delta(\beta)+1)-1$, $\forall n\in\Z$}
\end{lemma}
\begin{proof} 
Note that $\delta(0)=-1$, so we may assume that $\beta \neq 0$. 
Set $\mathcal{M}= \oplus_{n\in\Z}\,\mathcal{M}_n$, where  
$\mathcal{M}_n=\mathcal{M}(n\beta)$ for all integer $n$. Then
$\mathcal{M} \subset \mathcal{L}$ is a weakly $\Z$-graded Lie algebra.

First assume that $\delta(\beta)\neq 0$ or $-2$.
Set $d(n)=n(\delta(\beta)+1)-1$.
By Lemma \ref{lemma_69}, the local Lie algebra 
$V(\beta)=\mathcal{M}(-\beta)\oplus\mathcal{L}(\alpha)\oplus
\mathcal{M}(\beta)$ admits 
$\widetilde{\mathcal{P}}_{\delta(\beta),u}^{loc}$ as a $W$-section,
for some $u\in\C/\Z$.
By Lemma \ref{lemma_68},  $\widetilde{\mathcal{P}}_{\delta(\beta),u}$ is minimal.
Thus by Lemma \ref{lemma_64}, $\widetilde{\mathcal{P}}_{\delta(\beta),u}$ is a 
$W$-section of $\mathcal{M}$. Hence  the $W$-module
$\widetilde\Omega^{d(n)}_{nu}$ is a subquotient
of $\mathcal{M}(n\beta)$.
So we have $d(\pm n)\in\deg \mathcal{M}(\pm n\beta)$.

Moreover, it is obvious that 
$d(n)+d(-n)=-2$. Therefore it follows from Lemma \ref{lemma_70} that  
$d(n)=\delta(n\beta)$, and the lemma is proved in this case.

Next assume that $\delta(\beta)=-2$.
By Lemma \ref{lemma_69}, the local Lie algebra 
$V(\beta)$ admits 
$\widetilde{\mathcal{P}}_{-2,u}^{loc}$ as a $W$-section,
for some $u\in\C/\Z$.
However $\widetilde{\mathcal{P}}_{-2,u}^{loc}$ is the local part
of weakly $\Z$-graded Lie algebra
$\widetilde{\mathcal{P}}^+(u)$, which is
minimal by Lemma \ref{lemma_68}. Hence $\widetilde{\mathcal{P}}^+(u)$ is a
$W$-section of $\mathcal{M}$. Hence the $W$-module
$\widetilde\Omega^{-n-1}_{nu}$ is a subquotient
of $\mathcal{M}(n\beta)$ for all $n\geq -1$. It follows
that $\deg \mathcal{M}(n\beta)$ is single valued for
all $n\geq 0$, and that  $\deg \mathcal{M}(n\beta)=-n-1$.
Thus we get
$\delta(n\beta)=-n-1$ for all $n\geq 0$
Since $\delta(n\beta)+\delta(-n\beta)=-2$, it follows that
$\delta(n\beta)=-n-1$ for all $n\in\Z$, and the
lemma is proved in this case.

The last case $\delta(\beta)=0$  is identical to
the previous one, because $\delta(-\beta)=-2$.
\end{proof}
\section{The degree function $\delta$ is affine}\label{sect_17}

As before, fix once for all, a primitive element $\alpha\in\Lambda$ such
that  $l(\alpha)\neq 0$. It follows from previous considerations that
the Lie subalgebra
$\mathcal{L}(\alpha)=\oplus_{n\in\Z}\,\mathcal{L}_{n\alpha}$
is isomorphic to $W$. For simplicity, any $\Z\alpha$-coset
${\beta}\in\Lambda/\Z\alpha$ will be called a \textit{coset}.
For a coset ${\beta}$, set
$\mathcal{M}({\beta})=
\oplus_{\gamma\in{\beta}}\,\mathcal{L}_{\gamma}$. We have

$$[\mathcal{M}(\beta_1),\mathcal{M}(\beta_2)]
\subset \mathcal{M}(\beta_1+\beta_2)$$

\noindent for any cosets $\beta_1,\,\beta_2$.

\begin{lemma}\label{lemma_72} Let $\beta,\,\gamma$
be cosets. Assume that
$\delta(\beta)\neq 0$. Then we have

\centerline{$[\mathcal{M}({\beta}), \mathcal{M}({\gamma})]\neq 0$.} 
\end{lemma}
\begin{proof}
Assume that $\delta(\beta)\neq 0$. 
Set $\mathcal{M}= \oplus_{n\in\Z}\,\mathcal{M}_n$, where  
$\mathcal{M}_n=\mathcal{M}(n\beta)$ for all integer $n$. Then
$\mathcal{M}$ is a weakly $\Z$-graded Lie algebra.

\textit{Step 1:}  We claim that:

\centerline{$[\mathcal{M}(\beta),\mathcal{M}(\beta)]\supset
[L_0,\mathcal{M}(2\beta)]$.}

By Lemma \ref{lemma_69}, the local Lie algebra 
$V(\beta):=\mathcal{M}(-\beta)\oplus\mathcal{L}(\alpha)\oplus
\mathcal{M}(\beta)$ admits 
$\widetilde{\mathcal{P}}_{\delta(\beta),u}^{loc}$ as a $W$-section,
for some $u\in\C/\Z$.

If $\delta(\beta)\neq -2$, the 
weakly $\Z$-graded Lie algebra $\widetilde{\mathcal{P}}_{\delta(\beta),u}$ is
minimal. Thus by Lemma \ref{lemma_64}, $\widetilde{\mathcal{P}}_{\delta(\beta),u}$ is a 
$W$-section of $\mathcal{M}$. Similarly 
the  weakly $\Z$-graded Lie algebra $\widetilde{\mathcal{P}}^+(u)$
is the minimal Lie algebra associated to the local
Lie algebra $\widetilde{\mathcal{P}}_{-2,u}^{loc}$. Thus
$\widetilde{\mathcal{P}}^+(u)$ is a $W$-section of $\mathcal{M}$ if
$\delta(\beta)= -2$

In both case, we have 
$[\widetilde{\Omega}^{\delta(\beta)}_u, \widetilde{\Omega}^{\delta(\beta)}_u]
=\widetilde{\Omega}^{2\delta(\beta)+1}_{2u}$. Since 
$\widetilde{\Omega}^{2\delta(\beta)+1}_{2u}$ is a $W$-subquotient of
$\mathcal{M}_2$, the $\C$-graded vector spaces
$[L_0,\widetilde{\Omega}^{2\delta(\beta)+1}_{2u}]$ and 
$[L_0,\mathcal{M}(2\beta)]$ are isomorphic, the claim follows.

\textit{Step 2:} Let $M\in \mathcal{S}(W)$ and let
$m\in M$ be a non-zero vector such that
$L_0.m=xm$ for some $x\neq 0$. Then we have

\centerline{ $L_1.m\neq 0$ or $L_2.m\neq 0$.}

This claim follows easily from Kaplansky-Santharoubane Theorem.

\textit{Step 3:} Fix $\beta_0\in\beta$ and $\gamma_0\in\gamma$
with $l(\beta_0)\neq 0$ and $l(\gamma_0)\neq 0$. We claim that

\centerline{$[L_{\beta_0},L_{\gamma_0}]\neq 0$ or
$[L_{2\beta_0},L_{\gamma_0}]\neq 0$.}

By Lemma \ref{lemma_62}, $\beta_0$ lies in $\Sigma$, and by Lemma \ref{lemma_46} the Lie subalgebra 
$\mathcal{L}(\beta_0)=\oplus_{n\in\Z}\,\mathcal{L}_{n\beta_0}$
is isomorphic to $W$. Thus the
$\mathcal{L}(\beta_0)$-module 

\centerline{$\mathcal{M}(\beta_0,\gamma_0) :=
\oplus_{n\in\Z}\,\mathcal{L}_{\gamma_0+n\beta_0}$,}

\noindent belongs to $\mathcal{S}(W)$. Since $l(\gamma_0)\neq 0$, it follows
from Step 2 that 
$[L_{\beta_0},L_{\gamma_0}]\neq 0$ or
$[L_{2\beta_0},L_{\gamma_0}]\neq 0$.

\textit{Final step:} Let $\beta_0$ and $\gamma_0$ as before. 
By definition,  $L_{\beta_0}$ belongs to $\mathcal{M}(\beta)$. By Step 1, 
$[\mathcal{M}(\beta),\mathcal{M}(\beta)]$ contains
$L_{2\beta_0}$. Thus by the previous claim,
we have 
$[\mathcal{M}(\beta),L_{\gamma_0}]\neq 0$
or $[[\mathcal{M}(\beta),\mathcal{M}(\beta)],L_{\gamma_0}]\neq 0$.
In both cases, we have
$[\mathcal{M}(\beta),L_{\gamma_0}]\neq 0$ and therefore
$[\mathcal{M}(\beta),\mathcal{M}(\gamma)]\neq 0$. 
\end{proof}

\begin{lemma}\label{lemma_73} Let $\beta,\,\gamma$ be cosets. We have:

\centerline {$\delta(\beta+\gamma)=\delta(\beta)+\delta(\gamma)+1$.}
\end{lemma}
\begin{proof}
\textit{Step 1:}
Let $\eta$ be any coset. By Lemma \ref{lemma_71}, we have

\centerline{$\delta(n\eta)=n(\delta(\eta)+1)-1$,}

\noindent  so we have

\centerline{$\delta(n\eta)\equiv -1$ for all $n$ or
$\vert \delta(n\eta)\vert\rightarrow\infty$ when $n\rightarrow\infty$.}

\noindent So there exists $k\neq0$ such that the set
$\{\delta(k\beta),\,\delta(-k\beta),
\,\delta(k\gamma),\,\delta(k(\beta+\gamma))\}$
contains neither $0$ nor $1$.

\textit{Step 2:} Let $k$ as before and let

$\mu^+: \mathcal{M}(k\beta)\times \mathcal{M}(k\gamma)
\rightarrow \mathcal{M}(k(\beta+\gamma))$, and

$\mu^-: \mathcal{M}(-k\beta)\times \mathcal{M}(k(\beta+\gamma))
\rightarrow \mathcal{M}(k\gamma)$ 

\noindent be the  $W$-equivariant bilinear maps
induced by the Lie bracket. By Lemma \ref{lemma_72},
$\mu^\pm$ are non-zero. We have
\begin{align*}
\deg\mu^+ +\deg\mu^-
=
&[\delta(k(\beta+\gamma))-\delta(k\beta)-\delta(k\gamma)]
+[\delta(k\gamma)-\delta(-k\beta)-\delta(k(\beta+\gamma))] \\
=
&-\delta(k\beta)-\delta(-k\beta).
\end{align*}
So it follows from  Lemma \ref{lemma_70} that

\centerline{ $\deg\mu^+ +\deg\mu^-=2$.}

Observe that, in the list of Lemma \ref{lemma_50},
the bilinear maps of degree
$-2$, $-1$ or $2$ involves at least one module of degree
$0$ or $1$. Thus it follows from the choice of
$k$ that the degrees of
$\mu^\pm$ are $0, 1$ or $3$. 
So the only solution of the previous equation is:

\centerline{
$\deg\mu^+ =\deg\mu^-=1$,}

\noindent and
therefore $\delta(k(\beta+\gamma))=\delta(k\beta)+\delta(k\gamma)+1$.
It follows from Lemma \ref{lemma_71} that:

$k(\delta(\beta+\gamma)+1)-1= 
[k(\delta(\beta)+1)-1]+[k(\delta(\gamma)+1)-1]+1$,

\noindent from which the relation
$\delta(\beta+\gamma)=\delta(\beta)+\delta(\gamma)+1$
follows. 
\end{proof}

\begin{lemma}\label{lemma_74} 
Let $M$, $N$ in $\mathcal{S}(W)$.
\begin{enumerate}
\item[(i)] If there is a non-zero $W$-morphism $\mu:M\rightarrow N$,
then $\deg M=\deg N$.
\item[(ii)] If there is a non-zero $W$-invariant bilinear map 
$\nu: M\times
N\rightarrow \C$, then $\deg M=1-\deg N$.
\end{enumerate}
\end{lemma}

\begin{proof}
If $\mu$ is an isomorphism,
the assertion is clear. Otherwise both $M$ and $N$ are reducible,
and  $\deg M=\deg N=\{0,1\}$. Thus Assertion (i) is proved. The bilinear
map $\nu$ gives rise to a non-zero morphism $\nu^\ast:M\rightarrow N'$,
thus we get $\deg M=\deg N'=1-\deg N$. 
\end{proof}

Recall that $\mathcal{L}(\alpha)$ is identified with $W$.

\begin{lemma}\label{lemma_75} For any coset $\beta$, the $W$-module 
$\mathcal{M}(\beta)$ is irreducible.
\end{lemma}
\begin{proof}
\textit{Step 1:} Let $x\in\C$,  $x\neq -1$, and
set $\Lambda_x=\delta^{-1}(x)$. We claim that there are no
finite sets $F\subset \Lambda$ such that:
 
\centerline{$\Lambda=F+\Lambda_x$.}

\noindent Indeed, it can be assumed that $\Lambda_x\neq\emptyset$.
By the previous lemma, $\delta$ is affine.
Since $\delta(0)=-1$ and $\delta^{-1}(x)\neq\emptyset$, the
function $\delta$ takes infinitely many values. However $\delta$ takes
only finitely many values on $F+\Lambda_x$ for any finite set $F$.
So the claim is proved.

\textit{Step 2:}  First prove that $\mathcal{M}(\beta)$ does
not contain a trivial $W$-submodule. Assume otherwise. There exists
$\mu\in \beta$ such that 
$L_\mu$ is a  non-zero $W$-invariant vector.

For any coset $\gamma$ such that
$[L_\mu, \mathcal{M}(\gamma)]\neq 0$, the operator
$\ad L_\mu$ provides a non-zero $W$-morphism
from $\mathcal{M}(\gamma)\rightarrow \mathcal{M}(\beta+\gamma)$. So
the previous lemma implies

\centerline{ $\deg \mathcal{M}(\gamma)=\deg \mathcal{M}(\beta+\gamma)$.}

By Lemma \ref{lemma_73}, we have
$\delta(\beta+\gamma)=\delta(\beta)+\delta(\gamma)+1$. Since $\mathcal{M}(\beta)$ is reducible, we get $\delta(\beta)=0$ or $1$. So we get

\centerline{$\delta(\beta+\gamma)=\delta(\gamma)+1$, or
$\delta(\beta+\gamma)=\delta(\gamma)+2$.}

The unique solution of these equations is

\centerline{$\delta(\beta)=0$, $\delta(\gamma)=0$, and
$\delta(\beta+\gamma)=1$.}

It follows that $\Omega(\mu)\subset \Lambda_1$, where
$\Omega(\mu)=\Supp [\mathcal{L}, L_\mu]$. By the first step, there are  no
finite set $F$ such that  
$\Lambda=F+\Omega(\mu)$, which contradicts Lemma \ref{lemma_43}.

\textit{Step 3:}  Now prove that $\mathcal{M}(\beta)$ does
not contain a codimension one $W$-submodule. Assume otherwise. Let
$H\subset \mathcal{M}(\beta)$ be a codimension one $W$ subspace
and let $\mu$ be the unique element of
$\beta\setminus\Supp H$. It follows that 
$L^\ast_\mu$ is a $W$-invariant vector of the graded dual $\mathcal{L}'$ of
$\mathcal{L}$.

For any coset $\gamma$  such that
$[\mathcal{M}(\gamma),\mathcal{M}(\beta-\gamma)]\not\subset H$, the Lie bracket
provides a non-zero bilinear map:
$\mathcal{M}(\gamma)\times\mathcal{M}(\beta-\gamma)\rightarrow
\mathcal{M}(\beta)/H\simeq \C$. The previous lemma implies

\centerline{$\deg \mathcal{M}(\gamma)=1-\deg \mathcal{M}(\beta-\gamma)$.}

\noindent By Lemma \ref{lemma_73}, we have
 $\delta(\beta)=\delta(\gamma)+\delta(\beta-\gamma)+1$.
Since $\mathcal{M}(\beta)$ is reducible, we get $\delta(\beta)=0$ or $1$.
So we get

\centerline{$\delta(\gamma)+\delta(\beta-\gamma)=0$, or
$\delta(\gamma)+\delta(\beta-\gamma)=-1$.}

The unique solution of these equations is

\centerline{$\delta(\beta)=1$, $\delta(\gamma)=0$, and
$\delta(\beta-\gamma)=0$.}

It follows that $\mathcal{M}(\gamma).L^\ast_\mu\neq 0$ only if
$\delta(\gamma)=0$. It follows that
$\Omega^\ast(\mu)\subset \Lambda_0-\mu$, where
 $\Omega^\ast(\mu)=\Supp \mathcal{L}. L^\ast_\mu$. 
By the first step, there are  no
finite set $F$ such that  
$\Lambda=F+\Omega^\ast(\mu)$, which contradicts Lemma \ref{lemma_63}.

\textit{Step 4:} It follows that the $W$-module $\mathcal{M}(\beta)$ 
contains neither a trivial submodule nor a codimension one submodule.
By Kaplansky-Santha\-rou\-ba\-ne Theorem, the
$W$-module  $\mathcal{M}(\beta)$  is irreducible. 
\end{proof}

\section{The quasi-two-cocycle $c$}\label{sect_18}

In this section, a certain  $W$-equivariant  map
$\phi:\mathcal{L}\rightarrow \mathcal{P}$ is defined. Indeed, $\phi$ is not a Lie 
algebra morphism, but we have

\centerline{$\phi([L_\lambda,L_\mu])=c(\lambda,\mu) 
\{\phi(L_\lambda),\phi(L_\mu)\}$}

\noindent  for some $c(\lambda,\mu)\in \C^\ast$. 
The main result of the section, namely  
Lemma \ref{lemma_79}, shows that $c$ satisfies a two-cocycle identity.
However its validity domain is only a (big) subset of
$\Lambda^3$. Since $c$ is not an ordinary two-cocycle, it is informally called a
``quasi-two-cocycle".

Let $\mathcal{P}$ be the Poisson algebra of symbols of twisted
pseudo-differential operators. Recall that

$$\mathcal{P}=\oplus_{(\delta,u)\in{Par}}\,\,\Omega_u^\delta$$

\noindent Set $A=\C[z,z^{-1}]$,
$A_u=z^u A$ for $u\in \C/\Z$. As before, it will be 
convenient to represent elements in $\Omega_u^\delta$ as 
$f\partial^{-\delta}$, where $f\in A_{u-\delta}$.
Relative to the Poisson bracket, we have
$\{\Omega_u^\delta,\Omega_v^\eta\}\subset \Omega_{u+v}^{\delta+\eta+1}$.

Let $(\delta,u)$, $(\eta,v)$ and
$(\tau,w)\in Par$.

\begin{lemma}\label{lemma_76} Let $a, b, c$ be three scalars.
Assume that 

\centerline{$a\{X,\{Y,Z\}\}+b \{Z,\{X,Y\}\}+c \{Y,\{Z,X\}\}=0$}

\noindent for any $(X,Y,Z)\in
\Omega_u^{-\delta}\times\Omega_v^{-\eta}\times\Omega_w^{-\tau}$

If none of the four couples
$(\eta,\tau)$, $(\delta,\eta+\tau-1)$, $(\eta,\delta)$,
$(\tau,\delta+\eta-1)$ is $(0,0)$, then we have

\centerline{$a=b$.}
\end{lemma}
\begin{proof} 
Assume that none of the four couples is $(0,0)$.
Define three maps 
$\mu_i:\Omega_u^{-\delta}\times\Omega_v^{-\eta}\times\Omega_w^{-\tau}
\rightarrow\Omega^{-\delta-\eta-\tau+2}_{u+v+w},$ 
$(X,Y,Z)\mapsto\mu_i(X,Y,Z)$ as follows:

$\mu_1(X,Y,Z)=\{X,\{Y,Z\}\}$

$\mu_2(X,Y,Z)=\{Z,\{X,Y\}\}$

$\mu_3(X,Y,Z)=\{Y,\{Z,X\}\}$

\noindent Set $X=f\partial^\delta$, 
$Y=g\partial^\eta$ and $Z=h\partial^\tau$, where
$f\in A_{u+\delta}$, $g\in A_{v+\eta}$ and $h\in A_{w+\tau}$. We have:
\begin{align*}
&\mu_1(X,Y,Z) \\
=
&\{f\partial^\delta,
\{ g\partial^\eta,h\partial^\tau\}\} \\
=
&\{f\partial^\delta,(\eta gh'-\tau
g'h)\partial^{\eta+\tau-1}\} \\
=
&[\delta\eta fgh''-\delta\tau fg''h +(\delta\eta-\delta\tau) fg'h' 
-(\eta+\tau-1)\eta f'gh' \\
&+(\eta+\tau-1)\tau
f'g'h]\partial^{\delta+\eta+\tau-2}.
\end{align*}
So $[a\mu_1+b\mu_2+c\mu_3](X,Y,Z)$ can be expressed as:

$(A fgh'' +B fg''h +C f''gh +D fg'h' +E f'gh' +E
f'g'h)\partial^{\delta+\eta+\tau-2}$, 

\noindent where the six coefficients are given by:

$A=(a-c)\delta\eta$

$B=(b-a)\tau\delta$

$C=(c-b) \eta\tau$

$D=[a(\eta-\tau)-b(\delta+\eta-1)+c(\tau+\delta-1)]\delta$

$E=[-a(\tau+\eta-1)+b(\delta+\eta-1)+c(\tau-\delta)]\eta$

$F=[a(\tau+\eta-1)+b(\tau-\eta)-c(\delta+\tau-1)]\tau$.

The equation  $a\mu_1+b\mu_2+c\mu_3=0$ implies that the six coefficients
$A, B,\dots ,F$ are all zero.  

If $\delta\tau\neq 0$, the equality  $a=b$ follows from $B=0$.
Assume otherwise. Since $\delta$ or $\tau$ is zero, we have
$\eta\neq 0$.

If $\tau\neq 0$ but $\delta=0$, it follows from $C=0$ that
$b=c$. Therefore the identity $E=0$ implies that
$(b-a)(\tau+\eta-1)=0$. Since $\delta=0$, we get
$\tau+\eta-1\neq 0$ and thus $a=b$.

The case $\tau=0$ but $\delta\neq 0$ is strictly similar.

In the case $\tau=\delta=0$, from $E=0$ we get
$(b-a)(\eta-1)=0$. Since $(\delta,\eta+\tau-1)=(0,\eta-1)$,
it follows that $\eta-1\neq 0$ and therefore $a=b$.
\end{proof}

Let $\delta\in\C$, let $u,v,w\in\C/\Z$ and
let $\theta: \Omega_u^{0}\times\Omega_v^{0}\rightarrow\Omega^1_{u+v}$
be any $W$-equivariant bilinear map. 

\begin{lemma}\label{lemma_77}Let $a,b,c$ be three scalars, with $a\neq 0$. Assume that

\centerline{$a\{X,\theta(Y,Z)\}+b\{Z,\{X,Y\}\}+c\{Y,\{Z,X\}\}=0$}

\noindent for any $(X,Y,Z)\in
\Omega_u^{-\delta}\times\Omega_v^{0}\times\Omega_w^{0}$.
Then we have $\theta=0$.
\end{lemma}
\begin{proof}
Define three maps 
$\mu_i:\Omega_u^{-\delta}\times\Omega_v^{0}\times\Omega_w^{0}
\rightarrow\Omega^{-\delta+2}_{u+v+w}, 
(X,Y,Z)\mapsto\mu_i(X,Y,Z)$ as follows:

$\mu_1(X,Y,Z)=\{X,\theta(Y,Z)\}$

$\mu_2(X,Y,Z)=\{Z,\{X,Y\}\}$

$\mu_3(X,Y,Z)=\{Y,\{Z,X\}\}$

As before, identify $\Omega^0_s$ with $A_s$.
By Lemma \ref{lemma_51}, there are two constants $A, B$
such that $\theta(g,h)=A h\d g+ B g\d h$. In term of symbols, we get
$\theta(g,h)=[A g'h+ B gh']\partial^{-1}$.

Set $X=f\partial^\delta$, 
$Y=g$ and $Z=h$, where
$f\in A_{u+\delta}$, $g\in A_v$ and $h\in A_w$. We have:
\begin{align*}
\mu_1(X,Y,Z)=
&\{f\partial^\delta,[A g'h+ B gh']\partial^{-1}\} \\
=
&[A\delta fg''h + B\delta fgh''+ \delta(A+B) fg'h' 
+Af'g'h +B f'gh']\partial^{\delta-2}. \\
\mu_2(X,Y,Z)=
&\{h,\{f\partial^\delta,g\}\} \\
=
&-\delta(\delta-1) fh'g'\partial^{\delta-2} \\
\mu_3(X,Y,Z)=
&\{g,\{h, f\partial^\delta\}\} \\
=
&\delta(\delta-1) fh'g'\partial^{\delta-2}.
\end{align*}
The coefficient of the monomial $f'g'h$ and the monomial $f'gh'$
in the expression $(a\mu_1+b\mu_2+c\mu_3)(X,Y,Z)$ are respectively $A$ and
$B$. Thus $A=B=0$ and so $\theta$ vanishes.
\end{proof}

Let $\mathcal{L}$ be a non-integrable  Lie algebra in the class $\mathcal{G}$.
As before, fix once for all a primitive element $\alpha\in\Lambda$ such
that  $l(\alpha)\neq 0$. One can normalize $L_0$ in a way that
$l(\alpha)=1$.

By Lemmas \ref{lemma_46} and \ref{lemma_62}, the Lie subalgebra
$\mathcal{L}(\alpha)=\oplus_{n\in\Z}\,\mathcal{L}_{n\alpha}$
is isomorphic to $W$. 
For any $\Z\alpha$-coset $\beta$, set
$\mathcal{M}(\beta)= \oplus_{\gamma\in\beta}\,\mathcal{L}_{\gamma}$. 
For such $\beta$, denote by $s(\beta)$ the 
spectrum of $L_0$ on $\mathcal{M}(\beta)$. It is clear that
$s(\beta)$ is exactly a $\Z$-coset, therefore
$s(\beta)$ is a well defined element of $\C/\Z$.

By Lemma \ref{lemma_75}, the $W$-module $\mathcal{M}(\beta)$ is
irreducible, and by Kaplansky-Santharoubane Theorem, there is
an isomorphism of
$W$-modules 

\centerline{$\phi_\beta:\mathcal{M}(\beta)\rightarrow 
\Omega^{\delta(\beta)}_{s(\beta)}$.}

\noindent  Fix once for all
the isomorphism $\phi_\beta$ for all $\beta\in \Lambda/\Z\alpha$.
Let $\phi:\mathcal{L}\rightarrow \mathcal{P}$ be the map whose
restriction to each $\mathcal{M}(\beta)$ is
$\phi_\beta$. Thus $\phi$ is a morphism of $W$-modules. As it will be seen
later on, some modification of $\phi$ will be a morphism of Lie algebras.

Let $\beta$, $\gamma$ be two $\Z\alpha$-cosets. 
It is clear that 
$[\mathcal{M}(\beta),\mathcal{M}(\gamma)]\subset \mathcal{M}(\beta+\gamma)$,
so the Lie bracket provides a morphism of $W$-modules

\centerline{$[~,~]:\mathcal{M}(\beta)\times \mathcal{M}(\gamma)
\rightarrow \mathcal{M}(\beta+\gamma)$. }

 \noindent By Lemma \ref{lemma_73}, we have
$\delta(\beta+\gamma)=\delta(\beta)+\delta(\gamma)+1$.
Therefore the bracket of symbols provides another morphism of
$W$-modules:

\centerline{$\{~,~\}: \Omega^{\delta(\beta)}_{s(\beta)}\times
\Omega^{\delta(\gamma)}_{s(\gamma)} 
\rightarrow\Omega^{\delta(\beta+\gamma)}_{s(\beta+\gamma)}$.}

\noindent Thus  we get a diagram
of $W$-modules :
\[ \UseTips
\newdir{ >}{!/-5pt/\dir{>}}
\xymatrix @=1pc @*[r]
{
    \mathcal{M}(\beta)\times \mathcal{M}(\gamma) \ar[rrr]^{\phantom{abcdefg}[~,~]} 
    \ar[dd]^{\phi_\beta\times \phi_\gamma}
    &&& \mathcal{M}(\beta+\gamma) \ar[dd]^{\phi_{\beta+\gamma}} \\
    &&& \\
    \Omega_{s(\beta)}^{\delta(\beta)} \times \Omega_{s(\gamma)}^{\delta(\gamma)} \ar[rrr]^{\phantom{abcdef}\{~,~\}} &&& 
    \Omega_{s(\beta+\gamma)}^{\delta(\beta+\gamma)} \\
} \]
\noindent This diagram is almost commutative:

\begin{lemma}\label{lemma_78} There exists
 a function $c:\Lambda/\Z\alpha\times 
\Lambda/\Z\alpha\rightarrow \C^\ast$
such that 

\centerline{$\phi([X,Y])=
c(\beta,\gamma)\{\phi(X),\phi(Y)\}$}

\noindent for any $\beta,\,\gamma\in \Lambda/\Z\alpha$ and any
$X,\,Y\in \mathcal{M}(\beta)\times \mathcal{M}(\gamma)$.
Moreover 

\centerline{
$c(\beta,\gamma)=c(\gamma,\beta)$}

\noindent if $(\delta(\beta),\delta(\gamma))\neq (0,0)$.
\end{lemma}
\begin{proof}
Let $\beta,\gamma\in \Lambda/\Z\alpha$.

First consider the case $(\delta(\beta),\delta(\gamma))\neq (0,0)$.
By Lemma \ref{lemma_51}, any $W$-equivariant bilinear map
$\Omega^{\delta(\beta)}_{s(\beta)}\times
\Omega^{\delta(\gamma)}_{s(\gamma)}\rightarrow
\Omega^{\delta(\beta+\gamma)}_{s(\beta+\gamma)}$ is proportional to the
Poisson bracket of symbols. Thus there is a constant
$c(\beta,\gamma)\in \C$ such that
$\phi([X,Y])=
c(\beta,\gamma)\{\phi(X),\phi(Y)\}$
for any 
$X,\,Y\in \mathcal{M}(\beta)\times \mathcal{M}(\gamma)$.
Moreover, by Lemma \ref{lemma_72}, we have
$[\mathcal{M}(\beta),\mathcal{M}(\gamma)]\neq 0$. Therefore 
$c(\beta,\gamma)\neq 0$. Since the Poisson bracket and the Lie bracket
are skew symmetric, we also have $c(\beta,\gamma)=c(\gamma,\beta)$.

Next consider the case $(\delta(\beta),\delta(\gamma))=
(0,0)$. We claim that 

\centerline{\it $[\mathcal{M}(\beta),\mathcal{M}(\gamma)]=0$.}

 In order to prove the claim, we may assume that
$\delta$ takes the value $0$. Define
the $W$-equivariant bilinear map
$\theta: \Omega^{\delta(\beta)}_{s(\beta)}\times
\Omega^{\delta(\gamma)}_{s(\gamma)}
\rightarrow\Omega^{\delta(\beta+\gamma)}_{s(\beta+\gamma)}$
by the requirement:
$\phi_{\beta+\gamma}([X,Y])=
\theta(\phi_\beta(X),\phi_\gamma(Y))$.

Since $\delta(0)=-1$, $\delta$ takes the value $0$ and
$\delta$ is affine, there is some $\xi\in \Lambda/\Z\alpha$ such that
$\delta(\xi)\not\in\{0,-1\}$.
Let $X,\,Y,\,Z\in \mathcal{M}(\beta)\times \mathcal{M}(\gamma)\times
\mathcal{M}(\xi)$. Since $\delta(\xi)+1\neq 0$, we get

$\phi([X,[Y,Z]])=c(\beta,\gamma+\xi)\{\phi(X),\phi([Y,Z])\}$.

\hskip2.09cm$=c(\beta,\gamma+\xi)c(\gamma,\xi)
\{\phi(X),\{\phi(Y),\phi(Z)\}\}$

Similarly, we have

$\phi([Y,[Z,X]])=c(\gamma,\beta+\xi)c(\beta,\xi)
\{\phi(Y),\{\phi(Z),\phi(X)\}\}$, and

$\phi([Z,[X,Y]])=c(\xi,\alpha+\beta)\{\phi(Z),\theta(X,Y)\}$.

\noindent Thus the Jacobi identity  in $\mathcal{L}$: 
$[X,[Y,Z]]+[Y,[Z,X]]+[Z,[X,Y]]=0$ implies that

$c(\beta,\gamma+\xi)c(\gamma,\xi)\{X,\{Y,Z\}\}
+c(\gamma,\beta+\xi)c(\beta,\xi)\{Y,\{Z,X\}\}$

\hskip6cm$+c(\xi,\alpha+\beta)\{Z,\theta(X,Y)\}=0$,

\noindent for any $(X, Y,Z)\in 
\Omega^{0}_{s(\beta)}\times
\Omega^{0}_{s(\gamma)}\times\Omega^{\delta(\xi)}_{s(\xi)}$. Since
$c(\xi,\alpha+\beta)$ is not zero, it follows from Lemma \ref{lemma_77}
that $\theta$ vanishes. Thus the claim is proved.

Since $[\mathcal{M}(\beta),\mathcal{M}(\gamma)]=0$ and
$\{\Omega^{0}_{s(\beta)}, \Omega^{0}_{s(\gamma)}\}=0$,
any value $c(\beta,\gamma)\in\C^\ast$ is suitable.
\end{proof}

\begin{lemma}\label{lemma_79} Let $\beta$, $\gamma$ and $\eta$ be three 
$\Z\alpha$-cosets. If none of the four couples
$(\delta(\beta),\delta(\gamma))$, 
$(\delta(\beta+\gamma),\delta(\eta))$, 
$(\delta(\gamma),\delta(\eta))$ and
$(\delta(\beta),\delta(\gamma+\eta))$ is $(0,0)$, then we have:

\centerline{$c(\beta,\gamma)c(\beta+\gamma,\eta)=
c(\beta,\gamma+\eta)c(\gamma,\eta)$.}
\end{lemma}
\begin{proof}
As in the previous proof, we have:

\centerline{$\phi([X,[Y,Z]])=c(\beta,\gamma+\eta)c(\gamma,\eta)
\{\phi(X),\{\phi(Y),\phi(Z)\}\}$.} 

\noindent for any $(X,Y,Z)\in
\mathcal{M}(\beta)\times \mathcal{M}(\gamma)\times \mathcal{M}(\eta)$. It follows
from the Jacobi identity in $\mathcal{L}$ that

$c(\beta,\gamma+\eta)c(\gamma,\eta)\{X,\{Y,Z\}\}
+ c(\eta,\beta+\gamma)c(\beta,\gamma)\{Z,\{X,Y\}\}$

\hskip5.5cm$+c(\gamma,\eta+\beta)c(\gamma,\eta)\{Y,\{Z,X\}\}=0$,

\noindent for any 
$(X,Y,Z)\in
\Omega^{\delta(\beta)}_{s(\beta)}\times
\Omega^{\delta(\gamma)}_{s(\gamma)}\times
\Omega^{\delta(\eta)}_{s(\eta)}$. Thus the lemma follows
from Lemma \ref{lemma_76}. 
\end{proof}

\noindent{\textbf{Remark.}}~
\textit{The equation satisfied by $c$ is exactly 
the equation of a two-cocycle of $\Lambda/\Z\alpha$ with values in 
$\C^\ast$, except that its validity domain is a subset of 
$(\Lambda/\Z\alpha)^3$. So
$c$  is a ``quasi-two-cocycle".}

\section{Proof of Theorem \ref{theorem_3}}\label{sect_19}

As before, let $\mathcal{L}\in\mathcal{G}$ be a 
non-integrable Lie algebra.
Fix once for all a primitive element 
$\alpha\in\Lambda$ such that  $l(\alpha)\neq 0$ and recall
that $\mathcal{L}(\alpha)\simeq W$. As before, we normalize $L_0$ in such a way that $l(\alpha)=1$.
In the previous section, a $W$-equivariant map
$\phi:\mathcal{L}\rightarrow \mathcal{P}$ has been defined. The map
$\phi$ is not a Lie algebra morphism, but the defect is
accounted by a map $c:\Lambda/\Z\alpha\times
\Lambda/\Z\alpha\rightarrow \C^\ast$.
 
In this section, it is
proved that  $c$ is indeed a ``quasi-boundary". 
This allows us to modify $\phi$ to get 
an algebra morphism  $\psi:\mathcal{L}\rightarrow \mathcal{P}$, from which 
Theorem \ref{theorem_3} is deduced.

\begin{lemma}\label{lemma_80} Let $M$ be a lattice and
let $R=\oplus_{m\in M}\,R_m$ be a commutative  associative
$M$-graded algebra satisfying the following conditions:
\begin{enumerate}
\item[(i)] $\dim R_m\leq 1$ for all $m\in M$,
\item[(ii)] $X.Y\neq 0$ for any two non-zero homogeneous elements $X$, $Y$ of
$R$.
\end{enumerate}
Then there exists an algebra morphism $\chi:R\rightarrow \C$ with 
$\chi(X)\neq 0$ for any non-zero homogeneous element $X\in R$.
\end{lemma}
\begin{proof}
Let $S$ be the set of non-zero homogeneous elements of $R$.
By hypothesis, $S$ is a multiplicative subset. 
Since $S/\C^\ast\simeq \Supp R$ is countable, the algebra $R_S$ has
countable dimension, therefore any maximal ideal of $R_S$ provides
an algebra morphism $\chi: R_S\rightarrow \C$. So its restriction to
$R$ is an algebra morphism $\chi:R\rightarrow \C$ with 
$\chi(X)\neq 0$ for any $X\in S$. 
\end{proof}

Let $M$ be a lattice, and let $d:M\rightarrow \C$ be an
additive map. Let
$c:M\times M\rightarrow\C^\ast,\,(l,m)\mapsto c(l,m)$ be a function.

\begin{lemma}\label{lemma_81} 
Assume the following hypotheses for all $l,m,n\in M$:
\begin{enumerate}
\item[(i)]  $c(l,m)=c(m,l)$ whenever $(d(l),d(m))\neq (-1,-1)$.
\item[(ii)] $c(l,m)c(l+m,n)=c(l,m+n)c(m,n)$ whenever none of the four couples
$(d(l),d(m))$, $(d(l+m),d(n))$, $(d(l),d(m+n))$ and 
$(d(m),d(n))$ is  $(-1,-1)$.
\end{enumerate}
Then there exists a function $b:M\rightarrow\C^\ast$ such that

\centerline{$c(l,m)=b(l)b(m)/b(l+m)$}

\noindent  for any couple $(l,m)$ with 
$(d(l),d(m))\neq (-1,-1)$.
\end{lemma}

\begin{proof}
Define an algebra structure on $\C[M]$ by the formula:

\centerline{$e^l\ast e^m=c(l,m)e^{l+m}$.}

Set $N=\{m\in M\vert\,\Re d(m)\geq 0\}$, where the notation $\Re z$ means
the real part of the complex number $z$ and
let $\C[N]$ be the subalgebra of $\C[M]$ with basis 
$(e^n)_{n\in N}$. Since $d(n)\neq -1$ for any
$n\in N$, it follows
from Identities (i) and (ii) that  $\C[N]$ is a commutative and
associative algebra. By the previous lemma, there is an algebra
morphism $\chi:\C[N]\rightarrow \C$ with 
$\chi(X)\neq 0$ for any  element $X\in S$, where 
$S$ denotes the set of non-zero homogeneous elements in $\C[N]$.

By restriction, the product $\ast$ defines a bilinear map
$\beta:\C[N]\times \C[M]\rightarrow \C[M]$. 
There are no couples $(n,m)\in N\times M$ with
$(d(n),d(m))=(-1,-1)$. Therefore
$\beta$ is a structure of $\C[N]$-module on $\C[M]$.
Indeed for $n_1,n_2\in N$ and $m\in M$, we have
$e^{n_1}\ast(e^{n_2}\ast e^m)=(e^{n_1}\ast e^{n_2})\ast e^m$.

As each
$X\in S$ acts bijectively on $\C[M]$, 
$\C[M]$ contains $\C[N]_S$. Since $M=N-N$, $\C[M]$ is isomorphic to $\C[N]_S$ as a $\C[N]$-module. Therefore
 $\chi$ extends to a morphism of $\C[N]$-modules
$\chi:\C[M]\rightarrow\C$. So it satisfies

\centerline{$\chi(e^n\ast e^m)=\chi(e^n) \chi(e^m)$}

\noindent for any couple 
$(n,m)\in N\times M$.

Next, prove the stronger assertion that

\centerline{$\chi(e^m\ast e^l)=\chi(e^m) \chi(e^l)$}

\noindent
for any couple $(m,l)$  with 
$(d(m),d(l))\neq (-1,-1)$. Choose  any $n\in N$ with
$\R d(m+n)\geq 0$. 
Since $n\in N$, it follows from the previous
identity that: 

\centerline{$\chi(e^n\ast (e^m\ast e^l))=\chi(e^n)\chi(e^m\ast e^l)$.}

\noindent Similarly, it follows from the fact that
$n\in N$ and $n+m\in N$ that

\centerline{$\chi((e^n\ast e^m)\ast e^l)=\chi(e^n\ast e^m)\chi(e^l)$
$=\chi(e^n)\chi(e^m)\chi(e^l)$}

\noindent By hypothesis
$(d(m),d(l))\neq(-1,-1)$, and each couple $(d(n),d(m+l))$, $(d(n),d(m))$
and $(d(n+m),d(l))$ contains a scalar with non-negative real part.
So none of these couples is $(-1,-1)$, and therefore we have

\centerline{$e^n\ast (e^m\ast e^l)=(e^n\ast e^m)\ast e^l$}

\noindent It follows that
$\chi(e^m\ast e^l)=\chi(e^m)\chi(e^l)$ and therefore we have

\centerline{$c(m,l)\chi(e^{m+l})=\chi(e^m)\chi(e^l)$,}

\noindent for any  $(m,l)$ with  $(d(m),d(l))\neq (-1,-1)$. Thus
the function $b(m)=\chi(e^m)$ satisfies the required identity. 
\end{proof}

\noindent{\textbf{Remark.}}~
\textit{
It follows that a symmetric two-cocycle
of $M$ with value in $\C^\ast$ is a boundary. This corollary is
obvious. Indeed a symmetric two-cocycle gives rise to
a central extension}

\centerline
{$1\rightarrow \C^\ast\rightarrow \widehat{M}\rightarrow
M\rightarrow 0$}

\noindent \textit{where $\widehat{M}$ is an abelian group. It  is split because $M$ is
free in the category of abelian groups. }

Let  $b:\Lambda/\Z\alpha\rightarrow \C^\ast$ be a function. 
Recall that 
$\mathcal{L}=\oplus_{\beta\in\Lambda/\Z\alpha}\,
\mathcal{M}(\beta)$. Define a new morphism of $W$-modules 
$\psi:\mathcal{L} \rightarrow \mathcal{P}$ by the formula:

 \centerline{$\psi(X)=b(\beta)\phi(X)$}

\noindent for any $X\in \mathcal{M}(\beta)$ and any 
$\beta\in \Lambda/\Z\alpha$.

\begin{lemma}\label{lemma_82} There exists a function
$b:\Lambda/\Z\alpha\rightarrow \C^\ast$ such that 
$\psi:\mathcal{L}\rightarrow \mathcal{P}$ is a morphism of Lie algebras.
\end{lemma}

\begin{proof}
Set $M=\Lambda/\Z\alpha$ and for
 $\beta\in M$, set
$d(\beta)=-1-\delta(\beta)$. It follows from Lemma \ref{lemma_73} that the map
$d:M\rightarrow \C$ is additive. By Lemmas \ref{lemma_78} and \ref{lemma_79}, 
the quasi-two-cocycle $c$ and the additive map $d$ satisfies the
hypothesis of the previous lemma. Therefore there exists a function 
$b:\Lambda/\Z\alpha\rightarrow \C^\ast$ such that

\centerline{$c(\beta,\gamma)=b(\beta)b(\gamma)/b(\beta+\gamma)$}

\noindent for any couple $(\beta,\gamma)$ with
$(\delta(\beta),\delta(\gamma))\neq (0,0)$.

Choose such a function $b$. We claim that
$\psi([X,Y])=\{\psi(X),\psi(Y)\}$ for any $X\in \mathcal{M}(\beta)$,
$Y\in\mathcal{M}(\gamma)$ and any $\beta$, $\gamma$ be in 
$\Lambda/\Z\alpha$.

First assume that $(\delta(\beta),\delta(\gamma))\neq (0,0)$.
We have
\begin{align*}
\psi([X,Y])=
&b(\beta+\gamma)\,\phi([X,Y]) \\
=
&b(\beta+\gamma)c(\beta,\gamma)\,\{\phi(X),\phi(Y)\} \\
=
&b(\beta) b(\gamma)\,\{\phi(X),\phi(Y)\}.
\end{align*}
Thus $\psi([X,Y])=\{\psi(X),\psi(Y)\}$.

Consider now the case $(\delta(\beta),\delta(\gamma))= (0,0)$.
Since $\{\Omega^0_u,\Omega^0_s\}=0$, $\forall u,\, s\in \C/\Z$
it follows that $\psi([X,Y])=\{\psi(X),\psi(Y)\}$
because both sides of the identity are zero.

Therefore $\psi$ is an algebra morphism. 
\end{proof}

Define the map $\pi:\Lambda\rightarrow \C^2$ by the
formula:

\centerline
{$\pi(\lambda)=(l(\lambda)-\delta(\lambda)-1,-1-\delta(\lambda))$.}

\noindent it follows from Lemma \ref{lemma_73} that $\pi$ is additive.
Since $\mathcal{P}$ is $\C^2$-graded, one can define
the $\Lambda$-graded Lie algebra $\pi^\ast\mathcal{P}$.
When $\pi$ is one-to-one, $\pi^\ast\mathcal{P}$ is the Lie algebra
$W_\pi$ defined in the introduction. In general, the notation
$\pi^\ast$ has been defined in Section \ref{sect_1.6} 
and $W_\pi$ in  Section \ref{sect_12.7}. 

\begin{lemma}\label{lemma_83} We have $\mathcal{L}\simeq\pi^\ast\mathcal{P}$.
\end{lemma}
\begin{proof}
Let $\lambda\in\Lambda$. By defintion,
$\psi(L_\lambda)=f\partial^{-\delta(\lambda)}$, for some twisted function
$f$. Moreover $[L_0,L_\lambda]=l(\lambda)L_\lambda$,
therefore we have 
$z\frac{d}{dz} f=[l(\lambda)-\delta(\lambda)]f$, therefore $f$ is
proportional to $z^{l(\lambda)-\delta(\lambda)}$. It follows
that $\psi$ maps isomorphically
$\mathcal{L}_\lambda$ to $\mathcal{P}_{\pi(\lambda)}$. By Lemma \ref{lemma_1}, 
$\mathcal{L}$ is precisely $\pi^\ast\mathcal{P}$. 
\end{proof}

Recall that the condition $\mathcal{C}$ is:

\centerline{$\Im \pi\not\subset \C\rho$
and $2\rho\notin\Im \pi$.}

\begin{thm}\label{theorem_3} Let $\Lambda$ be a lattice.
\begin{enumerate}
\item[(i)] If  $\mathcal{L}\in \mathcal{G}$ is a primitive
non-integrable  Lie algebra, then there
is an injective additive map
$\pi:\Lambda\rightarrow \C^2$
satisfying condition $\mathcal{C}$ such that
$\mathcal{L}\simeq W_\pi$.
\item[(ii)] Conversely, if $\pi:\Lambda\rightarrow \C^2$
is injective and satisfies condition $\mathcal{C}$,
then the Lie algebra $W_\pi$ is simple
(and, in particular, it is primitive).
\end{enumerate}
\end{thm}

\begin{proof}
By the previous lemma, 
$\mathcal{L}\simeq\pi^\ast\mathcal{P}=W_\pi$. Set $M=\Ker \pi$. There is a
sublattice
$\Lambda_1$ such that
$\Lambda=M\oplus\Lambda_1$. It follows that
$W_\pi\simeq \C[M]\otimes W_{\pi_1}$, where
$\pi_1:\Lambda_1\rightarrow \C$ is the restriction 
of $\pi$ to $\Lambda_1$. Since $\mathcal{L}$ is primitive, it follows that 
$M=0$, hence $\pi$ is injective.

If  $\pi(\Lambda)\subset \C\rho$, then
$W_\pi$ is abelian. If $2\rho\in \pi(\Lambda)$, then
$E_{-2\rho}$ belongs to $W_\pi$ and therefore
$W_\pi\neq [W_\pi,W_\pi]$. Since $\mathcal{L}$ is simple graded,
$\pi$ satisfies the condition $\mathcal{C}$.

The converse follows from Lemma \ref{lemma_49}. 
\end{proof}

\noindent{\textbf{Remark}.}
\textit{
Set $\omega=(1,0)$, so that $E_\omega$ is the symbol of
$z^2\partial$. It follows from the proof that, for any 
primitive vector $\alpha\in\Lambda$ with $l(\alpha)\neq 0$, there
exists a unique $\pi_{\alpha}:\Lambda\rightarrow \C^2$ such that
$\pi_{\alpha}(\alpha)=\omega$ and $\mathcal{L}\simeq W_{\pi_\alpha}$.
Therefore, there are many injective additive maps $\pi$ such that
$\mathcal{L}\simeq W_\pi$ and some of them do not contain
$\omega$  in their image.}

\bigskip
{\it Authors addresses:}

Universit\'e Claude Bernard Lyon 1,

Institut Camille Jordan, UMR 5028 du CNRS

43, bd du 11 novembre 1918

69622 Villeurbanne Cedex

FRANCE

{\it Email addresses:}

iohara@math.univ-lyon1.fr

mathieu@math.univ-lyon1.fr

\end{document}